%% file: Arxiv_version_new.tex
\documentclass[10pt]{article}
\usepackage{kpfonts}
\usepackage[T1]{fontenc}
\PassOptionsToPackage{hypertexnames=false}{hyperref}
\input{_macros}
\definecolor{OliveGreen}{rgb}{0,0.6,0}
\usepackage{natbib}
\usepackage{custom}
\usepackage[normalem]{ulem}
\usepackage[ruled,vlined]{algorithm2e} 
\allowdisplaybreaks{}
\makeatletter
\renewcommand{\paragraph}{%
  \@startsection{paragraph}{4}%
  {\z@}{1.25ex \@plus 1ex \@minus .2ex}{-1em}%
  {\normalfont\normalsize\bfseries}%
}
\makeatother

\renewcommand{\C}{\normalfont{\textsf{C}}}
\renewcommand{\c}{\normalfont{\textsf{c}}}

\DeclareTextFontCommand{\textsc}{\upshape\scshape}
\usepackage{xurl}
\usepackage{adjustbox,colortbl,booktabs,makecell,hhline,placeins}
\usepackage{tikz}
\usetikzlibrary{arrows.meta,calc}
\usepackage{pgfplots}
\pgfplotsset{compat=1.18}
\definecolor{rateblue}{RGB}{232,244,252}
\definecolor{rateyellow}{RGB}{255,249,220}
\definecolor{fzero}{RGB}{35,95,170}
\definecolor{fone}{RGB}{200,65,45}
\renewcommand{\H}{\normalfont{\textsf{H}}}
\newcommand{\A}{\normalfont{\textsf{A}}}
\newcommand{\loc}{\normalfont{\textsf{loc}}}
\newcommand{\FunctionPanel}[2]{%
\begin{axis}[
  width=5.6cm,height=5.4cm,
  xmin=-2,xmax=3,ymin=-#2,ymax=#2,
  axis lines=box,axis line style={gray!55},
  xtick={-1,0,1,2},ytick={0},
  tick label style={font=\small},
  xlabel={$t=(x-\theta)/d$},
  ylabel={$f_j(x)/(\c_0d^\gamma)$},
  title={$\gamma=#1$},
  title style={font=\large},
  legend style={draw=none,fill=white,font=\small,
    at={(0.04,0.96)},anchor=north west},
  samples=161,no markers,
  every axis plot/.append style={line width=1.2pt},
]
\draw[gray!40,thin] (axis cs:-2,0)--(axis cs:3,0);
\draw[gray!35,densely dotted,thin]
  (axis cs:-1,-#2)--(axis cs:-1,#2)
  (axis cs:2,-#2)--(axis cs:2,#2);
\addplot[fzero,domain=-2:-1] {-(1-x)^(#1)};
\addlegendentry{$f_0$}
\addplot[fzero,forget plot,domain=-1:0] {-abs(2*x)^(#1)};
\addplot[fzero,forget plot,domain=0:3] {x^(#1)};
\addplot[fone,dashed,domain=-2:1] {-(1-x)^(#1)};
\addlegendentry{$f_1$}
\addplot[fone,dashed,forget plot,domain=1:2] {(2*(x-1))^(#1)};
\addplot[fone,dashed,forget plot,domain=2:3] {x^(#1)};
\fill[fzero] (axis cs:0,0) circle[radius=1.7pt];
\fill[fone] (axis cs:1,0) circle[radius=1.7pt];
\end{axis}%
}

\title{The Root Finding Problem Revisited: Beyond the
Robbins-Monro procedure}

\begin{document}
\renewcommand{\thefootnote}{\fnsymbol{footnote}}
\setcounter{footnote}{2}

\author{%
  Yue Yu\textcolor{magenta}{\thanks{Department of Statistics, University of Michigan; Email: \texttt{\{yueuy,moulib,yritov\}@umich.edu}}}
   \and
   Moulinath Banerjee\textcolor{magenta}{\footnotemark[3]}
   \and
   Ya'acov Ritov\textcolor{magenta}{\footnotemark[3]}
}
\maketitle
\begin{abstract}

We study the sequential root-finding problem with repeated noisy measurements. We propose sequential-test bisection (\textsc{STB}), which uses a nonparametric sequential sign test to determine when to move to a new design point. Compared with variants of stochastic approximation, \textsc{STB} offers broader adaptivity and faster convergence. When the regression function satisfies $|f(x)|\geq \c_0|x-\theta|^\gamma$ in a neighborhood of $\theta$ for constants $\c_0,\gamma>0$ where $\theta$ is the root, i.e.  $f(\theta) = 0$, we establish non-asymptotic guarantees showing that \textsc{STB}, without using the knowledge of $\gamma$, achieves the minimax adaptive optimal convergence rate $\widetilde{O}(n^{-1/(2\gamma)})$, outperforming Robbins--Monro and related stochastic-approximation methods in convergence rate. Under sub-Gaussian noise, \textsc{STB} provides finite-sample confidence sequences for the root $\theta$ without estimating a derivative. In the positive-slope model, we augment \textsc{STB} with {randomized weighted-section and a line fitted to all observations}, leading to asymptotic normality and consistent variance estimation. We formalize the efficiency by a convolution theorem of H\'ajek-Le Cam type.
Furthermore, when the regression function is discontinuous at the root, we show that Robbins--Monro converges at rate $1/n$, whereas \textsc{STB} converges super-polynomially. We extend this construction to multivariate nonparametric level-set estimation by sampling along rays. As a technical byproduct, we establish a new central limit theorem under random stopping, with proof techniques that may be of independent interest.
\end{abstract}

\addcontentsline{toc}{section}{Appendix}
\setcounter{tocdepth}{3}
\begingroup
\makeatletter
\def\nocontentsline{%
  \let\@@addcontentsline\addcontentsline
  \def\addcontentsline##1##2##3{\let\addcontentsline\@@addcontentsline}%
}
\makeatother
\tableofcontents
\endgroup

\section{Introduction}
\label{section:intro}

The seminal paper \citep{robbins_stochastic_1951} considers the problem of using successive approximations $X_n$ to find the unique root $\theta$ of a regression function $f$ in the
regression model 
\begin{align}
\label{eq:DGP}
    Y_n = f(X_n) + \varepsilon_n,\quad n\in \bbN^+,
\end{align}
where $\varepsilon_n$ denotes unobservable random errors. We consider a target regression function \(f\) and a point \(\theta\in\mathbb R\) satisfying the zero-crossing condition
\begin{align}
\label{eq:zero-crossing}
(x - \theta)\,f(x) > 0
\quad \text{for all} \quad x \neq \theta.
\end{align}
We use ``root'' for the unique negative-to-positive crossing, including a jump discontinuity. In the continuous case $f(\theta)=0$. At a jump, $f(\theta)$ may be any specified finite value. The sign condition~\eqref{eq:zero-crossing} is imposed only for $x\neq\theta$.

Starting from an initial design point $X_1$, one observes $Y_1$, and then chooses $X_{n+1}$ based on the history $\{(X_i,Y_i)\}_{1\leq i \leq n}$. To average out the errors $\varepsilon_n$, \cite{robbins_stochastic_1951} proposed stochastic approximation (SA) with the update
\begin{align}
\tag{$\text{RM}$}\label{eq:RM-SA}
 X_{n+1}=X_n-a_nY_n,\qquad n\geq1,
\end{align}
where $a_n>0$, $\sum_n a_n=\infty$ and $\sum_n a_n^2<\infty$. For independent centered errors with uniformly bounded second moments, the noise series $\sum_n a_n\varepsilon_n$ converges in $L^2$ and almost surely (see~\cite{blum_approximation_1954,kushner_stochastic_2003}). A common choice is $a_n=a/n$ with $a>0$. The bottleneck for classical Robbins-Monro procedure is that, even in the positive-slope model $f(x) =\beta(x-\theta) + o(|x-\theta|)$, attaining the parametric rate requires the step-size constant $a$ to be sufficiently large so that $a\beta>1/2$. Otherwise, convergence can be arbitrarily slow (see the first row of Table~\ref{table:rate} and the accompanying discussion).

To tackle this issue, adaptive stochastic approximation (ASA) replaces the unknown positive slope $\beta=f'(\theta)$ by an estimate in the step-size. In the schemes of~\cite{lai_consistency_1981}, the root estimate is updated by $X_{n+1}=X_n-Y_n/(n\widehat\beta_n)$, where $\widehat\beta_n>0$ is a truncated least-squares slope estimate. Lai and Robbins established its consistency and $\sqrt{n}$-asymptotic normality. Another famous variant of SA proposed by Polyka and Juditsky, later on refered as PJ averaging, uses $\overline X_n=n^{-1}\sum_{i=1}^nX_i$ from~\eqref{eq:RM-SA}, with $a_n=a n^{-s}$, $a>0$ and $1/2<s<1$. In the same positive-slope model, ~\cite{ruppert1988efficient,polyak_acceleration_1992} show the same optimal asymptotic behavior as ASA, with a step-size that does not use $\beta$. Most of the well-known results assume the first order expansion of the regression function nears $\theta$ with positive derivative $f'(\theta)>0$.

The Robbins--Monro procedure and its variants like ASA and PJ averaging have attracted considerable attention in statistics, operations research, and machine learning, owing to their elegant and simple form. In particular, stochastic gradient descent (SGD) has been widely studied for applications involving large-scale datasets \citep{zhang_solving_2004,moulines_non-asymptotic_2011,ruder_overview_2017,vatter_evolution_2023} and dynamic distribution shifts \citep{perdomo_performative_2021,cutler_stochastic_2024}.
Recently, there has also been growing interest in asymptotic and non-asymptotic inference 
for root-finding problems~\cite{su_higrad_2023,lee_fast_2022,fang_online_2018}. For an overview of further applications and variants of stochastic approximation methods, we refer the reader to Section \ref{section:related work}, the surveys \citep{lai_stochastic_2003,lai_stochastic_2021}, the textbook \citep{kushner_stochastic_2003}, and the works cited therein. Despite the long history and far-reaching influence of Robbins--Monro \textsc{SA} on modern machine learning, few studies have examined its behavior beyond the positive-slope setting. The comparisons below focus on how the local properties of $f$ affect convergence, inference, and practical inconvenience brought by the recursive formula.

\textbf{Rate of convergence.} The theory established under the positive-slope model explains the role of the derivative in SA and the improvements obtained by ASA and PJ. However, a positive-slope expansion at the unknown root is a substantive restriction on the data generating process. The mean response may change sign with a vanishing derivative or switch sign abruptly at a threshold, producing weaker or stronger local signals than the positive-slope model describes. Unless such behavior can be ruled out in advance, it is desirable to develop procedures that \textit{adapt} to the local signal strength. Both the existing literature and our comparison results highlight the challenge and impossibility of attaining optimal estimation rates with SA variants when the statistician is agnostic about the local behavior of the regression function. 



\textbf{Inference.} 
Under asymptotic normality, a parametric convergence rate validates asymptotic inference via Wald-type confidence intervals, provided that the asymptotic variance can be estimated consistently. However, root approximation and variance estimation impose competing demands on the design: the former requires the points $\{X_i\}_{1\leq i\leq n}$ to concentrate around $\theta$, whereas the latter benefits from a more dispersed design. This tension can slow variance estimation dramatically, as noted in \cite{lai_adaptive_1979} and lead to conservative inference based on asymptotic approximations. Section~\ref{section:related work} discusses in detail how unknown features of $f$ affect variance estimation for ASA and PJ. Overall, there is a natural need for non-asymptotic coverage guarantees without substantially compromising confidence interval length.

\textbf{Practice.}
 Alongside these statistical properties, practical considerations shape experimental implementation. In engineering measurement and clinical dose-finding studies, changing a design point may require reconfiguration or coordination, making repeated measurements at a fixed setting attractive. In dose-finding studies, responses from different participants at the same dose provide information about a target toxicity or efficacy probability~\citep{yin2024sufentanil}. A successful response often leads to the next participant receiving the same dose. Clinical and regulatory rules also limit dose changes, so the next dose may not match the value prescribed by a \textit{recursive} SA update.
\begin{table}[!t]
\centering
\newcommand{\ratecell}[1]{\(\vcenter{\vskip4pt\hbox{\shortstack[c]{#1}}\vskip4pt}\)}
\renewcommand{\arraystretch}{1.25}
\setlength{\tabcolsep}{4pt}
\resizebox{\linewidth}{!}{
\begin{tabular}{|c|c<{\cellcolor{rateblue}}|c<{\cellcolor{rateblue}}|c<{\cellcolor{rateblue}}|c<{\cellcolor{rateblue}}|c<{\cellcolor{rateblue}}|c<{\cellcolor{rateyellow}}|}
\hline
& & \multicolumn{3}{|c<{\cellcolor{rateblue}}|}{$\gamma=1$ (positive-slope)} & & \\
\cline{3-5}
Method & $0<\gamma<1$
& $\beta>1/2$
& $\beta=1/2$
& $0<\beta<1/2$
& \ratecell{Flat root~\eqref{eq:higher order}\\[2pt]$\gamma>1$}
& \ratecell{Jump crossing\\[2pt]Assumption~\ref{assumption:jump point}} \\
\hline
\ratecell{\textsc{SA}\\[2pt]$a_n=1/n$\\[2pt]\citep{robbins_stochastic_1951}}
& \ratecell{$n^{-1/(1+\gamma)}$\\[2pt]$\frac12<\gamma<1$\\[2pt]
\citep{kaniovski_non-standard_1995}\\[2pt]
\texttt{NA}, $0<\gamma\leq\frac12$}
& \ratecell{$n^{-1/2}$\\[2pt]\citep{chung_stochastic_1954}}
& \ratecell{$\big(\log n/n\big)^{1/2}$\\[2pt]\citep{major_limit_1973}}
& \ratecell{$n^{-f'(\theta)}$\\[2pt]\citep{major_limit_1973}}
& \ratecell{$(\log n)^{-1/(\gamma-1)}$\\[2pt]\citep[Theorem~8.2]{ljung_stochastic_1992}}
& \ratecell{$n^{-1}$\\[2pt]\citep{lim_convergence_2011}} \\
\hline
\ratecell{\textsc{ASA}\\[2pt]\citep{lai_adaptive_1979,lai_consistency_1981}}
& \texttt{NA}
& \multicolumn{3}{|c<{\cellcolor{rateblue}}|}{\ratecell{$n^{-1/2}$\\[2pt]\textit{asymptotic normality}\\\textit{\& efficiency}}}
& \texttt{NA} & \texttt{NA} \\
\hline
\ratecell{\textsc{PJ averaging}\\\citep{polyak_acceleration_1992}}
& \ratecell{$O(n^{-s/(\gamma+1)})$\\[2pt]\textbf{Theorem~\ref{thm:pj-upper}}}
& \multicolumn{3}{|c<{\cellcolor{rateblue}}|}{\ratecell{$n^{-1/2}$\\[2pt]\textit{asymptotic normality}\\\textit{\& efficiency}}}
& \ratecell{$O\big(n^{-\{\frac{s}{\gamma+1}\wedge\frac{1-s}{\gamma-1}\}}\big)$\\[2pt]\textbf{Theorem~\ref{thm:pj-upper}}\\[2pt]
$\Omega(n^{-(1-s)/(\gamma-1)})$ if $s\in (s_\star,1)$\\
[2pt]\textbf{Proposition~\ref{prop:pj-flat-lower}}}
& \ratecell{$O(n^{-s})$\\[2pt]\textbf{Theorem~\ref{thm:pj-upper}}} \\
\hline
\ratecell{\textsc{STB}}
& \ratecell{$\widetilde O(n^{-1/(2\gamma)})$\\[2pt]\textbf{Theorem~\ref{thm:power}}}
& \multicolumn{3}{|c<{\cellcolor{rateblue}}|}{\ratecell{$\widetilde O(n^{-1/2})$\\[2pt]\textbf{Theorem~\ref{thm:power}}}}
& \ratecell{$\widetilde O(n^{-1/(2\gamma)})$\\[2pt]\textbf{Theorem~\ref{thm:power}}}
& \ratecell{$\exp(-\widetilde{O}(n^{1/2}))$\\[2pt]\textbf{Theorem~\ref{thm:jump}}} \\
\hline
\ratecell{\textsc{ASTB}}
& \texttt{NA}
& \multicolumn{3}{|c<{\cellcolor{rateblue}}|}{\ratecell{$n^{-1/2}$\\[2pt]\textit{asymptotic normality}\\[2pt]\textbf{Theorem~\ref{thm:normality}}\\[4pt]\textit{efficiency}\\[2pt]\textbf{Proposition~\ref{prop:weighted-gaussian-efficiency}}}}
& \texttt{NA} & \texttt{NA} \\
\hline
\ratecell{\textbf{Minimax}\\\textbf{lower bound}}
& \multicolumn{5}{|c<{\cellcolor{rateblue}}|}{\ratecell{$\Omega(n^{-1/(2\gamma)})$\\\textbf{Theorem~\ref{thm:minimax}}}}
& \ratecell{\texttt{NA}} \\
\hline
\end{tabular}}
\vspace{0.2cm}
\caption{Absolute root estimation error, with $n$ counting individual observations and $\beta=f'(\theta)$ in the positive-slope columns. Blue cells describe continuous crossings and yellow cells describe jumps. $\widetilde O(\cdot)$ omits logarithmic factors. We refer readers to Figure~\ref{fig:rate-exponents} in the Supplementary Material for a visual comparison of the convergence rates.}
\label{table:rate}
\end{table}

\textit{Main contributions.} We propose a conceptually simple algorithm, \emph{sequential-test bisection} (\textsc{STB}), implementing estimation via testing: the observations determine both how long to sample at a design point and how to update the interval. Initialized with a given interval containing the root, we sample at the mid-point until a sequential sign test crosses its threshold or the observation budget ends. If the sample average at the random stopping time is positive, we move to the left half of the interval{. Otherwise,} we move to the right half.  Each completed update halves the interval width, and the stagewise error allocation delicately controls the probability of ever choosing the incorrect half. 

\textsc{STB} is ``doubly adaptive'' in two senses. In the positive-slope model shown in the three columns headed $\gamma=1$ in Table~\ref{table:rate}, \textsc{STB} achieves a nearly parametric rate $\widetilde{O}(n^{-1/2})$ in Theorem~\ref{thm:power} without using $f'(\theta)$. More generally, under the condition $|f(x)|\gtrsim |x-\theta|^\gamma$ for $x$ near $\theta$ and some $\gamma>0$, Theorem~\ref{thm:power} establishes the high-probability rate $\widetilde O(n^{-1/(2\gamma)})$ without knowledge of $\gamma$. Its rate of convergence matches the minimax lower bound in Theorem~\ref{thm:minimax}. We attribute this adaptivity to sequential testing: the observed evidence determines how long to sample at each design point, and its sign determines the direction of the next interval update. ASA and classical PJ already achieve $\sqrt{n}$ accuracy under their positive-slope assumptions, as established in~\cite{lai_consistency_1981} and~\cite{polyak_acceleration_1992}. For flat roots with $\gamma>1$ and $0<s<1$, we establish the PJ upper bound $O(n^{-\{s/(\gamma+1)\wedge(1-s)/(\gamma-1)\}})$ in Theorem~\ref{thm:pj-upper}. Its largest exponent, $1/(2\gamma)$, is obtained at $s_\star=(\gamma+1)/(2\gamma)$, which uses the oracle information $\gamma$ unlike \textsc{STB}, which does not. When the regression function exhibits ``change-point'' behavior with a two-sided jump, $\lim_{x\uparrow\theta}f(x)=-\mu_-$ and $\lim_{x\downarrow\theta}f(x)=\mu_+$ as in the last column of Table~\ref{table:rate}, Theorem~\ref{thm:jump} establishes that \textsc{STB} enjoys a \textit{super-polynomial} rate of convergence in the total number of observations.

For inference, the sign tests and bisection structure also automatically generate confidence sequences with non-asymptotic time-uniform coverage, as established in Proposition~\ref{prop:type1-error}. Under the additional conditions of Corollary~\ref{cor:interval-length}, their lengths decay at the same order as the optimal estimation error.

For the positive-slope model, we replace bisection by randomized weighted-section sampling in Algorithm~\ref{alg:weighted-section} and fit a simple linear regression to all test observations in Algorithm~\ref{alg:weighted-estimator}. Theorem~\ref{thm:normality} establishes a $\sqrt n$ Gaussian limit and consistent variance estimation under a finite $(2+\delta)$th noise moment, removing the logarithmic loss in the root-error bound. Under Assumption~\ref{assumption:efficiency}, the estimate attains the minimum asymptotic variance in the stated model and the Gaussian information bound among regular estimators with predictable designs. The latter comparison takes the form of a H\'ajek--Le Cam convolution theorem.

\par
When several experimental inputs vary, the practical target can be the boundary of the input combinations for which the mean response stays below a prescribed threshold. Beyond univariate root finding, Theorem~\ref{thm:ray-reconstruction} combines ray estimates through local polynomial regression to recover a multivariate level-set function under Assumption~\ref{ass:star}.
 
\textit{\textbf{Organization}.}
Section~\ref{section:related work} reviews related topics: stochastic approximation, tests of power one and sequential inference. Section~\ref{section:algorithm} introduces the core estimating-via-testing procedure, sequential-test bisection (\textsc{STB}). Section~\ref{section:asymptotics} establishes non-asymptotic convergence bounds and an adaptive minimax lower bound, while Section~\ref{section:CI} develops finite-sample confidence sequences for the intervals. The proof of the convergence rate appears in Appendix~\ref{appendix:proof:thm:power}.
Section~\ref{section:normality-efficiency} develops weighted-section and augmented \textsc{STB}, connects randomly stopped averages to an all-observation estimate, and establishes asymptotic normality, variance estimation and efficiency in the specified models. We extend the problem to multivariate level-set estimation in  Section~\ref{sec:multivariate} by combining ray estimates through local polynomial regression for level-set recovery. The Supplementary Material contains the remaining technical details and simulation results.

\subsection{Notations}\label{subsec:notation}
The index $n$ counts individual observations and $t$ counts stages. The subscripts $\ell$ and $r$ denote left and right endpoints, as in $X_{\ell,t}$ and $X_{r,t}$. We use $a_n$ for Robbins--Monro step-sizes and $a$ for the constant in $a_n=a/n$. The parameter $\alpha_t=e^{-(K+t)}$ determines the sequential threshold at stage $t$. The parameter $s$ denotes a step-size exponent in Polyak--Juditsky averaging. We write $a_n\asymp b_n$ when there exist constants $\c,\C$ such that $\c a_n \leq b_n \leq \C a_n$ for sufficiently large $n$, and $\widetilde{\mathcal O}(b_n)$ for an upper bound up to logarithmic factors. The notation $a_n\sim b_n$ means $a_n/b_n\to1$. The notation $\mathcal O_P(1)$ means boundedness in probability and $o_P(1)$ means convergence to zero in probability. We use $\wedge$ and $\vee$ for minimum and maximum, and $\textsf{sign}(0)=0$. All logarithms are natural unless a base is displayed. The notation $R=b_0-a_0$ denotes the length of the interval $\mathcal I=[a_0,b_0]$. Generic constants $\c,\C>0$ may change between lines. 

\section{Related work}
\label{section:related work}
\textsc{STB}
 connects root-finding with noisy observations with sequential testing. We review the results most relevant to the present contexts and procedures.

\textit{Stochastic approximation.}
The method of~\cite{robbins_stochastic_1951} led to consistency and distributional results including those of~\cite{blum_approximation_1954,chung_stochastic_1954,sacks_asymptotic_1958}.In the linear model $f(x)=\beta(x-\theta)$, $\beta>0$, with a deterministic initial point and i.i.d.\ centered errors of variance $\sigma^2\in(0,\infty)$, the iteration with $a_n=a n^{-s}$, $a>0$, $1/2<s\leq1$, has the following limit, with $2a\beta>1$ required when $s=1$:
\begin{align}\label{eq:asymptotics_SA}
 n^{s/2}(X_{n+1}-\theta)\overset{d}{\longrightarrow}
 \mathcal{N}(0,\sigma_{s,\beta}^2),
\end{align}
where
\begin{align}\label{eq:asymptotics_SA_variance}
 \sigma_{s,\beta}^2=
 \begin{cases}
 a\sigma^2/(2\beta),&1/2<s<1,\\
 a^2\sigma^2/(2a\beta-1),&s=1,\quad 2a\beta>1.
 \end{cases}
\end{align}
For $s=1$, the second expression is minimized at $a=1/\beta$, {with minimum variance $\sigma^2/\beta^2$}. The same variance appears in~\cite[Theorem 2]{lai_adaptive_1979}, where the step-size constant is parametrized by its reciprocal.

As shown in the first row of Table~\ref{table:rate}, the change of rate for \textsc{SA} occurs at $2af'(\theta)=1$. With $a_n=a/n$ and $0<af'(\theta)<1/2$, the error can have order $n^{-af'(\theta)}$. Under the local remainder condition in~\cite[Theorem 2]{lai_adaptive_1979}, the normalized iterate has an almost-sure limit. For the flat-root expansion~\eqref{eq:higher order}, the logarithmic order $(\log n)^{-1/(\gamma-1)}$ is obtained under the conditions of~\cite[Assumption 8.1 and Theorem 8.2]{ljung_stochastic_1992}. These results concern the specified updates.

\textit{ASA and variance estimation.}
Slope estimation underlies the ASA procedures of~\cite{venter_extension_1967,lai_adaptive_1979,lai_consistency_1981}. They use least-squares slope estimates with either known positive lower and upper bound on $\beta$ or the specified stochastic truncation sequences, and have an $n^{1/2}$ Gaussian limit and $\sum_{i=1}^n(X_i-\theta)^2$ of logarithmic order. \cite{lai_adaptive_1979} show $\sqrt{\log n}(\widehat\beta_n^{\normalfont{\textsf{LS}}}-\beta)\Rightarrow\mathcal N(0,\beta^2)$. {Let $V=\sigma^2/\beta^2$ and $\widehat V_n^{\textsf{ASA}}=\widehat\sigma_n^2/(\widehat\beta_n^{\normalfont{\textsf{LS}}})^2$, with the latter set to zero if its denominator vanishes.} If $\sigma^2$ is known, or if $\widehat\sigma_n^2-\sigma^2=o_P((\log n)^{-1/2})$, {we obtain, by a first-order expansion of $x\mapsto x^{-2}$,}
\begin{equation}\label{eq:asa-variance-rate}
\sqrt{\log n}(\widehat V_n^{\textsf{ASA}}-V)\ \Rightarrow\ \mathcal N(0,4V^2).
\end{equation}
Consequently $(\widehat V_n^{\textsf{ASA}}/V)^{1/2}-1=O_P((\log n)^{-1/2})$, which describes the relative error in a plug-in standard error. {Slutsky's theorem still establishes} first-order normal inference for the root estimate. Multivariate adaptive schemes are studied by~\cite{wei_multivariate_1987}.

\textit{PJ averaging.}
\cite[Theorem~2]{polyak_acceleration_1992} prove that step-sizes $a_n=a n^{-s}$ with $a>0$ and $1/2<s<1$ {yield}
\begin{equation}\label{eq:pj-clt}
 \sqrt n(\overline X_n-\theta)\overset{d}{\longrightarrow}
 \mathcal{N}(0,\sigma^2/\beta^2).
\end{equation}
The limiting variance in~\eqref{eq:pj-clt} matches optimally tuned SA in~\eqref{eq:asymptotics_SA_variance}. Estimating this variance by substitution still requires estimates of $\sigma^2$ and $\beta$. Thus all three positive-slope columns in Table~\ref{table:rate} have the same PJ rate. Flat roots and jumps do not satisfy the required nonsingular local linearization.

To study flat roots and jumps without the positive-slope expansion required for the classical PJ limit, Section~\ref{SM:pj-rates} proves convergence rate when $f(x)\gtrsim |x-\theta|^\gamma$ for some $\gamma>0$. For $\gamma>1$ and $0<s<1$, the PJ mean absolute-error bound is $\mathbb{E}|\overline X_n-\theta|=O\!(n^{-\{s/(\gamma+1)\wedge(1-s)/(\gamma-1)\}}).$ The fastest rate in this upper bound occurs at $s=(\gamma+1)/(2\gamma)$, which requires knowing the signal strength $\gamma$. {Proposition~\ref{prop:pj-flat-lower} establishes matching examples} for $s>(\gamma+1)/(2\gamma)$. {For $0\leq\gamma\leq1$, Theorem~\ref{thm:pj-upper} implies} $O(n^{-s/(\gamma+1)})$, with $\gamma=0$ denoting its jump convention.

\textit{Tests of power one.}
At each stage, \textsc{STB} collects observations at a fixed location until their partial sum crosses either a positive or a negative threshold. For every nonzero mean the test eventually stops almost surely. At mean zero it can continue indefinitely. This is the power-one property studied in~\cite{farrell_asymptotic_1964,robbins_expected_1974,lai_power-one_1977}. Inspired by Wald's sequential probability ratio test~\citep{wald_sequential_1945}, the stagewise sequential test compares observation sums with a boundary $T(j,\alpha_t)$ \eqref{eq:boundary} constructed using the known sub-Gaussian scale $v$. However, our method is nonparametric and does not require a likelihood ratio. Expected sample sizes and error probabilities are also studied in~\cite{darling_further_1968,lai_nonlinear_1977,pollak_approximations_1975}. 

\textit{Anytime-valid confidence sequences.}
The recent popular interest in sequential procedure centers on developing confidence sequences for anytime-valid inference~\citep{ville1939etude,shafer_test-martingales_2011,darling_iterated_1967,balsubramani_sharp_2014,darling_confidence_1967,howard_time-uniform_2020,howard_time-uniform_2021}. Recent work on e-values and e-processes~\citep{vovk_e-values_2019,grunwald_safe_2024} provides a framework for constructing confidence sequences~\citep{ruf_composite_2023,ramdas_admissible_2022,ramdas_game-theoretic_2023,orabona_portfolio_2024} as well. Extensions cover vector-valued parameters under sub-$\psi$ or sub-Gaussian conditions~\citep{whitehouse_timeuniform_2026,chugg_spheres_2025}, functionals~\citep{casgrain_sequential_2024}, and quantiles of i.i.d. observations using bounded indicators~\citep{howard_quantiles_2022}. 
Motivated by this growing literature, we show how \textsc{STB} can be used for this purpose. For a prescribed tolerance level $\Delta$, \textsc{STB} uses sign-test crossings to select the left or right half-interval and guarantees that, with probability at least $1-\Delta$, every stage makes the correct update. Thus \textsc{STB} provides a confidence sequence automatically by covering the root $\theta$ with probability at least $1-\Delta$ simultaneously over time.


\textit{Level-set estimation.} \cite{tsybakov1997nonparametric} studies density level sets from \textit{i.i.d.} observations, obtaining minimax rates in symmetric-difference loss and Hausdorff distance by combining level-set smoothness with two-sided separation from the density level. For bounded regression responses under passive designs, \cite{willett2007minimax} derive risk bounds for penalized dyadic partitions, using a weighted symmetric-difference loss. \cite{castro_minimax_2008} study sequentially selected binary responses and obtain nearly minimax rates for excess classification risk over smooth level-set classes. In the same spirit, Section~\ref{sec:multivariate} estimates a regression level-set from sequential noisy measurements and controls the uniform radial error and the resulting Hausdorff distance. Level-set methods also support environmental monitoring and reliability assessment. 

\section{Sequential-test bisection}
\label{section:algorithm}

We now describe sequential-test bisection (\textsc{STB}), given in Algorithm~\ref{alg:stb}. Overall, the procedure combines sequential testing to determine the sign of $f(x)$ with bisection based on the test result.

\begin{algorithm}
\DontPrintSemicolon
\caption{\textsc{STB}: \textit{sequential-test bisection}}
\label{alg:stb}\label{alg:stage_sampling}
\KwIn{Known interval $\mathcal I=[X_{\ell,1},X_{r,1}]\owns \theta$, observation budget $n\in\mathbb N^+$, the noise bound $v$ in~\eqref{eq:noise-mgf},
error tolerance $\Delta\in(0,1)$.
Choose $K\geq\lceil\log_2(4/\Delta)\rceil$ and $T(j,\alpha_t)=v\sqrt{-2j\log(j+1)\log\alpha_t}$.}
$t\gets1$, $q\gets0$\;
\While{$q<n$}{
$X_t\gets(X_{\ell,t}+X_{r,t})/2$, $j\gets0$, $S\gets0$, $\alpha_t\gets e^{-(K+t)}$\;
\Repeat{
\label{line:threshold1}
$|S|>T(j,\alpha_t)$ or $q=n$}{
Observe $Y=f(X_t)+\varepsilon$ at the same location $X_t$\;
$j\gets j+1$, $q\gets q+1$, $S\gets S+Y$\;
}
\If{
\label{line:threshold2}
$|S|\leq T(j,\alpha_t)$}{Exit the loop in line~\ref{line:loop} with current interval $\mathcal I_t$\;
}
$\widehat f_t\gets S/j$\;
\eIf{$S>0$
\label{line:select_start}
}{
$X_{\ell,t+1}\gets X_{\ell,t}$, $X_{r,t+1}\gets X_t$\;
}{
$X_{\ell,t+1}\gets X_t$, $X_{r,t+1}\gets X_{r,t}$\;
\label{line:select_end}
}
$t\gets t+1$\;
$\mathcal I_t\gets [X_{\ell,t},X_{r,t}]$\;
}
\KwOut{The current interval and the estimate in~\eqref{eq:uncapped-output}}
\end{algorithm}

For the centered errors in model~\eqref{eq:DGP}, we use a known bound $v>0$ such that
\begin{equation}\label{eq:noise-mgf}
\mathbb{E}\{\exp(\lambda\varepsilon_n)\}\leq\exp(v^2\lambda^2/2),\qquad\lambda\in\mathbb R.
\end{equation}

Thus $v^2$ is a sub-Gaussian variance proxy, which may exceed the noise variance. The algorithm uses $v$ to scale its sequential-test threshold $T(j,\alpha_t)$, defined in Algorithm~\ref{alg:stb}.
The procedure has two time scales: $n$ counts observations and $t$ counts
sampling stages. At stage $t$, \textsc{STB} continues collecting observations at $X_t$ until the sample sum crosses the threshold in~\eqref{eq:boundary} or the observation budget is exhausted. A crossing determines the selected half of the interval, as in lines~\ref{line:select_start} to~\ref{line:select_end}, and, if sampling continues, its midpoint becomes the next design point. The current interval is $\mathcal I_t=[X_{\ell,t},X_{r,t}]$, where
$\ell$ and $r$ denote its left and right endpoints. Initially,
$\theta\in\mathcal I_1=\mathcal I$ is assumed known. Here $X_t$ denotes a stage location. The observation-indexed design in~\eqref{eq:DGP} repeats this location throughout the stage. Given $X_{t}\in \mathcal{I}$, each stage's sampling rule is a power-one test based on sample sums
\citep{lai_power-one_1977}: under a nonzero mean (i.e., $f(X_t)\neq 0$) it eventually stops with
probability one.

For $j$ i.i.d. centered errors of positive variance, the {law of iterated logarithm} yields fluctuations of order $\sqrt{j\log\log j}$. The threshold $v\sqrt{2j(K+t)\log(j+1)}$ permits a summable tail bound over both observation times and stages, as shown in~\eqref{eq:stage-error-sum}, which controls erroneous crossings. We discuss this further in Supplementary Material, Remark~\ref{remark:necessity of log}, and compare this threshold in light of law of iterated logarithm.

A stage can remain unfinished when the budget ends. We denote the current interval after $n$ observations by $\mathcal I(n)$ and use $m_t(n)$ to count observations actually collected at stage $t$, including an unfinished stage. The counts satisfy $\sum_t m_t(n)=n$. Select the most sampled stage, breaking ties by the smallest index, and the final estimate is constructed
as follows \begin{equation}\label{eq:uncapped-output}
 t_n=\min\mathop{\rm arg\,max}_{t:m_t(n)>0}m_t(n),\qquad
 \widehat\theta_n=\Pi_{\mathcal I(n)}(X_{t_n}),
\end{equation}
where $\Pi_{[a,b]}(x)=(x\vee a)\wedge b$. The estimate is the design point from the most sampled stage, projected onto the current interval, with ties resolved in favor of the earliest stage. This rule allows the convergence analysis to account for all $n$ observations even when the last sign test is unfinished, while projection cannot increase the estimation error whenever the interval contains $\theta$.
\subsection{Within-stage sequential test}
\label{subsec:stage}
For simplicity, we first consider i.i.d.\ centered errors satisfying~\eqref{eq:noise-mgf}. At each stage, the errors are independent of the preceding history.
At a location $X_t$ selected from past observations, let
$S_{j,t}=\sum_{i=1}^jY_{i,t}=jf(X_t)+\sum_{i=1}^j\varepsilon_{i,t}$.
The stopping time is
\begin{equation}\label{eq:criterion}
N_t=\inf\{j\geq1:|S_{j,t}|>T(j,\alpha_t)\},
\end{equation}
where the threshold has the exact form used in lines~\ref{line:threshold1} and~\ref{line:threshold2} of Algorithm~\ref{alg:stb}:
\begin{equation}\label{eq:boundary}
T(j,\alpha_t)=v\sqrt{-2j\log(j+1)\log\alpha_t},\qquad
\alpha_t=e^{-(K+t)},\quad K\geq2.
\end{equation}

The parameter $\alpha_t$ controls the threshold height. For the continuing experiment, $t$ indexes stages and $i$ indexes measurements within a stage. Let
\begin{equation}\label{eq:stage-filtration}
\mathcal F_t=\sigma\bigl(\mathcal F_0,X_s,N_s,Y_{i,s}\mathbf1_{\{i\leq N_s\}}:\ 1\leq s\leq t,\ i\geq1\bigr).
\end{equation}
Thus $\mathcal F_t$ {contains the locations}, stopping times and observations through stage $t$. On the event that stage $t$ is reached, $X_t$ is $\mathcal F_{t-1}$-measurable. The within-stage filtration is $\mathcal F_{t,j}=\sigma(\mathcal F_{t-1}\cup \sigma(X_t,Y_{1,t},\ldots,Y_{j,t}))$, with $\mathcal F_{t,0}=\mathcal F_{t-1}$. 
For each deterministic $j\geq1$ and $x>0$, Chernoff's inequality applied to the independent errors provides
$\mathbb P\Big(\pm\sum_{i=1}^j\varepsilon_{i,t}\geq x\Big)
\leq\inf_{\lambda>0}e^{-\lambda x+j\lambda^2v^2/2}
=e^{-x^2/(2jv^2)}$. A union bound therefore yields
\begin{equation}\label{eq:stage-concentration}
\mathbb P\Big(\exists j\geq1:
\Big|\sum_{i=1}^j\varepsilon_{i,t}\Big|>T(j,\alpha_t)
\,|\,\mathcal F_{t-1}\Big)
\leq2\sum_{j=1}^{\infty}(j+1)^{\log\alpha_t}.
\end{equation}
Therefore, the one-sided false-sign bound is uniform over all
positive mean locations:
\begin{equation}\label{eq:stage-wrong-sign}
\begin{split}
\sup_{x\in\mathcal I:\,f(x)>0}
\mathbb P_x(N_t<\infty,\ S_{N_t,t}<0)
&\leq\sup_{x\in\mathcal I:\,f(x)>0}
\mathbb P_x(\exists j\geq1:S_{j,t}<-T(j,\alpha_t))\\
&\leq\sum_{j=1}^{\infty}(j+1)^{\log\alpha_t}.
\end{split}
\end{equation}
The subscript $x$ means that this stage is sampled at the fixed location
$x$. The negative mean bound can be shown by symmetry. 

Recall that $\log \alpha_t = -(K + t)$. For $p=K+t>1$, we obtain the following accumulated error bound over stages.  We sum the within-stage crossing bounds over sampling times and stages and the resulting bound proves that all intervals contain $\theta$ with probability at least $1-\Delta$,
\begin{equation}\label{eq:stage-error-sum}
2\sum_{j=1}^{\infty}(j+1)^{-p}
\leq2^{1-p}\Bigl(1+\frac{2}{p-1}\Bigr),\quad
\sum_{t=1}^{\infty}2\sum_{j=1}^{\infty}(j+1)^{-(K+t)}
\leq2^{1-K}\Bigl(1+\frac{2}{K}\Bigr)\leq2^{2-K}\leq \Delta.
\end{equation}
This justifies the choice
$K\geq\lceil\log_2(4/\Delta)\rceil$ used in Algorithm~\ref{alg:stb} and we will formalize the result in Section~\ref{section:CI}. 

\section{Convergence theory}
\label{section:asymptotics}

We study model~\eqref{eq:DGP} under the following regularity  Assumption~\ref{assumption:regularity} on the noise and regression function.
\begin{assumption}\label{assumption:regularity}
There is a known interval $\mathcal I=[a_0,b_0]$ containing $\theta$ in its interior.
\begin{enumerate}
\item 
\label{condition:noise}
For each observation, the error $\varepsilon_n$ is independent of the past and of all other errors. The errors are centered, have common variance $0<\sigma^2<\infty$, and follow the sub-Gaussian moment generating function bound~\eqref{eq:noise-mgf} with the known input $v$.
\item 
\label{condition:sign}
The regression function $f$ satisfies~\eqref{eq:zero-crossing} and, for every $\eta>0$ for which the indexing set is nonempty,
\begin{align*}
\inf_{\{x\in\mathcal I:|x-\theta|\geq\eta\}}(x-\theta)f(x)>0.
\end{align*}
\item 
\label{condition:upper_envelope}
There are finite constants $\C_1,\C_2>0$ such that
\begin{equation}\label{eq:growth}
|f(x)|\leq \C_1|x|+\C_2,\qquad x\in\mathcal I.
\end{equation}
\end{enumerate}
\end{assumption}

Assumption~\ref{assumption:regularity}.\ref{condition:noise} imposes sub-Gaussianity~\citep{ledoux1991probability} on the noise, which we use to establish the non-asymptotic convergence and coverage guarantees in the following subsections. Conditional moment generating function assumptions of this form underlie the exponential-supermartingale construction in~\cite{howard_time-uniform_2021} and are commonly imposed in the literature on root-finding and stochastic optimization \citep[see, e.g.,][]{chen2025concentration,liu2023high,lan2012optimal,frikha2012concentration,nemirovski2009robust}. The procedure uses a known upper bound $v$ on the noise's $\psi_2$-Orlicz norm, $v\geq \|\varepsilon_i\|_{\psi_2}$. For bounded responses (e.g., binary responses), such a bound is readily available. Moreover, we do \textit{not} require $v\approx \|\varepsilon_i\|_{\psi_2}$. A conservative upper bound therefore improves the coverage in Proposition~\ref{prop:type1-error} at the cost of only a larger \textit{multiplicative factor} in the convergence guarantee of Theorem~\ref{thm:power}. See Remark~\ref{remark:tradeoff} for further discussion.

Assumption~\ref{assumption:regularity}.\ref{condition:sign} separates the signal from zero away from $\theta$, so a small function value forces a design point close to the crossing. {Assumption~\ref{assumption:regularity}.\ref{condition:upper_envelope} implies} $\sup_{x\in\mathcal I}|f(x)|\leq\C_\star=\C_1(|a_0|\vee|b_0|)+\C_2$. Related requirements appear in~\cite[Theorem~2.1, assumptions~1--2]{kaniovski_non-standard_1995}: their SA analysis imposes a global linear-growth bound and $(x-\theta)f(x)\geq\c(x-\theta)^2$ on $\mathbb R$. Here the design stays in $\mathcal I$, and the separation assumption allows flat roots, discontinuities and nonmonotone functions. Assumptions~\ref{assumption:regularity}.\ref{condition:sign} and~\ref{assumption:regularity}.\ref{condition:upper_envelope} are standard assumptions in much of the root-finding literature.

\subsection{Differentiable regimes}
\label{subsection:differentiable-regimes}
The convergence rate depends on how quickly $|f(x)|$ separates from zero near $\theta$. Assumption~\ref{assumption:power} includes positive derivatives, flat roots and local power behavior.
\begin{assumption}[Local signal lower bound]
\label{assumption:power}
For constants $\c_0,\gamma>0,\eta\in (0,|b_0-\theta|\wedge |a_0-\theta|)$, the regression function satisfies
\begin{equation}\label{eq:scalar-power-separation}
 |f(x)|\geq \c_{0}|x-\theta|^\gamma,\qquad |x-\theta|\leq \eta.
\end{equation}
\end{assumption}
The positive-slope and flat-root models in Examples~\ref{example:nonvanishing-derivative} and~\ref{example:flat-root} illustrate Assumption~\ref{assumption:power}. The assumption also permits nondifferentiable crossings. We call $\gamma$, the power in the local lower bound on $|f(x)|$, \textit{the signal strength}. The multiplicative constant $\c_0$ sets the scale of that bound. The next theorem implies an estimation error bound of order $\widetilde O(n^{-1/(2\gamma)})$ simultaneously over sufficiently large observation budgets at fixed confidence. Its polynomial exponent matches the adaptive Gaussian lower bound in Theorem~\ref{thm:minimax}. The sampling rule uses neither $\gamma$ nor $\c_0$.
The detailed proof together with the exact form for $n_0$ and $\C_\gamma$ are relegated to Appendix~\ref{appendix:proof:thm:power}. 
\begin{theorem}
\label{thm:power}
Under Assumptions~\ref{assumption:regularity} and~\ref{assumption:power}, let $\Delta\in(0,1)$ and $K\geq\lceil\log_2(4/\Delta)\rceil$. For Algorithm~\ref{alg:stb} with output~\eqref{eq:uncapped-output}, we write $T_n=\lceil\frac{\log_2(n+1)}{2\gamma}\rceil+1$ and $B_n(v)=4v\{T_n(K+T_n)\log(n+1)/n\}^{1/2}$. With probability at least $1-\Delta$, simultaneously for all $n$ with $n\geq4T_n$ and $B_n(v)<m_\eta:=\inf_{x\in\mathcal I:|x-\theta|\geq\eta}|f(x)|$, for some sufficiently small $\eta>0$,
\begin{equation}\label{eq:scalar-power-explicit}
|\widehat\theta_n-\theta|\leq R(n+1)^{-1/(2\gamma)}\vee\{B_n(v)/\c_0\}^{1/\gamma}.
\end{equation}
In particular, for all $n\geq n_0:=\inf_{n}\{n\geq 4T_n,B_n(v)<m_{\eta}\}$,
\begin{equation}\label{eq:scalar-power-rate}
|\widehat\theta_n-\theta|\leq\bigl[R\vee \C_\gamma(v/\c_0)^{1/\gamma}\bigr]\Big\{\frac{\log^2 n(K+\log n)}n\Big\}^{1/(2\gamma)},
\end{equation}
where $\C_\gamma$ depends only on $\gamma$. 
\end{theorem}

 At $X_t$, the partial sums determine how many repeated measurements are needed to reach the threshold. On the event controlling all stagewise errors, {a crossing determines the correct sign} and hence the next interval update. Weak evidence (i.e. when $|S_{j,t}|\leq T(j,\alpha_t)$) leads to further sampling. This heuristic explains why the same rule attains Theorem~\ref{thm:power} for unknown $\gamma$, including the positive-slope case with unknown $f'(\theta)$.

\begin{remark}\label{remark:martingale-noise}
When the noises are not necessarily independent,
Theorem~\ref{thm:power} also holds under the conditional mean and moment generating function bounds, with a common known $v$ and predictable sampling locations. Section~\ref{subsec:martingale-noise} proves this martingale-noise extension and specifies which other guarantees are preserved.
\end{remark}

The following examples relate the local behavior of $f$ to the rate in Theorem~\ref{thm:power}.
\begin{example}[Positive-slope model]\label{example:nonvanishing-derivative}
Suppose Assumption~\ref{assumption:regularity} holds, $f(\theta)=0$, and $f$ is differentiable at $\theta$ with $\beta=f'(\theta)>0$ and 
\begin{equation*}
 f(\theta+h)=\beta h+o(|h|),
\end{equation*}
for all sufficiently small $|h|$. Thus Assumption~\ref{assumption:power} holds with $\gamma=1$ and $\c_0=\beta/2$. For fixed $\Delta\in(0,1)$ and $K\geq\lceil\log_2(4/\Delta)\rceil$, {Theorem~\ref{thm:power} shows that} $
 |\widehat\theta_n-\theta|
 =O\big(\big(\log^3 n/n\big)^{1/2}\big)$
with probability at least $1-\Delta$, for all sufficiently large $n$.
\end{example}
\begin{example}[Flat root]\label{example:flat-root}
The regression function satisfies $f(\theta)=0$ and
\begin{equation}\label{eq:higher order}
 f(x)=\beta\textsf{sign}(x-\theta)|x-\theta|^\gamma
       (1+o(1)),\qquad x\to\theta,
 \qquad \beta>0,\quad\gamma>1.
\end{equation}
For all sufficiently small nonzero $h$, the factor $1+o(1)$ lies between $1/2$ and $3/2$. Consequently $|f(\theta+h)|\geq(\beta/2)|h|^\gamma$, while
\begin{equation*}
 \frac{f(\theta+h)-f(\theta)}h
 =\beta|h|^{\gamma-1}(1+o(1))\longrightarrow0.
\end{equation*}
Thus $f'(\theta)=0$, and Assumption~\ref{assumption:power} holds with $\c_0=\beta/2$. If Assumption~\ref{assumption:regularity} also holds, {Theorem~\ref{thm:power}, with fixed $\Delta$ and $K$, implies} $
 |\widehat\theta_n-\theta|
 =O\!\big\{\big(\log^3 n/n\big)^{1/(2\gamma)}\big\}$
with probability at least $1-\Delta$, for all sufficiently large $n$. {Larger $\gamma$ leads to a slower rate} because the signal is weaker near the root. The expansion permits any real $\gamma>1$. If it arises from a Taylor expansion whose first nonzero derivative has integer order $\gamma$, the sign change requires $\gamma$ to be odd and $\beta=f^{(\gamma)}(\theta)/\gamma!>0$.
\end{example}

\begin{example}   
Assumption~\ref{assumption:power} also permits $0<\gamma<1$. A precise SA comparison is available for $1/2<\gamma<1$. For the SA update with $a_n=1/n$ and deterministic initialization, suppose $f$ is continuous on $\mathbb R$, $(x-\theta)f(x)\geq\c(x-\theta)^2$ and $|f(x)|\leq\C_1|x|+\C_2$ for all $x\in\mathbb R$, with positive constants. The errors are independent and identically distributed, centered and sub-Gaussian, with positive variance. As $h\downarrow0$, suppose
\begin{equation*}
f(\theta+h)=\beta_+h^\gamma+O(h^{\gamma+\zeta}),\qquad
f(\theta-h)=-\beta_-h^\gamma+O(h^{\gamma+\zeta}),
\end{equation*}
where $\beta_+,\beta_->0$ and $\zeta>(1-\gamma)/(1+\gamma)$. \cite[Theorem~2.1]{kaniovski_non-standard_1995} imply $n^{1/(1+\gamma)}(X_{n+1}-\theta)\Rightarrow Z$ after $n$ observations, where $Z$ has a nondegenerate continuous density. Compared with the rate established in Theorem~\ref{thm:power}, as $1/(2\gamma)>1/(1+\gamma)$, Theorem~\ref{thm:power} shows the \textsc{STB} convergence strictly faster, even with its logarithmic factors. This comparison concerns probabilities at the same observation budget, and \cite{kaniovski_non-standard_1995} covers $1/2<\gamma<1$. For PJ in this range, Theorem~\ref{thm:pj-upper} provides upper bounds.

A one-sided jump can also satisfy  Assumption~\ref{assumption:power}: if $f(\theta+h)\sim\beta h^\gamma$ as $h\downarrow0$ and $f(\theta-h)\to-c$ with $\beta,c>0$, then for sufficiently small $\eta$, one may take $\c_0=\min\{\beta/2,c/(2\eta^\gamma)\}$. Its rate is governed by the vanishing side through Theorem~\ref{thm:power}, under Assumption~\ref{assumption:regularity}.

\end{example}

\subsection{Minimax optimality in the differentiable regimes}
\label{subsec:minimax}
The preceding bounds raise the question of which part of their dependence on the signal strength $\gamma$ is unavoidable. We compare sequential procedures at the same deterministic observation budget $n$. For $\c_0,\gamma,\underline{\eta}>0$, the class $\mathcal F_\gamma^{\loc}(\c_0,\underline{\eta})$ consists of functions on $\mathcal I$ with a crossing $\theta(f)\in(a_0,b_0)$ that satisfy Assumption~\ref{assumption:regularity}.\ref{condition:sign} and
\begin{equation*}
 |f(x)|\geq\c_0|x-\theta(f)|^\gamma
 \quad\text{for }x\in\mathcal I\text{ with }|x-\theta(f)|\leq\underline{\eta}.
\end{equation*}
We write $\H_i=(X_1,Y_1,\ldots,X_i,Y_i)$ as the information up to $i$th observation. A procedure $\A\in\mathcal A_n$ specifies measurable probability kernels $K_i(dx\mid\H_{i-1})$ on $\mathcal I$, $1\leq i\leq n$, and a terminal kernel $K_{n+1}(dz\mid\H_n)$ for $\widehat\theta_n(\A)$, with $\H_0$ empty. Conditional on the observed history $\H_{i-1}$, $K_i(\cdot\mid\H_{i-1})$ is the distribution of the next location $X_i$, after which we observe $Y_i=f(X_i)+\varepsilon_i$ with an independent error. In our \textsc{STB}, it is a deterministic kernel.

\begin{theorem}
\label{thm:minimax}
Suppose the errors are independent $\mathcal{N}(0,\sigma^2)$ random variables, with $\sigma>0$. For every $\gamma>0$ there is a constant $\c_\star=\c_\star(\gamma)>0$ such that, for all $n\geq1$ and $\c_0,\underline{\eta}>0$,
\begin{align}
    \inf_{\A\in \mathcal{A}_n} \sup_{f\in \mathcal{F}_\gamma^{\loc}(\c_0,\underline{\eta})}
    \bbP\Big(\big|\widehat{\theta}_n(\A)- \theta(f)\big|\geq \c_\star\Big\{R\wedge \Big(\frac{\sigma^2}{\c_0^2 n}\Big)^{1/(2\gamma)}\Big\}\Big)\geq \frac{1}{4}.
\end{align}
Consequently,
\begin{align}
 \inf_{\A\in\mathcal A_n}\sup_{f\in\mathcal F_\gamma^{\loc}(\c_0,\underline{\eta})}
 \mathbb{E}\big|\widehat\theta_n(\A)-\theta(f)\big|^2
 \geq \frac{\c_\star^2}{4}
 \Big\{R^2\wedge\Big(\frac{\sigma^2}{\c_0^2 n}\Big)^{1/\gamma}\Big\}.
\end{align}
\end{theorem}
\begin{remark}
\label{remark:minimax-randomization}
The lower bound allows adaptive locations and counts repeated measurements separately. It identifies $n^{-1/(2\gamma)}$ as the unavoidable polynomial scale in the Gaussian experiment, even for continuous strictly increasing alternatives. Similar minimax lower bounds for expected point error are given by~\cite{ramdas2013optimal} for Lipschitz convex objectives with global growth power $\gamma+1$, {with the same polynomial exponent $1/(2\gamma)$ in the number of noisy first-order queries}. 
\end{remark}

\subsection{Discontinuous crossings}
\label{subsection:discontinuous regression function}
The preceding rate bound converts signal size into root error through a local power lower bound. When the regression function exhibits ``change-point" behavior, $|f(x)|$ stays bounded away from zero for $x$ sufficiently close to, and distinct from, the root $\theta$. We use this separation to obtain faster-than-polynomial estimation error decay in the total observation budget, with the SA and PJ comparisons stated in Table~\ref{table:rate} and Section~\ref{SM:pj-rates}.
\begin{assumption}[Local two-sided jump]
\label{assumption:jump point}
There are constants $\mu_-,\mu_+>0$ such that
$\lim_{x\uparrow\theta}f(x)=-\mu_-$ and
$\lim_{x\downarrow\theta}f(x)=\mu_+$.
\end{assumption}
The one-sided limits imply that $|f(x)|\geq\min(\mu_-,\mu_+)/2$ in a sufficiently small punctured neighborhood of $\theta$. Assumption~\ref{assumption:regularity}.\ref{condition:sign} extends this to a positive lower bound on $|f|$ throughout $\mathcal I\setminus\{\theta\}$. Like the differentiable regimes, the algorithm uses the observed sign evidence without knowing the gap.

\begin{theorem}
\label{thm:jump}
Under Assumptions~\ref{assumption:regularity} and~\ref{assumption:jump point}, let $\Delta\in(0,1)$ and $K\geq\lceil\log_2(4/\Delta)\rceil$, $b=\inf_{x\in\mathcal I\setminus\{\theta\}}|f(x)|>0$, and $\kappa=(b\log 2)/(16\sqrt2\,v).$ Let $n_0$ be the smallest integer $n\geq3$ satisfying $n/\log(n+1)\geq(256v^2(K+1)^2)/b^2$ and $ n\log(n+1)\geq b^2/4v^2.$ For Algorithm~\ref{alg:stb} with output~\eqref{eq:uncapped-output}, simultaneously for all $n\geq n_0$,
\begin{equation}\label{eq:jump-budget}
 |\widehat\theta_n-\theta|
 \leq R\,\exp\big\{-\kappa\sqrt{n/\log(n)}\big\},
\end{equation}
with probability at least $1-\Delta$.
\end{theorem}
\begin{remark}
A jump keeps the testing signal away from zero as the locations approach $\theta$. Here $t$ indexes stages, and $N_t$ counts measurements until the sum crosses $v\sqrt{2j(K+t)\log(j+1)}$ in absolute value. Balancing the drift $j|f(X_t)|$ with this threshold leads to $j/\log(j+1)=2v^2(K+t)/f(X_t)^2$. For fixed $K,v$ and a common i.i.d. error distribution, Proposition~\ref{thm:EN} makes this balance precise: $\mathbb E[N_t\mid\mathcal F_{t-1}]\asymp t\log t$ for a fixed nonzero signal. As each interval update shortens the interval by half, summing these stage-cost scales heuristically suggests $n\asymp t^2\log t$, explaining the exponent in~\eqref{eq:jump-budget}. This is a super-polynomial bound in the observation budget. The explicit constant $\kappa=b\log2/(16\sqrt2\,v)$ characterizes the effects of the signal gap.
\end{remark}

\section{Automatic confidence sequences}
\label{section:CI}
In Algorithm~\ref{alg:stb}, the interval $\mathcal{I}(n)$ is updated by checking the sign of the randomly stopped average and halving the interval. We now show that the sequential testing built into each stage makes these intervals a confidence sequence. Proposition~\ref{prop:type1-error} allows practitioners to inspect the interval after each stage and stop when a desired width is reached, or report the current interval when the measurement budget is exhausted, while preserving the coverage probability. \par Corollary~\ref{cor:interval-length} complements this coverage guarantee by bounding how quickly the interval narrows under the midpoint conditions. Together, these results relate a desired accuracy to the number of observations needed, including the flat root and jump crossing scenarios.

\begin{definition}
\label{def:CS}
A sequence of random intervals $\{\mathcal I_t\}_{t\geq1}$ is a
$(1-\Delta)$ confidence sequence for $\theta$ if
\begin{equation}\label{eq:cs-definition}
\mathbb P(\mathcal I_t\owns \theta\text{ for every reached stage }t)
\geq1-\Delta.
\end{equation}
\end{definition}

Time-uniform coverage already appears in the classical work of
\cite{darling_confidence_1967}. Modern constructions include those of
\cite{howard_time-uniform_2021,waudby-smith_estimating_2024}. Building on PJ averaging, \cite{xie_asymptotic_2024} obtain asymptotic time-uniform inference for averaged iterates in the positive-slope model using Gaussian approximation techniques under moment conditions on the noise. However, it is unclear how their results extend to the weaker-signal setting with $\gamma>1$. Their approach also requires estimating the asymptotic variance of the root estimator. Here, under the sub-Gaussian noise condition in Assumption~\ref{assumption:regularity}.\ref{condition:noise}, we obtain a non-asymptotic time-uniform confidence sequence \textit{without} variance estimation, covering all cases where $|f(x)|\gtrsim |x-\theta|^\gamma$ as specified in Assumption~\ref{assumption:power}.

\begin{proposition}[Finite-sample simultaneous coverage]
\label{prop:type1-error}
Under Assumptions~\ref{assumption:regularity}.\ref{condition:noise} and \ref{assumption:regularity}.\ref{condition:sign}, and with an initial interval containing $\theta$, Algorithm~\ref{alg:stb} uses
$\alpha_t=e^{-(K+t)}$, where
$K\geq\lceil\log_2(4/\Delta)\rceil$ and $\Delta\in(0,1)$. Then the intervals $\mathcal{I}_t = [X_{\ell,t},X_{r,t}]$ form a $(1-\Delta)$ confidence sequence for $\theta$:
\begin{equation}\label{eq:stb-cs}
\mathbb P\bigl(\exists\text{ reached stage }t:
\mathcal I_t\not\owns \theta\bigr)\leq\Delta.
\end{equation}
\end{proposition}
To complement the coverage results, we establish the following bound on the interval length. The more elaborated version of the result below can be found in Corollary~\ref{cor:interval-length-full}.
\begin{corollary}
\label{cor:interval-length}
Under Assumption~\ref{assumption:regularity}, run Algorithm~\ref{alg:stb} with fixed $\Delta\in(0,1)$ and $K\geq\lceil\log_2(4/\Delta)\rceil$. Let $X_t^\circ$ be the successive midpoints obtained by bisection. With probability at least $1-\Delta$, the following bounds hold simultaneously for all sufficiently large $n$.
\begin{enumerate}
    \item 
Under Assumption~\ref{assumption:power}, if $|X_t^\circ-\theta|\geq\lambda R2^{-(t-1)}$ for every $t\geq1$ and some fixed $\lambda\in(0,1/2)$, then, for $n\geq n_0$ with $n_0$ and $\C_\gamma$ as in Theorem~\ref{thm:power},
\begin{equation}\label{eq:interval-power-rate}
|\mathcal I(n)|\leq
\bigl[R\vee\lambda^{-1}\C_\gamma(v/\c_0)^{1/\gamma}\bigr]
\Bigl\{\frac{\log^2 n(K+\log n)}{n}\Bigr\}^{1/(2\gamma)}.
\end{equation}
\item 
Under Assumption~\ref{assumption:jump point}, if $X_t^\circ\neq\theta$ for every $t\geq1$, then, for $n\geq n_0$ with $b$ and $n_0$ as in Theorem~\ref{thm:jump},
\begin{equation}\label{eq:interval-jump-rate}
|\mathcal I(n)|\leq R\,\exp\{-\kappa\sqrt{n/\log n}\},
\qquad
\kappa=\frac{b\log 2}{16\sqrt{2}\,v}.
\end{equation}

\end{enumerate}
\end{corollary}
\begin{remark}
  For $f(x)=x$ on $\mathcal I=[-1,1]$, the first midpoint is the root $\theta=0$ and $f(0) = 0$. With overwhelming probability, the partial sums $S_{1,j}=\sum_{i=1}^j\varepsilon_{1,j}$ never cross the threshold $T(j,\alpha_t)$ in~\eqref{eq:boundary}, so $|\mathcal I(n)|=2$ for every $n$, although the first design point has zero estimation error. The separation condition $|X_t^\circ-\theta|\geq\lambda R2^{-(t-1)}$ in the power case of Corollary~\ref{cor:interval-length} excludes finite-stage exact root hits. In the jump case, the condition $X_t^\circ\neq\theta$ serves the same purpose.
\end{remark}

\begin{remark}\label{remark:interval-stalling}
We recommend \textsc{STB} to practitioners seeking a root estimate together with an interval that remains valid during sequential sampling. Choose the confidence level $1-\Delta$ and the corresponding $K\geq\lceil\log_2(4/\Delta)\rceil$, then report the last interval when it reaches the desired width or the available budget is exhausted.
\end{remark}

\section{Asymptotic normality and efficiency}
\label{section:normality-efficiency}
The preceding sections establish convergence rates and simultaneous interval coverage. We now specialize to the positive-slope model. By replacing stage-wise bisection with \textit{randomized weighted-section } and fitting a line to all observations, we remove the logarithmic loss from the $\sqrt{n}$ limit and attain asymptotic efficiency. We now discuss the augmented algorithm.


\subsection{Weighted-section rule and augmented \textsc{STB}}
\label{sec:weighted-stb}\label{app:weighted-efficiency}

Similar to \textsc{STB}, the inputs are a fixed threshold scale $v>0$ and an integer $K=K(n)$ fixed during the run. We use $\alpha_t=e^{-(K+t)}$ and the threshold
$T(j,\alpha_t)=v\sqrt{2j(K+t)\log(j+1)}$ as before. To have asymptotic distribution, we choose $K(n)=K_0+\lceil\log\log(n+e^e)\rceil, K_0\geq2$.
More generally, the results below require deterministic integers satisfying
\begin{equation}
 K\geq2,\qquad K\to\infty,\qquad K=o(\log n).
 \label{eq:weighted-K-conditions}
\end{equation}
Unlike the non-asymptotic analysis in Section~\ref{section:asymptotics}, the asymptotic analysis below only requires the Lyapunov condition~\eqref{eq:weighted-noise-moment} and the threshold scale $v$ need not bound the noise variance or tails.

\begin{figure}[!t]
\begin{minipage}{\linewidth}
\begin{algorithm}[H]
\color{black}
\caption{Augmented \textsc{STB}}
\label{alg:weighted-section}\label{alg:weighted-efficient}
\DontPrintSemicolon
\KwIn{$\mathcal I=[a_0,b_0]$, observation budget $n\geq1$, scale $v>0$, integer $K\geq2$.}
{Initialize $q\gets 0$, $(X_{\ell,1},X_{r,1})=(a_0,b_0)$, and $s_x=s_y=s_{xx}=s_{xy}=s_{yy}=0$}\;
\While{$q<n$
\label{line:loop}
}{
 $\lambda_t^0\gets 1/2$\\
 \If{both endpoints have been sampled from and their corresponding signal estimates satisfy $\widehat f_\ell<0<\widehat f_r$}{$\lambda_t^0\gets-\widehat f_\ell/(\widehat f_r-\widehat f_\ell)$\;}   
 Draw $Z_t\sim\normalfont{\textsf{{Unif}}}[-1/8,1/8]$ independently\;
 $\lambda_t\gets\normalfont{\textsf{clip}}_{[3/8,5/8]}(\lambda_t^0)+Z_t$ and $X_t\gets (1-\lambda_t)X_{\ell,t}+\lambda_tX_{r,t}$\;
$j\gets 0$ and $S\gets 0$\;
 \Repeat{$|S|>T(j,\alpha_t)$ or $q=n$}{
 {Observe $Y$ at $X_t$, $q\gets q+1$,\\ $s_x\gets s_x+X_t$, $s_y\gets s_y+Y$, $s_{xx}\gets s_{xx}+X_t^2$, $s_{xy}\gets s_{xy}+X_tY$, $s_{yy}\gets s_{yy}+Y^2$, and $(j,S)\gets(j+1,S+Y)$}\;
 }
 \If{$|S|\leq T(j,\alpha_t)$}{Exit the loop without an interval update\;}
 {$N_t\gets j$}\;
 \eIf{$\normalfont{\textsf{sign}}(S)>0$}{$(X_{\ell,t+1},X_{r,t+1})\gets(X_{\ell,t},X_t)$\;}{$(X_{\ell,t+1},X_{r,t+1})\gets(X_t,X_{r,t})$\;}
 {Replace the mean of the updated endpoint by $S/N_t$ and leave the other endpoint mean unchanged}\;
 $t\gets t+1$\;
}
\KwOut{{$n$ and the sums $s_x,s_y,s_{xx},s_{xy},s_{yy}$, including every observation in an incomplete final test, for Algorithm~\ref{alg:weighted-estimator}.}}
\end{algorithm}

\begin{algorithm}[H]
\color{black}
\caption{Final root and variance estimates from all observations}
\label{alg:weighted-estimator}
\DontPrintSemicolon
\KwIn{{$\mathcal I=[a_0,b_0]$, $n$ and the sums $s_x,s_y,s_{xx},s_{xy},s_{yy}$ from Algorithm~\ref{alg:weighted-section}.}}
Compute $\bar X_n=s_x/n$, $\bar Y_n=s_y/n$, $D_n=s_{xx}-s_x^2/n$ and $C_n=s_{xy}-s_xs_y/n$\; $\widehat\beta_n^{\normalfont{\textsf{LS}}}\gets C_n/D_n$ if $D_n>0$, and zero otherwise\;
\eIf{$D_n>0$ and $\widehat\beta_n^{\normalfont{\textsf{LS}}}>0$}{
 $\widehat\theta_n\gets \bar X_n-\bar Y_n/\widehat\beta_n^{\normalfont{\textsf{LS}}}$\;
 $\widehat\sigma_n^2\gets (s_{yy}-s_y^2/n-C_n^2/D_n)/n$ and $\widehat V_n=\widehat\sigma_n^2/(\widehat\beta_n^{\normalfont{\textsf{LS}}})^2$\;
}{
$\widehat\theta_n\gets(a_0+b_0)/2$ and $\widehat\sigma_n^2=\widehat V_n=0$\;
}
\KwOut{$\widehat\theta_n$, $\widehat\beta_n^{\normalfont{\textsf{LS}}}$, $\widehat\sigma_n^2$ and $\widehat V_n$.}
\end{algorithm}
\end{minipage}
\end{figure}

{We denote $\normalfont{\textsf{clip}}_{[a,b]}(u)=\max\{a,\min\{b,u\}\}$. Algorithm~\ref{alg:weighted-section} combines weighted-section positions with sequential sign tests. Algorithm~\ref{alg:weighted-estimator} constructs the final estimates by fitting a line to all observations.}

\begin{remark}[Linear interpolation and stopped means]\label{remark:weighted-intuition}

For $f(x)=\beta(x-\theta)$ and known endpoint values at $\ell<\theta<r$, the proposed weight is $-f(\ell)/\{f(r)-f(\ell)\}=(\theta-\ell)/(r-\ell)$, whose section point is exactly $\theta$. The sampling rule uses the endpoint means, clips the weight to $[3/8,5/8]$ and perturbs the proposed point before testing.
\end{remark}

{The positive and negative crossings replace the right and left endpoints, respectively. The proposed weight interpolates the endpoint means, each computed at the end of its sign test. Clipping and independent perturbation place $X_t$ in the central half of the interval, so each interval update contracts its width by a factor in $[1/4,3/4]$. The perturbation $Z_t$ makes an exact root hit have probability zero and controls exceptionally small nonzero signals. Once the endpoint and its mean have been updated, sampling proceeds immediately to the next stage if budget remains.}

{Let $t(i)$ be the stage containing observation $i$ in the order collected, so $Y_i=f(X_{t(i)})+\varepsilon_i$. Here $n$ is the actual total number of test observations, including any incomplete final test. We define}

\begin{equation}
\begin{aligned}
 \bar X_n&=n^{-1}\sum_{i=1}^nX_{t(i)},\qquad
 \bar Y_n=n^{-1}\sum_{i=1}^nY_i, \qquad D_n=\sum_{i=1}^n(X_{t(i)}-\bar X_n)^2,\\
 \widehat\beta_n^{\normalfont{\textsf{LS}}}
 &=\frac{\sum_{i=1}^n(X_{t(i)}-\bar X_n)(Y_i-\bar Y_n)}{D_n},\qquad
 \widehat\theta_n=\bar X_n-\frac{\bar Y_n}{\widehat\beta_n^{\normalfont{\textsf{LS}}}}.
\end{aligned}\label{eq:weighted-LS}
\end{equation}

We set $\widehat\beta_n^{\normalfont{\textsf{LS}}}=0$ if $D_n=0$. The root formula is used when $D_n>0$ and $\widehat\beta_n^{\normalfont{\textsf{LS}}}>0$. Otherwise we return the initial midpoint and set $\widehat\sigma_n^2=\widehat V_n=0$. The variance estimates are
\begin{equation}
\begin{aligned}
 \widehat\sigma_n^2&=\frac1n\sum_{i=1}^n
 \bigl\{Y_i-\bar Y_n-\widehat\beta_n^{\normalfont{\textsf{LS}}}(X_{t(i)}-\bar X_n)\bigr\}^2,\quad
 \widehat V_n=\widehat\sigma_n^2/(\widehat\beta_n^{\normalfont{\textsf{LS}}})^2.
\end{aligned}\label{eq:weighted-variance-estimate}
\end{equation}

Now we state the formal assumptions needed for the asymptotic normality and efficiency for \textsc{ASTB}.

\begin{assumption}[Positive-slope model]
\label{ass:weighted-efficient}
The regression function is finite and Borel measurable on $\mathcal I$, with $\theta\in(a_0,b_0)$.
\begin{enumerate}
    \item
    \label{condition:weighted-efficient:expansion}
There are constants $\beta>0$, $0\leq L<\infty$, and $r_0>0$ such that
\begin{equation}\label{eq:weighted-remainder}
 f(x)=\beta(x-\theta)+R_f(x),\qquad
 |R_f(x)|\leq L|x-\theta|^2\quad (|x-\theta|\leq r_0).
\end{equation}
\item \label{condition:weighted-efficient:noise}
The errors are independent and identically
distributed, independent of past data, with
mean zero, variance $0<\sigma^2<\infty$, and
\begin{equation}\label{eq:weighted-noise-moment}
 \mathbb E|\varepsilon_i|^{2+\delta}<\infty
 \qquad\text{for some }\delta>0.
\end{equation}
\end{enumerate}

\end{assumption}

The expansion implies $f'(\theta)=\beta$. Together with Assumption~\ref{assumption:regularity}.\ref{condition:sign}, it implies a global lower bound and local upper bounds. Choose $0<\rho\le r_0$, with $[\theta-\rho,\theta+\rho]\subset\mathcal I$, small enough that
\begin{equation}
\begin{aligned}
 |f(x)|&\ge\c_f|x-\theta| &&(x\in\mathcal I),\\
 |f(x)|&\le\C_f|x-\theta|,\qquad
 |R_f(x)|\le L|x-\theta|^2 &&(|x-\theta|\le\rho),
\end{aligned}\label{eq:weighted-global-separation}
\end{equation}
for some $\c_f,\C_f>0$.

Let $\mathcal H_0$ contain the fixed inputs and independent random variables $(Z_t)_{t\ge1}$, and define the observation-indexed filtration $
 \mathcal H_i=\sigma\bigl(\mathcal H_0,(X_{t(j)},Y_j):1\le j\le i\bigr).$
The next location is $\mathcal H_i$-measurable and $\varepsilon_{i+1}$ is independent of $\mathcal H_i$. In the one-stage analysis, $\mathcal F_{t-1}$ contains all preceding test observations and the chosen location $X_t$.

\subsection{Randomly stopped averages}
\label{subsection:RSA}
{The endpoint means depend on their sign-test stopping times.} We now quantify the observation cost and fluctuation of these averages conditional on the information available before the test. For Algorithm~\ref{alg:stb}, this is the history in~\eqref{eq:stage-filtration}. 
For Algorithm~\ref{alg:weighted-section}, it includes the preceding tests and the current position.
For the one-stage results in this subsection, the $\mathcal F_{t-1}$-measurable mean $\mu_t:=f(X_t)$ is assumed to satisfy $0<|\mu_t|\leq\C_\star$ almost surely. Conditional on $\mathcal F_{t-1}$, the
observations are $Y_{i,t}=\mu_t+\varepsilon_{i,t}$. In this subsection, the errors are i.i.d., independent of the pre-test history, with mean zero, variance $0<\sigma^2<\infty$, and the finite moment~\eqref{eq:weighted-noise-moment}. The uniform limits below hold for fixed $\C_{\star}$, $v$, $K$ and a fixed common error {distribution}, while the nonzero stage means may approach zero or change sign.  We define $S_{j,t}=\sum_{i=1}^jY_{i,t}$ and
\begin{equation*}
 N_t=\inf\{j\geq1:|S_{j,t}|>T(j,\alpha_t)\},\qquad
 M_t=S_{N_t,t}/N_t,
 \qquad \alpha_t=e^{-(K+t)},
\end{equation*}
with the threshold~\eqref{eq:boundary}. Here $N_t$ is the first crossing time of the continuing within-stage test. 

\begin{proposition}
\label{thm:EN}
We set $\delta_t=-2v^2\mu_t^{-2}\log\alpha_t$. {Under Assumption~\ref{ass:weighted-efficient}.\ref{condition:weighted-efficient:noise}, there are constants $\c,\C>0$ and $t_0\in\mathbb N$, depending on $\C_\star,v,K$ and the fixed error {distribution}, such that,}
for every $t\geq t_0$, almost surely,
\begin{equation}
\label{eq:N_order}
 \c\delta_t\log(\delta_t+1)
 \leq\mathbb{E}[N_t\mid\mathcal F_{t-1}]
 \leq \C\delta_t\log(\delta_t+1).
\end{equation}
Moreover, we have
\begin{equation}
\label{eq:ratio-1}
 \frac{\mathbb{E}[N_t\mid\mathcal F_{t-1}]}
 {\delta_t\log(\delta_t+1)}\longrightarrow1
 \qquad\text{almost surely}.
\end{equation}
The convergence is uniform over $0<|\mu_t|\leq \C_{\star}$ and, in particular, no convergence
assumption on $\mu_t$ is required.
\end{proposition}
\begin{remark}
Since $\delta_t=2v^2(K+t)/\mu_t^2$, the expected cost increases both when the signal is small and when the stage index raises the threshold. Uniformity in $0<|\mu_t|\leq \C_{\star}$ permits substitution of the adaptive mean $f(X_t)$ without assuming that the design locations converge.
\end{remark}
 We next study the limiting distribution of the stopped average $M_t$, which plays a key role in establishing the asymptotic normality for \textsc{ASTB} later.

\begin{theorem}
\label{thm:RSA}
In the preceding setting, for all sufficiently large $t$, let $\mathfrak{n}_t>1$
be the unique solution of
\begin{equation}
\label{eq:RSA-centering}
 \mathfrak{n}_t=\delta_t\log(\mathfrak{n}_t+1).
\end{equation}
As $t\to\infty$, the following conditional distribution functions, given $\mathcal F_{t-1}$, converge almost surely at every real argument, uniformly over $0<|\mu_t|\leq \C_{\star}$:
\begin{align}
 \frac{N_t-\mathfrak{n}_t}{2\sqrt{\mathfrak{n}_t/\mu_t^2}}
 &\overset{d}{\longrightarrow}\mathcal{N}(0,\sigma^2),
 \label{eq:RSA1}\\
 \sqrt{N_t}(M_t-\mu_t)
 &\overset{d}{\longrightarrow}\mathcal{N}(0,\sigma^2).
 \label{eq:RSA2}
\end{align}
\end{theorem}




\begin{proposition}\label{prop:weighted-pooled}
\begingroup
For a continuing stage of Algorithm~\ref{alg:weighted-section}, the updated endpoint mean is $\widehat f_t=M_t=S_{N_t,t}/N_t$. As $K+t\to\infty$,
\begin{equation}
 \sqrt{N_t}(\widehat f_t-\mu_t)\xrightarrow{d}\mathcal N(0,\sigma^2),\qquad
 \widehat f_t/\mu_t\xrightarrow{P}1.
 \label{eq:weighted-pooled}
\end{equation}
The conditional limits are uniform over $0<|\mu_t|\le\C_\star$ for each fixed finite $\C_\star$.
\endgroup
\end{proposition}
\begin{remark}
{The first limit uses the actual test count $N_t$. The second follows from stopping-time localization and $|\mu_t|\sqrt{N_t}\to\infty$.Section~\ref{supp:weighted-rsa} proves the proposition and explains its use without a global upper bound on $f$.}
\end{remark}
\subsection{Asymptotic normality and variance estimation}
\label{subsec:weighted-terminal-efficiency}
\begin{theorem}\label{thm:normality}
{Under Assumption~\ref{assumption:regularity}.\ref{condition:sign}, Assumption~\ref{ass:weighted-efficient}, and~\eqref{eq:weighted-K-conditions}, the estimator constructed by Algorithms~\ref{alg:weighted-section} and~\ref{alg:weighted-estimator}, with $n$ counting all observations and $\beta=f'(\theta)>0$, satisfies}

\begin{equation}
\sqrt n(\widehat\theta_n-\theta)
 =-\frac1{\beta\sqrt n}\sum_{i=1}^n\varepsilon_i+o_P(1)
 \xrightarrow{d}\mathcal N(0,\sigma^2/\beta^2).\label{eq:weighted-full-linear}
\end{equation}

Moreover, $\widehat\beta_n^{\normalfont{\textsf{LS}}}\to\beta$, $\widehat\sigma_n^2\to\sigma^2$ and $\widehat V_n\to\sigma^2/\beta^2$ in probability. Hence

\begin{equation}
\sqrt{n(\widehat V_n)^{-1}}(\widehat\theta_n-\theta)
 \xrightarrow{d}\mathcal N(0,1).\label{eq:weighted-studentized}
\end{equation}
\end{theorem} 
Theorem~\ref{thm:normality} supports asymptotically valid  inference, with a consistently estimated variance. Although this result can be used to construct a standard Wald-type confidence interval, its main purpose is to facilitate comparison with classical root-finding algorithms such as SA and ASA. As noted in Remark~\ref{remark:interval-stalling}, we recommend using the confidence sequence generated by \textsc{STB} when an additional multiplicative logarithmic factor in the interval length is acceptable.


\begin{proof}[Proof sketch of Theorem~\ref{thm:normality}]
We use the following shorthands: centered samples $\bar e_n=\bar X_n-\theta$ and second order remainders $R_i=R_f(X_{t(i)})$. On the event $\textsf{E}:=\{D_n>0$, $\widehat\beta_n^{\normalfont{\textsf{LS}}}>0\}$, the least squares formula leads to 
\begin{equation}
\sqrt n(\widehat\theta_n-\theta)
=-\frac{\sum_{i=1}^n\varepsilon_i}{\widehat\beta_n^{\normalfont{\textsf{LS}}}\sqrt n}
+\frac{\sqrt n\,\bar e_n(\widehat\beta_n^{\normalfont{\textsf{LS}}}-\beta)}{\widehat\beta_n^{\normalfont{\textsf{LS}}}}
-\frac{\sum_{i=1}^nR_i}{\widehat\beta_n^{\normalfont{\textsf{LS}}}\sqrt n}.
\label{eq:weighted-pooled-identity}\end{equation}
Write this as $\textsf{T}_1+\textsf{T}_2+\textsf{T}_3$, where
\[
\textsf{T}_1=-\frac{\sum_i\varepsilon_i}{\widehat\beta_n^{\normalfont{\textsf{LS}}}\sqrt n},\qquad
\textsf{T}_2=\frac{\sqrt n\,\bar e_n(\widehat\beta_n^{\normalfont{\textsf{LS}}}-\beta)}{\widehat\beta_n^{\normalfont{\textsf{LS}}}},\qquad
\textsf{T}_3=-\frac{\sum_iR_i}{\widehat\beta_n^{\normalfont{\textsf{LS}}}\sqrt n}.
\]
\textbf{Claim:}
Theorem~\ref{thm:RSA} conditionally on $\mathcal F_{t-1}$ to $Y_{j,t}=\mu_t+\varepsilon_{j,t}$, with $K_*=2$, $t^*=K+t-2$. yields $(N_t-\mathfrak n_t)/(2\sqrt{\mathfrak n_t/\mu_t^2})\Rightarrow\mathcal N(0,\sigma^2)$, hence $\mathbb P(N_t<\mathfrak n_t/2\mid\mathcal F_{t-1})=o(1)$ uniformly since $\mu_t^2\mathfrak n_t\to\infty$. Section~\ref{supp:weighted-expansion} verifies the required conditions to invoke Theorem~\ref{thm:RSA}.\par
$\textsf{T}_1$: We first prove $\widehat\beta_n^{\normalfont{\textsf{LS}}}-\beta=o_P(1/\log n)$in Equations~\eqref{eq:weighted-slope-identity}--\eqref{eq:weighted-slope-product} and~\eqref{eq:weighted-K-conditions}. Therefore, we obtain  $1/\widehat\beta_n^{\normalfont{\textsf{LS}}}\to_P1/\beta$. By combining standard CLT and Slutsky lemma, we reach the asymptotic normality: $\textsf{T}_1\Rightarrow\mathcal N(0,\sigma^2/\beta^2)$.\par
$\textsf{T}_2$: By~\eqref{eq:weighted-B}, $\sqrt n\,|\bar e_n|\le n^{-1/2}\sum_i|X_{t(i)}-\theta|=O_P(\log n)$. By the preceding claim and a dyadic partition, we show that  $D_n\asymp_P(\log n)^3$, supporting the slope bound. Hence $\textsf{T}_2=o_P(1)$. \par
$\textsf{T}_3$: We further show that the sum of remainder terms are negligible: $\sum_i|R_i|=O_P((\log n)^3)$, so $\textsf{T}_3=O_P((\log n)^3/\sqrt n)=o_P(1)$ in~\eqref{eq:weighted-remainders}. Combining the terms with variance consistency proves the theorem.
\end{proof}

\subsection{Minimal asymptotic variance and Gaussian efficiency}
\label{subsec:formal-efficiency}
We next compare the variance in Theorem~\ref{thm:normality} with lower bounds for regular estimators. The comparison requires smooth perturbations of the regression function. We impose the following extra condition.

\begin{assumption}
\label{assumption:efficiency}
For a continuous regression function $f$ on $\mathcal I=[a_0,b_0]$ with an interior root $\theta$ and $(x-\theta)f(x)>0$ for $x\ne\theta$, there exist $\beta>0$, $0<r_1<\min\{\theta-a_0,b_0-\theta\}$, and $0\leq L_1<\infty$ such that, with $J=[\theta-r_1,\theta+r_1]$,
\begin{enumerate}
\item\label{condition:efficiency-derivative}
the function $f$ is continuously differentiable on an open neighborhood of $J$, $f'(\theta)=\beta$, and $\beta/2\leq f'(x)\leq3\beta/2$ for every $x\in J$,
\item\label{condition:efficiency-lipschitz}
for every $x,y\in J$, $|f'(x)-f'(y)|\leq L_1|x-y|$.
\end{enumerate}

\end{assumption}

Let $\mathcal G$ be the resulting class of regression functions, with the local constants allowed to depend on $f$. A direction $g$ is continuous on $\mathcal I$ and has a Lipschitz continuous first derivative near $\theta$. An estimator is regular at $f$ if its centered $\sqrt n$ error has the same limit under every fixed finite collection of alternatives $f+n^{-1/2}\sum_jh_jg_j$, for fixed local parameters $h_j$. Competing procedures may use arbitrary sampling machinery.

We first consider an unknown common centered error distribution satisfying Assumption~\ref{ass:weighted-efficient}.\ref{condition:weighted-efficient:noise}. For this comparison, regularity is also required along bounded density tilts of the error {distribution}, recentered to have mean zero, with the corresponding constant change of $f$.

\begin{proposition}\label{prop:weighted-minimum}
Suppose Assumptions~\ref{ass:weighted-efficient} and~\ref{assumption:efficiency} hold. In the preceding  model, the limiting variance of every regular root estimator, when finite, is at least $\sigma^2/\beta^2$. Algorithm~\ref{alg:weighted-section} followed by Algorithm~\ref{alg:weighted-estimator} is regular and attains this bound under~\eqref{eq:weighted-K-conditions}.
\end{proposition}
\begin{remark}
The lower bound follows from bounded scores $s$ of the error {distribution}. Their information is $\mathbb E s(\varepsilon)^2$, and the root derivative is $-\mathbb E[\varepsilon s(\varepsilon)]/\beta$. {Taking the supremum of the resulting variance bounds, we obtain $\sigma^2/\beta^2$.} Section~\ref{supp:weighted-minimum} {verifies that these paths preserve the moment condition} and checks the conditions of the convolution theorem~\citep{bickel1993efficient,hajek1970characterization}.
\end{remark}

We further provide convolution theorem of Le Cam-H\'ajek type~\cite{bickel1993efficient,hajek1970characterization}: the errors are assumed to be i.i.d. $\mathcal N(0,\sigma^2)$ with fixed $\sigma>0$, and regularity is required under the smooth regression perturbations just defined. 

\begin{proposition}\label{prop:weighted-gaussian-efficiency}
For every $f\in\mathcal G$ in this Gaussian model, the sampling rule in Algorithm~\ref{alg:weighted-section}, followed by Algorithm~\ref{alg:weighted-estimator}, is regular under~\eqref{eq:weighted-K-conditions} and has limiting distribution $\mathcal N(0,\sigma^2/\beta^2)$. For any regular estimator under a predictable design, its limiting distribution is
\begin{equation}
\mathcal N(0,\sigma^2/\beta^2)*M\label{eq:weighted-convolution-bound}
\end{equation}

for a probability distribution $M$. Its limiting variance, when finite, is at least $\sigma^2/\beta^2$. Thus the estimator attains the regular Gaussian information bound.
\end{proposition}
\begin{remark}[Information bound]
For the alternatives $f_{n,u}=f-\beta u/\sqrt n$, the root is $\theta+u/\sqrt n+O(n^{-1})$. Predictable sampling kernels cancel from the likelihood ratio, which is exactly

\begin{equation}
\log\frac{d\mathbb P_{n,u}}{d\mathbb P_{n,0}}
 =u\mathcal Z_n-\frac{u^2I_f}{2},\qquad
 \mathcal Z_n=-\frac\beta{\sigma^2\sqrt n}\sum_{i=1}^n\varepsilon_i,
 \quad I_f=\frac{\beta^2}{\sigma^2}.\label{eq:weighted-exact-lan}
\end{equation}

The score is $\mathcal N(0,I_f)$ under every design. Together with regularity, this verifies the hypotheses of H\'ajek's convolution theorem~\citep{bickel1993efficient,hajek1970characterization}. Section~\ref{supp:weighted-gaussian} contains all the necessary technical details.
\end{remark}

\section{From root to level set: a multivariate extension}
\label{sec:multivariate}
To generalize beyond the univariate root-finding problem,
we apply the proposed \textsc{STB} with local polynomial regression to reconstruct the level-set of a function~\cite{tsybakov1997nonparametric,cavalier_nonparametric_1997,willett2007minimax}. At a selected point $\bm{x}\in\overline B(\bm{x}_0,R)\subset\mathbb R^d$, taken to be the domain of $f(\cdot)$, we observe $Y=f(\bm{x})+\varepsilon$, with the noise conditions specified below. The center $\bm{x}_0$ and search radius $R$ are known. We aim to estimate the sublevel set $K_\star=\{\bm{x}\in\overline B(\bm{x}_0,R):f(\bm{x})\leq0\}$ and its boundary $\partial K_\star$, the zero level set. In reliability assessment, $K_\star$ can represent operating conditions under which the mean response remains below a prescribed limit. Under the radial crossing assumption, Assumption~\ref{ass:star}, locating its boundary along each direction is a scalar root-finding problem that can be addressed through sequential noisy measurements.

For a direction $u\in\mathbb S^{d-1}$, consider the ray $\bm{x}_0+ru$, $0\leq r\leq R$. Assumption~\ref{ass:star} posits that the radial function $\rho_\star:\mathbb S^{d-1}\to[r_-,r_+]$, with $0<r_-\leq r_+<R$, assigns to $u$ the unique radius satisfying $f(\bm{x}_0+\rho_\star(u)u)=0$. The intersections $\bm{x}_0+\rho_\star(u)u$ form $\partial K_\star$, as illustrated in Figure~\ref{fig:levelset-rays}. We apply \textsc{STB} along finitely many rays to estimate these radii, then combine the estimates by local polynomial regression to reconstruct the boundary.

\begin{figure}[!t]
\centering
\adjustbox{margin=7pt,width=0.55\linewidth}{\input{fig/levelset_rays_3d_aos_body.tex}}
\caption{Level-set reconstruction in $d=3$.}
\label{fig:levelset-rays}
\end{figure}

\subsection{Radial model and assumptions}
\label{subsec:ray-model}
Let $(\bm X_i,Y_i)_{i\geq1}$ denote the design points and responses in observation order, with $\bm X_i\in\overline B(\bm x_0,R)$ and $Y_i=f(\bm X_i)+\varepsilon_i$. The directions and initial inputs are deterministic, while the next design point $\bm X_i$ is measurable with respect to $\mathcal G_{i-1}$, where $\mathcal G_i=\sigma(\bm x_0,(\bm X_j,Y_j)_{1\leq j\leq i})$. The conditional-noise assumption is $\mathbb E(\varepsilon_i\mid\mathcal G_{i-1})=0$ and $\mathbb E\{\exp(\lambda\varepsilon_i)\mid\mathcal G_{i-1}\}\leq\exp(\lambda^2v^2/2)$ for every $\lambda\in\mathbb R$, with known $v>0$.

To state smoothness of the level-set function, we use $\mathfrak{d}=d-1\geq1$ for the dimension of the directional coordinates, $A=2\sqrt d$, and $D=[-A-2,A+2]^{\mathfrak{d}}$. We use the $2d$ coordinate charts indexed by $a=(j,\tau)$, $j=1,\ldots,d$, $\tau\in\{-1,1\}$,
\begin{equation}\label{eq:ray-charts}
\Phi_{j,\tau}(z)=\frac{(z_1,\ldots,z_{j-1},\tau,z_j,\ldots,z_{\mathfrak{d}})}{\sqrt{1+\|z\|^2}},\qquad z\in\mathbb R^{\mathfrak{d}}.
\end{equation}
These charts parameterize directions on the sphere in a manner that is convenient for the subsequent analysis. 
\begin{assumption}
\label{ass:star}
There are a known center $\bm{x}_0$, a known radius $R$, and constants
$0<r_-\leq r_+<R$.
\begin{enumerate}
\item\label{condition:radial-crossing}
For every $u\in\mathbb S^{d-1}$, the function $g_u(r)=f(\bm{x}_0+ru)$, $0\leq r\leq R$, has a root $\rho_\star(u)\in[r_-,r_+]$ satisfying
\begin{equation}\label{eq:ray-crossing}
 g_u(\rho_\star(u))=0,\qquad
 (r-\rho_\star(u))g_u(r)>0\quad(r\ne\rho_\star(u)).
\end{equation}
\item\label{condition:radial-separation}
For constants $\c>0$ and $\gamma>0$,
\begin{equation}\label{eq:ray-separation}
 |g_u(r)|\geq \c|r-\rho_\star(u)|^\gamma,
 \qquad u\in\mathbb S^{d-1},\quad 0\leq r\leq R.
\end{equation}
\item\label{condition:radial-holder}For each chart $a=(j,\tau)$, define $F_a:\mathbb R^{\mathfrak d}\to[r_-,r_+]$ by $F_a(z)=\rho_\star(\Phi_a(z))$. It is the radius of the unique zero of $r\mapsto f(\bm x_0+r\Phi_a(z))$, expressed in the chart coordinates $z$.
For some $\mathfrak b>0$ and $L<\infty$, let $p=\lceil\mathfrak b\rceil-1$ and $\vartheta=\mathfrak b-p\in(0,1]$. For a multi-index $\iota=(\iota_1,\ldots,\iota_{\mathfrak d})\in\{0,1,2,\ldots\}^{\mathfrak d}$, we write $|\iota|=\sum_{j=1}^{\mathfrak d}\iota_j$ for its total order and $D^\iota F_a(z)=\partial^{|\iota|}F_a(z)/(\partial z_1^{\iota_1}\cdots\partial z_{\mathfrak d}^{\iota_{\mathfrak d}})$, with $D^0F_a=F_a$.
Each $F_a$ is $p$ times continuously differentiable on a neighborhood of $D$ and satisfies, for every chart $a$ and every multi-index,
\begin{equation}\label{eq:ray-holder}
\begin{aligned}
\sup_{z\in D}|D^\iota F_a(z)|&\leq L,\quad \forall|\iota|\leq p,
\end{aligned}
\end{equation}
and $
|D^\iota F_a(z)-D^\iota F_a(z')|\leq L\|z-z'\|_\infty^\vartheta$
for $|\iota|=p,\quad z,z'\in D.$
Thus $|\iota|=k$ means that the partial derivative has total order $k$.
\end{enumerate}
\end{assumption}

Assumption~\ref{ass:star}.\ref{condition:radial-crossing} is appropriate for a threshold region surrounding a known interior point, with one outward crossing in every direction and known bounds on its radius. A bounded convex region containing the center in its interior, as illustrated in Figure~\ref{fig:levelset-rays},  provides one example of this geometry. The sign requirement makes each radial interval a one-dimensional root-finding problem. Assumption~\ref{ass:star}.\ref{condition:radial-separation} quantifies how rapidly the response moves away from zero. A radial derivative uniformly bounded below by a positive constant is sufficient when $\gamma=1$. Similar to Assumption~\ref{assumption:power}, Assumption~\ref{ass:star}.\ref{condition:radial-holder} controls how the boundary radius varies with direction via the standard H\"older smoothness assumption in nonparametric estimation literature~\cite{tsybakov2009introduction} as well as the literature on level-set estimation and active boundary learning~\citep{tsybakov1997nonparametric,locatelli_adaptive_2018}. Condition~\ref{eq:ray-holder} enables reconstruction between rays, separately from the radial signal condition.  

We estimate the sublevel set in the search ball,
\begin{equation}\label{eq:radial-set}
 K_\star=\{\bm{x}\in\overline B(\bm{x}_0,R):f(\bm{x})\leq0\}
 =\{\bm{x}_0+ru:u\in\mathbb S^{d-1},\ 0\leq r\leq\rho_\star(u)\}.
\end{equation}
The assumptions imply that this set is compact, contains $\bm{x}_0$ in its
interior, and has boundary
$\partial K_\star=\{\bm{x}_0+\rho_\star(u)u:u\in\mathbb S^{d-1}\}$.

\subsection{Ray-wise \textsc{STB}}
\label{subsec:ray-algorithm}
We fix a total observation budget $n$ and an error probability $\Delta\in(0,1)$. Generate the directions deterministically as follows. Choose $0<h\leq1$ and the fixed integer $Q$ specified in~\eqref{eq:ray-grid-constants}. For each of the $2d$ charts $a=(j,\tau)$, form every grid point $z\in(h/Q)\mathbb Z^{\mathfrak d}\cap D$. Define

\begin{equation*}
 I_a=\{k\in\{1,\ldots,M\}:a(k)=a\},\qquad
 M=2d\bigl\{2\lfloor(A+2)Q/h\rfloor+1\bigr\}^{\mathfrak d}.
\end{equation*}
We enumerate the pairs $(a,z)$, with $j$ increasing and $\tau=-1$ before $\tau=1$, as $(a(k),z_k)$, $1\leq k\leq M$, and set $u_k=\Phi_{a(k)}(z_k)\in\mathbb S^{d-1}$. 
Thus $I_a$ is the set of observation-batch indices assigned to chart $a$. Directions that coincide across charts receive separate batches and separate indices. Within each chart the coordinate spacing is $h/Q$. Every direction on the sphere is within Euclidean distance $\sqrt{\mathfrak d}\,h/(2Q)$ of a generated direction: choose a chart corresponding to a largest absolute component of $u$, then round each coordinate to the nearest grid point. 


More generally, any $0<h\leq1$ with $M\leq n$ defines the procedure. Each ray receives $m_k=m=\lfloor n/M\rfloor$ observations and uses $K=\lceil\log_2(4M/\Delta)\rceil$.

\begin{algorithm}[!t]
\DontPrintSemicolon
\caption{Ray-wise \textsc{STB}}
\label{alg:ray-stb}
\KwIn{$\bm{x}_0,R,u_1,\ldots,u_M,v$, total budget $n\geq M$, and $\Delta\in(0,1)$.}
$m\gets\lfloor n/M\rfloor$\;
\For{$k=1,\ldots,M$}{
$m_k\gets \lfloor M^{-1}n\rfloor$\;
 Implement \textsc{STB} (Algorithm~\ref{alg:stb}) on $g_{u_k}$ over $[0,R]$, with observation budget $m_k$ and tolerance $\Delta/M$\;
 Report $\widehat\rho_k$ as the estimate~\eqref{eq:uncapped-output}\;
}
\KwOut{$\widehat\rho_1,\ldots,\widehat\rho_M$.}
\end{algorithm}
{The global separation in Assumption~\ref{ass:star}.\ref{condition:radial-separation} guarantees} a common starting budget, independent of $\Delta$. With $\C_\gamma$ as in Theorem~\ref{thm:power}, we take $A_\gamma=(2\gamma\log2)^{-1}+2/\log3$ and $\widetilde{\C}_{1}=R\vee\{\C_\gamma(v/\c)^{1/\gamma}\}$, and define $
m_\gamma=\min\{j\in\mathbb N:j\geq30,\ j\geq4A_\gamma\log(j+1)\}.$
The ray-error constant $\widetilde{\C}_{1}$ depends only on $R$, $\gamma$ and $v/\c$.
{If $m\geq m_\gamma$, we obtain from the scalar argument and a union bound over the $M$ rays}
\begin{equation}\label{eq:ray-uniform}
\begin{aligned}
\max_{1\leq k\leq M}|\widehat\rho_k-\rho_\star(u_k)|
&\leq\epsilon_{m,M,\Delta}:=\widetilde{\C}_{1}\Bigl\{\frac{\log^2m\,[\lceil\log_2(4M/\Delta)\rceil+\log m]}{m}\Bigr\}^{1/(2\gamma)}
\end{aligned}
\end{equation}
with probability at least $1-\Delta$. The proof is in Supplementary Material~\ref{appendix:proof:multivariate}.

\subsection{Radial regression and level-set recovery}
\label{subsec:ray-reconstruction}
We regress the responses $\widehat\rho_k$ on their directional coordinates to estimate the level-set function $\rho_\star$. A local polynomial of degree $p=\lceil\mathfrak b\rceil-1$ uses the smoothness in Assumption~\ref{ass:star}.\ref{condition:radial-holder}. We use $w^\iota=\prod_{j=1}^{\mathfrak d}w_j^{\iota_j}$ and $\iota!=\prod_{j=1}^{\mathfrak d}\iota_j!$, define the collection of monomials $v_p(w)=(w^\iota)_{|\iota|\leq p}$ and $P=\binom{\mathfrak{d}+p}{p}$, and write $e_0$ for the first coordinate vector. For the  \textit{kernel} $W(w)=\prod_{i=1}^{\mathfrak d}(1-|w_i|)_+$, let $\lambda_0=\lambda_{\min}(H_0)>0$, where
\begin{equation}\label{eq:ray-kernel}
H_0=\int_{[-1,1]^{\mathfrak d}}W(w)v_p(w)v_p(w)^\top\,dw.
\end{equation}
We choose the fixed integer $Q\geq1$ and the constant $\widetilde{\C}_{2}$ as follows:
\begin{equation}\label{eq:ray-grid-constants}
Q=\max\!\biggl\{1,\Bigl\lceil\frac{2\cdot4^{\mathfrak{d}} \mathfrak{d}(1+2p)P}{\lambda_0}\Bigr\rceil\biggr\},\qquad
 \widetilde{\C}_{2}=2d\{2(A+2)Q+1\}^{\mathfrak{d}}.
\end{equation}
By the chart grids and index sets $I_a$ generated in Section~\ref{subsec:ray-algorithm}, their total size satisfies $M\asymp h^{-\mathfrak d}$ and $M\leq\widetilde{\C}_{2}h^{-\mathfrak d}$. For $z\in[-A,A]^{\mathfrak{d}}$, we put $w_k(z)=(z_k-z)/h$ and define the local polynomial regression by
\begin{equation}\label{eq:ray-regression}
\tag{LP}
\begin{aligned}
H_a(z)&=Q^{-\mathfrak{d}}\sum_{k\in I_a}W(w_k(z))v_p(w_k(z))v_p(w_k(z))^\top,\\
\widehat F_a(z)&=e_0^\top H_a(z)^{-1}Q^{-\mathfrak{d}}\sum_{k\in I_a}W(w_k(z))v_p(w_k(z))\widehat\rho_k.
\end{aligned}
\end{equation}
The grid makes every $H_a(z)$ invertible uniformly in $z$ and $h$. Given a direction $u$, for $a=(j,\tau)$ and $\tau u_j>0$, we write $z_a(u)=(u_i/|u_j|)_{i\ne j}$ for the directional coordinates used to evaluate the chart fit, so $\Phi_a(z_a(u))=u$. To combine the fits continuously, we set $\kappa_d=1/(2\sqrt d)$ and
\begin{equation}\label{eq:ray-chart-weights}
\psi_{j,\tau}(u)=\frac{(\tau u_j-\kappa_d)_+}{\sum_{\ell=1}^d\sum_{\nu\in\{-1,1\}}(\nu u_\ell-\kappa_d)_+},\qquad
\widehat\rho_h(u)=\Pi_{[0,R]}\!\bigl\{\sum_a\psi_a(u)\widehat F_a(z_a(u))\bigr\}.
\end{equation}
The denominator is at least $\kappa_d$, the weights sum to one, and $\psi_a(u)>0$ implies $\|z_a(u)\|_\infty<A$. A term with $\psi_a(u)=0$ is defined as zero. The final projection makes sure that the continuous fitted radius lies in the search range. Finally, we estimate the sublevel set and level-set by
\begin{align}
\widehat K_{M,h}&=\{\bm{x}_0+ru:u\in\mathbb S^{d-1},\ 0\leq r\leq\widehat\rho_h(u)\},\label{eq:ray-set-estimator}\\
\widehat\Gamma_{M,h}&=\{\bm{x}_0+\widehat\rho_h(u)u:u\in\mathbb S^{d-1}\}=\partial\widehat K_{M,h}.\label{eq:ray-boundary-estimator}
\end{align}

For nonempty compact sets $K_1,K_2\subset\mathbb R^d$, their Hausdorff distance is
\begin{equation*}
d_H(K_1,K_2)=
 \sup_{\bm{x}\in K_1}\inf_{\bm{y}\in K_2}\|\bm{x}-\bm{y}\|_2
 \;\vee\;
 \sup_{\bm{y}\in K_2}\inf_{\bm{x}\in K_1}\|\bm{x}-\bm{y}\|_2,
\end{equation*}
where $\|\cdot\|_2$ is the Euclidean norm. We use this distance to compare $\widehat K_{M,h}$ with $K_\star$ and $\widehat\Gamma_{M,h}$ with $\partial K_\star$.

\begin{theorem}\label{thm:ray-reconstruction}
Suppose Assumption~\ref{ass:star} hold for a fixed $\mathfrak b>0$. Consider the estimator is obtained from Algorithm~\ref{alg:ray-stb} and the local polynomial construction above, with $\widetilde{\C}_{3}=2\cdot3^{\mathfrak{d}}\sqrt P/\lambda_0$ and $\widetilde{\C}_{4}=\widetilde{\C}_{3}\mathfrak{d}^p/p!.$
If $m\geq m_\gamma$, then, with probability at least $1-\Delta$,
\begin{equation}\label{eq:ray-hausdorff}
\begin{aligned}
d_H(\widehat K_{M,h},K_\star)\vee d_H(\widehat\Gamma_{M,h},\partial K_\star)
&\leq\|\widehat\rho_h-\rho_\star\|_\infty\leq\widetilde{\C}_{3}\epsilon_{m,M,\Delta}+\widetilde{\C}_{4}Lh^{\mathfrak b}.
\end{aligned}
\end{equation}

In particular, if we set the optimal bandwidth
\begin{equation}\label{eq:ray-bandwidth}
\ell_{n,\Delta}=\log^2(n+1)\log(4n/\Delta),\qquad
 h=\Bigl(\frac{\ell_{n,\Delta}}n\Bigr)^{1/(2\gamma\mathfrak b+d-1)},
\end{equation}
then for every $n$ and $0<\Delta<1$ satisfying $
\ell_{n,\Delta}\leq n, nh^{d-1}\geq2\widetilde{\C}_{2}m_\gamma,$
all three errors in~\eqref{eq:ray-hausdorff} are, with probability at least $1-\Delta$, simultaneously bounded by
\begin{equation}\label{eq:ray-global-rate}
\begin{aligned}
\bigl[\widetilde{\C}_{4}L+\widetilde{\C}_{3}\widetilde{\C}_{1}\{2\widetilde{\C}_{2}(2+1/\log2)\}^{1/(2\gamma)}\bigr] \Bigl\{\frac{\log^2(n+1)\log(4n/\Delta)}n\Bigr\}^{\mathfrak b/(2\gamma\mathfrak b+d-1)}.
\end{aligned}
\end{equation}
The preceding constants and $m_\gamma$ are independent of $\Delta$ and it implies $\widetilde O_{P}(n^{-\mathfrak b/(2\gamma\mathfrak b+d-1)})$ for the Hausdorff distances.
\end{theorem}
\begin{remark}

For comparison, \cite{tsybakov1997nonparametric} obtains the optimal rate in Hausdorff distance of the order $(n/\log n)^{-\mathfrak b/\{(2\gamma+1)\mathfrak b+d-1\}}$
for \textit{i.i.d.} observations from a bounded density with a H\"older radial boundary, with $\mathfrak b$ and $\gamma$ denoting the corresponding smoothness and separation powers. Our fast rate stems from the flexibility to choose the covariate values at which we query the response. The construction in~\cite{castro_minimax_2008} combines one-dimensional active estimates with polynomial interpolation for binary regression, with level-set smoothness at least one and a two-sided margin condition. Their results focus on expected excess classification risk. Instead, our bound~\eqref{eq:ray-hausdorff} controls uniform radial error and Hausdorff distance with high probability.

\end{remark}

\section{Concluding remarks}
\label{section:concluding}
In the present work, we systematically analyze the algorithmic and theoretical aspects of estimation-via-testing in the noisy root-finding problem and develop \textsc{STB}, which uses the evidence at each location to decide when to move and which interval to examine next. 

Under sub-Gaussian noise, for regression function satisfying $|f(x)|\gtrsim |x-\theta|^\gamma$, \textsc{STB} leads to adaptive minimax optimal convergence rate $\widetilde{O}(n^{-1/(2\gamma)})$ and time-uniform root coverage. For positive-slope roots, aggregation over all observations supports asymptotically valid normal inference under Lyapunov noise condition, with efficiency formalized. The multivariate extension links radial estimation to level-set approximation through the two error terms in~\eqref{eq:ray-hausdorff}. Finer resolution uses more directions and leaves fewer observations per ray. Balancing these terms guides the allocation when the signal and smoothness exponents are specified and motivates adaptation to unknown exponents.

{A natural next step is to extend the fixed-budget limit in Theorem~\ref{thm:normality} to a continuing procedure with the finite-sample coverage of Proposition~\ref{prop:type1-error} and interval width governed by $\sigma^2/\beta^2$, up to time-uniform logarithmic factors. Confidence sequences~\citep{howard_time-uniform_2021} and asymptotic time-uniform inference for averaged SA~\citep{xie_asymptotic_2024} offer tools for this extension.}


Another direction is adaptation of the radial reconstruction to unknown $\gamma$ and $\mathfrak b$. Joint adaptation is available for binary response models~\citep[Theorem~1]{locatelli_adaptive_2018}. Here the aim is to select directions, polynomial degree and bandwidth under conditional sub-Gaussian errors while attaining the high-probability radial and Hausdorff bounds of Theorem~\ref{thm:ray-reconstruction}.

\bibliographystyle{plainnat-no-url}
\bibliography{YY,arxiv_missing_references}

\newpage
\appendix
\section*{Appendix}

The supplementary material contain all the proof arguments and simulations. Section~\ref{section: auxiliary proof} {controls stage lengths $|\mathcal{I}(n)|$ and sign errors and establishes the martingale-noise extension}. These bounds support the uniform limits for randomly stopped averages in Section~\ref{appendix:proof:thm:RSA}. Sections~\ref{SM:proof:thm:minimax}--\ref{appendix:proof:thm:type1} establish the adaptive lower bound, the jump-crossing rate and simultaneous interval coverage, respectively. Section~\ref{appendix:proof:multivariate} combines the raywise error bounds with local polynomial approximation to prove the rate of convergence for level-set estimation. Section~\ref{SM:pj-rates} derives some PJ upper and lower bounds and compares their polynomial exponents with the STB and SA results. {Section~\ref{supp:weighted-inference} proves the asymptotic expansion and efficiency result.} Section~\ref{section:simulation} contains all simulation results and experimental details.
\input{appendix_new/power_proof}

\input{appendix_new/auxiliary_proof}
\input{appendix_new/CLT_proof}
\begingroup
\emergencystretch=2em
\input{appendix_new/minimax_proof}
\endgroup
\input{appendix_new/jump_proof}
\input{appendix_new/anytime_validity_proof}
\begingroup
\emergencystretch=2em
\input{appendix_new/multivariate_proof}
\endgroup
\input{appendix_new/pj_proof}

\input{appendix_new/weighted_inference_proof}
\endgroup 

\input{section/Simulation_Study}

\end{document}

%% file: _macros.tex
\usepackage{comment,url,graphicx,subcaption,relsize,bm,bbm}
\usepackage{amssymb,amsfonts,amsmath,amsthm,amscd,dsfont,mathrsfs,mathtools,multirow,microtype,nicefrac,pifont}
\usepackage{float,psfrag,epsfig,color,url,hyperref}
\usepackage{upgreek}
\usepackage[dvipsnames]{xcolor}
\usepackage{epstopdf,bbm,mathtools,enumitem}
\usepackage[toc,page]{appendix}
\usepackage[mathscr]{euscript}

\usepackage[top=1in, bottom=1.2in, left=1.2in, right=1.2in]{geometry}

\def\balign#1\ealign{\begin{align}#1\end{align}}
\def\baligns#1\ealigns{\begin{align*}#1\end{align*}}
\def\balignat#1\ealign{\begin{alignat}#1\end{alignat}}
\def\balignats#1\ealigns{\begin{alignat*}#1\end{alignat*}}
\def\bitemize#1\eitemize{\begin{itemize}#1\end{itemize}}
\def\benumerate#1\eenumerate{\begin{enumerate}#1\end{enumerate}}

\newenvironment{talign*}
 {\let\displaystyle\textstyle\csname align*\endcsname}
 {\endalign}
\newenvironment{talign}
 {\let\displaystyle\textstyle\csname align\endcsname}
 {\endalign}

\def\balignst#1\ealignst{\begin{talign*}#1\end{talign*}}
\def\balignt#1\ealignt{\begin{talign}#1\end{talign}}

\let\originalleft\left
\let\originalright\right
\renewcommand{\left}{\mathopen{}\mathclose\bgroup\originalleft}
\renewcommand{\right}{\aftergroup\egroup\originalright}

\def\tinycitep*#1{{\tiny\citep*{#1}}}
\def\tinycitealt*#1{{\tiny\citealt*{#1}}}
\def\tinycite*#1{{\tiny\cite*{#1}}}
\def\smallcitep*#1{{\scriptsize\citep*{#1}}}
\def\smallcitealt*#1{{\scriptsize\citealt*{#1}}}
\def\smallcite*#1{{\scriptsize\cite*{#1}}}

\def\<{\left\langle} 
\def\>{\right\rangle}

\DeclareSymbolFont{rsfs}{U}{rsfs}{m}{n}
\DeclareSymbolFontAlphabet{\mathscrsfs}{rsfs}
\ifdefined\nonewproofenvironments\else
\ifdefined\ispres\else
\newtheorem{theorem}{Theorem}
\newtheorem{lemma}[theorem]{Lemma}
\newtheorem{remark}[theorem]{Remark}
\newtheorem{corollary}[theorem]{Corollary}
\newtheorem{definition}[theorem]{Definition}

\renewenvironment{proof}[1][Proof]{%
  \par\noindent\textbf{#1.}\hspace{0.3em}%
}{%
  \hfill\qed\par\vspace{0.1in}%
}
\newenvironment{proof-sketch}{\noindent\textbf{Proof Sketch}
  \hspace*{1em}}{\qed\bigskip\\}
\newenvironment{proof-idea}{\noindent\textbf{Proof Idea}
  \hspace*{1em}}{\qed\bigskip\\}
\newenvironment{proof-of-lemma}[1][{}]{\noindent\textbf{Proof of Lemma {#1}}
  \hspace*{1em}}{\qed\\}
  \newenvironment{proof-of-proposition}[1][{}]{\noindent\textbf{Proof of Proposition {#1}}
  \hspace*{1em}}{\qed\\}
\newenvironment{proof-of-theorem}[1][{}]{\noindent\textbf{Proof of Theorem {#1}}
  \hspace*{1em}}{\qed\\}
\newenvironment{proof-attempt}{\noindent\textbf{Proof Attempt}
  \hspace*{1em}}{\qed\bigskip\\}

\newtheorem*{remark*}{Remark}

\fi

\newtheorem{proposition}[theorem]{Proposition}

\newtheorem{assumption}{Assumption}
\theoremstyle{definition}
\newtheorem{example}[theorem]{Example}

\fi
\makeatletter
\@addtoreset{equation}{section}
\makeatother

\hypersetup{
  colorlinks,
  linkcolor={blue},
  citecolor={blue},
  urlcolor={blue}
}



%% file: appendix_new/power_proof.tex
\section{Local signal rate}
\label{appendix:proof:thm:power}

We present a simple proof of Theorem~\ref{thm:power} here.

\begin{proof}[Proof of Theorem~\ref{thm:power}]
\label{proof:thm:scalar-power}
We write $A_t$ for the event that stage $t$ is reached. On the event $A_t$, $X_t$ is fixed under this conditional distribution $\bbP(\cdot \mid \mathcal{F}_{t-1})$. For probability calculations, we truncate and write $W_{j,t}=\sum_{i=1}^{j\wedge N_t}\varepsilon_{i,t}$. Section~\ref{subsec:martingale-noise} justifies this construction under conditional moment generating function bounds. Iterating the conditional moment generating function bound in Remark~\ref{remark:martingale-noise} therefore {implies} $\mathbb{E}(e^{\lambda W_{j,t}}\mid\mathcal F_{t-1})\leq e^{j\lambda^2v^2/2}$ for the sum $W_{j,t}$. At a random stage start, the same calculation follows by conditioning on its possible starting times. This covers adaptive locations and conditionally sub-Gaussian errors with varying conditional variances. The (conditional) Chernoff bound proves, for either sign and every $x>0$,
\begin{equation*}
\mathbb P(\pm W_{j,t}>x\mid\mathcal F_{t-1})\leq\inf_{\lambda>0}\exp\{-\lambda x+j\lambda^2v^2/2\}=e^{-x^2/(2jv^2)}.
\end{equation*}
This is the calculation covered in Section~\ref{subsec:stage}. {Substituting $x=T(j,\alpha_t)$ and summing over both signs and all $j\geq1$, we obtain}
\begin{equation*}
\mathbb P\!\bigl(A_t\cap\{\exists j\geq1:|W_{j,t}|>T(j,\alpha_t)\}\bigr)\leq\mathbb P(A_t)\,2\sum_{j=1}^\infty(j+1)^{-(K+t)}.
\end{equation*}
We apply the tower property to each reached stage and then a countable union bound. Thus the event
\begin{equation*}
 E=\Big\{|W_{m,t}|
 \leq v\sqrt{2m(K+t)\log(m+1)}
 \text{ for every reached }t\text{ and every }m\geq1\Big\}
\end{equation*}
satisfies, by~\eqref{eq:stage-error-sum},
\begin{equation*}
\mathbb P(E^c)\leq\sum_{t=1}^\infty2\sum_{j=1}^\infty(j+1)^{-(K+t)}\leq2^{1-K}(1+2/K)\leq2^{2-K}\leq\Delta.
\end{equation*}
The conditional bounds are valid at adaptive locations, and their summability covers every stage and every observation budget.
\par On the event $E$, a positive crossing satisfies $N_tf(X_t)=S_{N_t,t}-W_{N_t,t}>0$, and a negative crossing satisfies $N_tf(X_t)<0$. Assumption~\ref{assumption:regularity}.\ref{condition:sign} {determines} the correct half of the interval. If $X_t=\theta$, both closed halves contain $\theta$. Induction {shows that} $\theta\in\mathcal I(n)$ for every $n$, including during an unfinished stage. After $j$ completed stages its width is $R2^{-j}$.

If a stage has received $m\geq2$ observations, its test had not crossed the boundary after the first $m-1$ observations. This remains true when the $m$th observation completes the test. The stopping rule and the event $E$ therefore {imply} the two separate inequalities
\begin{equation*}
|S_{m-1,t}|\leq T(m-1,\alpha_t),\qquad|W_{m-1,t}|\leq T(m-1,\alpha_t).
\end{equation*}
Writing $\mu=f(X_t)$, we subtract the noise from the observed sum to obtain $(m-1)|\mu|\leq|S_{m-1,t}|+|W_{m-1,t}|\leq2T(m-1,\alpha_t)$. Hence
\begin{equation}\label{eq:uncapped-signal-bound}
 |\mu|\leq2v\sqrt{\frac{2(K+t)\log m}{m-1}}.
\end{equation}

For a fixed $n$, we suppress the subscription $n$ and let $T=T_n$ as in the theorem. If at least $T$ stages have completed, the projection in \eqref{eq:uncapped-output} lies in $\mathcal I(n)$ and has error at most
$R2^{-T}\leq R(n+1)^{-1/(2\gamma)}$. 

Otherwise, at most $T$ stages
have received observations. Their observed counts $m_t(n)$ sum to $n$, so the stage $t_n$ selected in \eqref{eq:uncapped-output} has $m=m_{t_n}(n)\geq n/T$. For $n\geq4T$, it has $m\geq2$ and
$m-1\geq n/(2T)$. Equation~\eqref{eq:uncapped-signal-bound} implies
\begin{equation*}
 |f(X_{t_n})|
 \leq4v\sqrt{\frac{T(K+T)\log(n+1)}{n}}.
\end{equation*}
The last display is $|f(X_{t_n})|\leq B_n(v)$. If $B_n(v)<m_\eta$, then $|X_{t_n}-\theta|<\eta$.\par  We apply the local power bound to obtain
\begin{equation*}
|\widehat\theta_n-\theta|\leq|X_{t_n}-\theta|\leq\{B_n(v)/\c_0\}^{1/\gamma}.
\end{equation*}
For an explicit starting condition, we set $\normalfont{\textsf{S}}_{\gamma}=(2\gamma\log2)^{-1}+2/\log3$, so $T_n\leq \normalfont{\textsf{S}}_{\gamma}\log(n+1)$ for $n\geq2$. We may take
\begin{equation}\label{eq:power-start-explicit}
\begin{aligned}\C_\gamma&=\{4\sqrt{8\normalfont{\textsf{S}}_{\gamma}(1\vee \normalfont{\textsf{S}}_{\gamma})}\}^{1/\gamma},\\
n_0&=\min\bigg\{n\geq 100:\ n\geq4\normalfont{\textsf{S}}_{\gamma}\log(n+1),\\
&\quad \frac{16v^2 \normalfont{\textsf{S}}_{\gamma}(1\vee \normalfont{\textsf{S}}_{\gamma})\log^2(n+1)\{K+\log(n+1)\}}n<m_\eta^2\bigg\}.
\end{aligned}
\end{equation}
They imply $n\geq4T_n$ and $B_n(v)<m_\eta$. Using $\log(n+1)\leq2\log n$ {establishes}~\eqref{eq:scalar-power-rate}. All arguments hold on the same event $E$, proving simultaneity in $n$.

\label{proof-end:thm:scalar-power}
\end{proof}

%% file: appendix_new/auxiliary_proof.tex
\section{Auxiliary theoretical results}
\label{section: auxiliary proof}

All probabilities in the one-stage arguments are conditional on the
information $\mathcal F_{t-1}$ available before the current stage.
The predictable mean $\mu_t$ is consequently fixed during that stage.
Until Section~\ref{subsec:martingale-noise}, we use the common centered i.i.d.\ error {distribution} specified in Section~\ref{subsection:RSA}, {with variance $\sigma^2$ and the finite moment~\eqref{eq:weighted-noise-moment}. The scale $v>0$ is fixed.} Following Section~\ref{subsection:RSA}, we use
\begin{align}
    b_t=-2v^2\log\alpha_t=2v^2(K+t),\qquad
 \delta_t=b_t/\mu_t^2,\qquad
 T(j,\alpha_t)=\sqrt{b_tj\log(j+1)}.
\end{align}

Constants and error bounds in the asymptotic statements below are
uniform over $0<|\mu_t|\leq {\C_{\star}}$. 
For a fixed stage, we write $W_j=\sum_{i=1}^j\varepsilon_{it}$ and suppress conditioning. Let $p=2+\min\{\delta,1\}\in(2,3]$. The finite moment assumption implies
\begin{equation}\label{eq:noise-maximal-moment}
 \mathbb E\max_{1\leq j\leq m}|W_j|^p\leq\C m^{p/2},\qquad m\geq1.
\end{equation}
For completeness, Taylor's inequality {implies}
$|x+y|^p\leq |x|^p+p|x|^{p-2}xy+\C_p\{|x|^{p-2}y^2+|y|^p\}$.
After conditioning on the first $m-1$ errors, the linear term vanishes. H\"older's inequality bounds the remaining terms, so $a_m=\mathbb E|W_m|^p$ satisfies
$a_m\leq a_{m-1}+\C a_{m-1}^{(p-2)/p}+\C$.
Induction {shows that} $a_m\leq\C' m^{p/2}$: a sufficiently large $\C'$ makes the increment of $\C'm^{p/2}$ dominate both added terms. Doob's $L^p$ maximal inequality then proves~\eqref{eq:noise-maximal-moment}.

For $2^r\leq j<2^{r+1}$, the denominator $\sqrt{j\log(j+1)}$ is at least a fixed positive constant times $2^{r/2}\sqrt{r+1}$. Markov's inequality on each block and summation imply
\begin{equation}\label{eq:noise-moment-envelope}
 \mathbb P\{\exists j\geq1:|W_j|>a\sqrt{j\log(j+1)}\}
 \leq\C a^{-p}\sum_{r\geq0}(r+1)^{-p/2}
 \leq\C' a^{-p},\qquad a>0.
\end{equation}
These deterministic constants depend on the fixed error {distribution} and $p$, and the same bounds hold conditionally at any stage. Where an exponential finite-sample bound is stated below, we additionally impose~\eqref{eq:noise-mgf}{. Under this additional condition, Chernoff's inequality implies}
\begin{equation*}
 \mathbb P(|W_j|\geq x\mid\mathcal F_{t-1})\leq2\exp\{-x^2/(2jv^2)\}.
\end{equation*}

\begin{proposition}
\label{prop:stage-finiteness}
For a fixed stage $t$, the following statements hold conditionally on the information before that stage.
\begin{enumerate}
    \item If $\mu_t\neq0$ and $T(j,\alpha_t)=o(j)$ as $j\to\infty$, then
$N_t<\infty$ almost surely. 
\item If 
 {$\mu_t=0$ and the errors also satisfy~\eqref{eq:noise-mgf}, then, for the stated threshold,}
\begin{align*}
 \mathbb P(N_t=\infty\mid\mathcal F_{t-1})
 \geq1-2\sum_{j=1}^{\infty}(j+1)^{-(K+t)}.
\end{align*}
\end{enumerate}
\end{proposition}
\begin{proof}[Proof of Proposition~\ref{prop:stage-finiteness}]
\label{proof:prop:stage-finiteness}
The errors have finite first moment, since their second moment is
$\sigma^2<\infty$. Conditional on the past they are independent and
identically distributed with mean zero. The {almost-sure convergence theorem for sample averages} \citep[Theorem~2.4.1]{durrett2019probability} therefore {implies} $W_j/j\to0$ almost surely under the
conditional {distribution}. 

If $\mu_t\neq0$, there is an almost surely finite
integer after which $|S_{jt}|/j=|\mu_t+W_j/j|>|\mu_t|/2$.
There is also a deterministic integer after which
$T(j,\alpha_t)/j<|\mu_t|/2$. At every integer beyond the larger
of these two integers the stopping inequality holds. Its first
occurrence is consequently finite. For the displayed threshold,
$T(j,\alpha_t)/j=\{b_t\log(j+1)/j\}^{1/2}\to0$ at each fixed stage.

For the second statement, we additionally use~\eqref{eq:noise-mgf}. If $\mu_t=0$, then $S_{jt}=W_j$. The event $N_t<\infty$ is the
countable union of the events $|W_j|>T(j,\alpha_t)$. Hence
\begin{align*}
 \mathbb P(N_t<\infty\mid\mathcal F_{t-1})
 &\leq\sum_{j=1}^{\infty}
      2\exp\{-T(j,\alpha_t)^2/(2jv^2)\}\\
 &=2\sum_{j=1}^{\infty}(j+1)^{-(K+t)}.
\end{align*}
Taking complements proves the argument. 
\label{proof-end:prop:stage-finiteness}
\end{proof}

The preceding proposition concerns one conditionally fixed stage. If all
selected means are nonzero almost surely, the proposition does imply
completion of all stages by a countable intersection. 

\subsection{Conditional expected sample size}
\begin{lemma}
\label{lemma:stopping-localization}
Under the common error-{distribution} condition in Section~\ref{subsection:RSA}, put $p=2+\min\{\delta,1\}$ and let $\mathfrak n_t$ be defined by~\eqref{eq:RSA-centering}. For each $0<\epsilon<1$, there are deterministic $\C_\epsilon$ and $t_0(\epsilon)$, depending on $\C_\star,v,K,\epsilon$ and the fixed error {distribution}, such that, for $t\geq t_0(\epsilon)$,
\begin{align}
 \mathbb P\{N_t\leq(1-\epsilon)\mathfrak n_t\mid\mathcal F_{t-1}\}
 &\leq\C_\epsilon(K+t)^{-p/2},\label{eq:stopping-left-tail}\\
 \mathbb P(N_t>j\mid\mathcal F_{t-1})
 &\leq\frac{\C_\epsilon}{|\mu_t|^p j^{p/2}},
 \quad j\in\mathbb N,\ j\geq(1+\epsilon)\mathfrak n_t.
 \label{eq:stopping-right-tail}
\end{align}
Uniformly over $0<|\mu_t|\leq\C_\star$,
\begin{equation}
 \mathbb E\bigl[|N_t/\mathfrak n_t-1|\mid\mathcal F_{t-1}\bigr]\longrightarrow0.
 \label{eq:stopping-relative-L1}
\end{equation}
\end{lemma}
\begin{proof}[Proof of Lemma~\ref{lemma:stopping-localization}]
\label{proof:lemma:stopping-localization}

We write $r=\mathfrak n_t$, $b=b_t$ and $\mu=\mu_t$ for the sake of convenience. The function $g(x)=x/\log(x+1)$ increases strictly from $1$ to infinity, because
$g'(x)=\{\log(x+1)-x/(x+1)\}/\{\log(x+1)\}^2>0$.
Since $\delta_t\geq2v^2(K+t)/\C_\star^2\to\infty$, the equation $g(r)=\delta_t$ has a unique solution with $r>\delta_t$ eventually. In particular $\inf_\mu r\to\infty$ and $\mu^2r=b\log(r+1)$.

For fixed $\epsilon\in(0,1)$, monotonicity of $g$ and
$\log(cr+1)/\log(r+1)\to1$ for each fixed $c>0$ yield, uniformly for large $r$,
\begin{align*}
 |\mu|j&\leq(1-\epsilon/8)T(j,\alpha_t),&&1\leq j\leq(1-\epsilon)r,\\
 T(j,\alpha_t)&\leq(1-\epsilon/8)|\mu|j,&&j\geq(1+\epsilon)r.
\end{align*}
An early crossing therefore requires $|W_j|>(\epsilon/8)\sqrt{bj\log(j+1)}$ for some $j$. Equation~\eqref{eq:noise-moment-envelope} proves~\eqref{eq:stopping-left-tail}. If $N_t>j\geq(1+\epsilon)r$, then $|\mu j+W_j|\leq T(j,\alpha_t)$, so $|W_j|\geq(\epsilon/8)|\mu|j$. Equation~\eqref{eq:noise-maximal-moment} and Markov's inequality prove~\eqref{eq:stopping-right-tail}, for either sign of $\mu$.

To justify expectations, put $m=\lceil(1+\epsilon)r\rceil$. Because $p/2>1$, the late tail is summable and
\begin{equation*}
 \frac1r\sum_{j\geq m}\mathbb P(N_t>j\mid\mathcal F_{t-1})
 \leq\frac{\C_\epsilon}{|\mu|^p r}\,m^{1-p/2}
 \leq\frac{\C_\epsilon'}{(\mu^2r)^{p/2}}\longrightarrow0
\end{equation*}
uniformly in the allowed means. The tail-sum identity implies
$\mathbb E[(N_t/r-1)_+\mid\mathcal F_{t-1}]\leq\epsilon+1/r+o(1)$.

Also note that
$\mathbb E[(1-N_t/r)_+\mid\mathcal F_{t-1}]\leq\epsilon+
\mathbb P\{N_t<(1-\epsilon)r\mid\mathcal F_{t-1}\}$.

Letting $t\to\infty$ and then $\epsilon\downarrow0$ proves~\eqref{eq:stopping-relative-L1}. Markov's inequality also proves relative convergence in conditional probability.
\label{proof-end:lemma:stopping-localization}
\end{proof}

\begin{proof}[Proof of Proposition~\ref{thm:EN}]
\label{proof:thm:EN}
We first compare the exact centering with the explicit normalization.
For all sufficiently large $\delta$, we use the solution $r$ of
$r=\delta\log(r+1)$. The preceding proof {shows that} $r>\delta$, so
\begin{equation*}
 r=\delta\log(r+1)\geq\delta\log\delta.
\end{equation*}
For large $\delta>0$, $2\delta\log\delta+1\leq\delta^2$.
Consequently
\begin{equation*}
 \frac{2\delta\log\delta}
      {\log(2\delta\log\delta+1)}\geq\delta.
\end{equation*}
The strict increase of $x/\log(x+1)$ therefore implies
$r\leq2\delta\log\delta$. {It follows from these bounds that}
\begin{equation*}
 \frac{\log(\delta\log\delta+1)}{\log(\delta+1)}
 \leq\frac{\log(r+1)}{\log(\delta+1)}
 \leq\frac{\log(2\delta\log\delta+1)}{\log(\delta+1)}.
\end{equation*}
Both outer ratios tend to one: their numerators are
$\log\delta+\log\log\delta+O(1)$ and their denominator is
$\log\delta+o(1)$. Since $r=\delta\log(r+1)$, it follows that
\begin{equation}
 \frac{\mathfrak{n}_t}{\delta_t\log(\delta_t+1)}\longrightarrow1.
 \label{eq:stopping-center-order}
\end{equation}
This convergence is uniform because the deterministic lower bound
$\delta_t\geq2v^2(K+t)/{\C_{\star}}^2$ tends to infinity.

Factor the desired ratio as
\begin{equation*}
 \frac{\mathbb{E}[N_t\mid\mathcal F_{t-1}]}
      {\delta_t\log(\delta_t+1)}
 =\frac{\mathbb{E}[N_t\mid\mathcal F_{t-1}]}{\mathfrak{n}_t}
  \frac{\mathfrak{n}_t}{\delta_t\log(\delta_t+1)}.
\end{equation*}
The first factor converges uniformly to one by
Lemma~\ref{lemma:stopping-localization}, and the second does so by
\eqref{eq:stopping-center-order}. For example, after increasing a
deterministic $t_0$, both factors lie between $1/2$ and $2$ for
every mean satisfying $0<|\mu_t|\leq\C_\star$. Their product lies between $1/4$ and $4$,
which proves~\eqref{eq:N_order} with $\c=1/4$, $\C=4$ and a
{$t_0$ depending on $\C_\star,v,K$.} Their product converges to one,
proving~\eqref{eq:ratio-1}. \par Because the uniform error bounds are
deterministic, the conclusion holds almost surely upon substituting
any predictable sequence of means $\mu_t$ satisfying $0<|\mu_t|\leq\C_\star$ almost surely. We remark that the preceding result holds without an assumption
that the sequence of means itself converges.
\label{proof-end:thm:EN}
\end{proof}

\subsection{Sign errors and convergence across stages}
\begin{lemma}[Probability of a false sign]
\label{lemma:false sign}
For a completed nonzero-mean stage under the conditions in Section~\ref{subsection:RSA}, with $p=2+\min\{\delta,1\}$,
\begin{equation}\label{eq:false-sign-moment}
 \mathbb P(M_t\mu_t<0\mid\mathcal F_{t-1})\leq\C(K+t)^{-p/2}.
\end{equation}
If the errors additionally satisfy the sub-Gaussian bound~\eqref{eq:noise-mgf}, then
\begin{equation}
 \mathbb P(M_t\mu_t<0\mid\mathcal F_{t-1})
 \leq\sum_{j=1}^\infty(j+1)^{-(K+t)}
 \leq2^{-(K+t)}\bigl(1+2/(K+t-1)\bigr).
 \label{eq:false sign1}
\end{equation}
\end{lemma}
\begin{proof}[Proof of Lemma~\ref{lemma:false sign}]
\label{proof:lemma:false sign}
The stage completes almost surely by Proposition~\ref{prop:stage-finiteness}. A wrong crossing requires $W_j<-T(j,\alpha_t)$ when $\mu_t>0$, or $W_j>T(j,\alpha_t)$ when $\mu_t<0$, for some $j$. Equation~\eqref{eq:noise-moment-envelope} proves~\eqref{eq:false-sign-moment}. Under~\eqref{eq:noise-mgf}, the one-sided Chernoff bound and a union bound prove the first inequality in~\eqref{eq:false sign1}. {For $a>1$, we obtain by integrating the decreasing function $x^{-a}$}
\begin{equation}\label{eq:power-series-bound}
 \sum_{j=1}^{\infty}(j+1)^{-a}\leq2^{-a}+\int_2^\infty x^{-a}\,dx
 =2^{-a}\bigl(1+2/(a-1)\bigr),
\end{equation}
which proves the second inequality with $a=K+t$.
\label{proof-end:lemma:false sign}
\end{proof}

\begin{lemma}
\label{lemma:almost surely}
Under the conditions of Theorem~\ref{thm:RSA}, with
$\alpha_t=e^{-(K+t)}$, almost surely,
\begin{equation*}
 N_t/\mathfrak{n}_t\longrightarrow1,\qquad
 N_t\longrightarrow\infty,\qquad M_t-\mu_t\longrightarrow0.
\end{equation*}
Only finitely many completed-stage sign errors occur almost surely.
\end{lemma}

\begin{proof}[Proof of Lemma~\ref{lemma:almost surely}]
\label{proof:lemma:almost surely}
Fix $\epsilon\in(0,1)$. Lemma~\ref{lemma:stopping-localization}, with $\epsilon/2$ to allow integer rounding, implies
\begin{equation*}
 \mathbb P\{|N_t/\mathfrak n_t-1|>\epsilon\mid\mathcal F_{t-1}\}
 \leq\C_\epsilon(K+t)^{-p/2},\qquad p=2+\min\{\delta,1\}>2,
\end{equation*}
for all sufficiently large $t$. For the late tail we used
$\mu_t^2\mathfrak n_t=2v^2(K+t)\log(\mathfrak n_t+1)$.
These bounds are deterministic and summable in $t$. \par Taking expectations and applying Borel--Cantelli for $\epsilon=1/q$, $q\geq2$, proves $N_t/\mathfrak n_t\to1$ almost surely. Moreover,
$\mathfrak n_t\geq\delta_t\log\delta_t$ and $\delta_t\geq2v^2(K+t)/\C_\star^2$, so $\inf_\mu\mathfrak n_t/t\to\infty$. Hence $N_t\geq t$ eventually almost surely.

For $h>0$, the inclusion
$\{|M_t-\mu_t|>h\}\subseteq\{N_t<t\}\cup\{\exists j\geq t:|W_j|>hj\}$
and dyadic blocks $2^rt\leq j<2^{r+1}t$ imply, by~\eqref{eq:noise-maximal-moment},
\begin{equation*}
 \mathbb P\{\exists j\geq t:|W_j|>hj\mid\mathcal F_{t-1}\}
 \leq\C h^{-p}\sum_{r\geq0}(2^rt)^{-p/2}
 \leq\C_h t^{-p/2}.
\end{equation*}
The probabilities of $N_t<t$ are also summable, by the early-tail bound with $\epsilon=1/2$. Borel--Cantelli and $h=1/q$ prove $M_t-\mu_t\to0$ almost surely. Finally, summing~\eqref{eq:false-sign-moment} proves that only finitely many completed-stage sign errors occur. 
\label{proof-end:lemma:almost surely}
\end{proof}

\begin{remark}
\label{remark:necessity of log} We pause and relation between our threshold choice and law of iterated logarithm (LIL).
For the i.i.d.\ centered errors assumed here, the finite, strictly
positive variance version of LIL \citep[Theorem~8.5.2]{durrett2019probability}, applied to $\varepsilon_{it}/\sigma$ and its negative, states
\begin{equation*}
 \limsup_{j\to\infty}
 \frac{W_j}{\sqrt{2\sigma^2j\log\log j}}=1,
 \qquad
 \liminf_{j\to\infty}
 \frac{W_j}{\sqrt{2\sigma^2j\log\log j}}=-1
 \quad\hbox{almost surely}.
\end{equation*}
Thus $|W_j|>\C\sqrt j$ infinitely often for every finite constant
$\C$: along the subsequence approaching either nonzero iterated-logarithm
limit, $|W_j|/\sqrt j$ tends to infinity. \par 
Such a boundary cannot
have a crossing probability bounded by a prescribed number below one.
Therefore we conjecture our result, from both theoretical and algorithmic aspect, does not require a factor $\sqrt{\log j}$ but do need a fact $\sqrt{\log \log j}$.

For completeness, an explicit smaller valid boundary can be constructed
under the additional moment generating function bound~\eqref{eq:noise-mgf}. For each integer $m\geq0$ and
$2^m\leq j<2^{m+1}$, we put
\begin{equation*}
 U_j=2v\sqrt{j\log\{2(m+1)(m+2)/\alpha\}},
 \qquad 0<\alpha<1.
\end{equation*}
For any integer $L\geq1$, the exponential process
$\exp\{\lambda W_j-j\lambda^2v^2/2\}$ is a nonnegative
supermartingale. Stop it at the first $W_j\geq x$ or at $L$.
Its expectation is at most its initial value one, by the bounded
stopping-time supermartingale inequality \citep[Theorem~4.4.1]{durrett2019probability}. On a crossing by $L$ its
value at stopping is at least
$\exp\{\lambda x-L\lambda^2v^2/2\}$. Taking
$\lambda=x/(Lv^2)$ and then repeating for $-W_j$ {yields}
\begin{equation*}
 \mathbb P\bigl(\max_{1\leq j\leq L}|W_j|\geq x\bigr)
 \leq2e^{-x^2/(2Lv^2)}.
\end{equation*}
On block $m$, $U_j$ is at least
$\sqrt{2v^2\,2^{m+1}\log\{2(m+1)(m+2)/\alpha\}}$.\par 
The preceding maximal inequality with $L=2^{m+1}$ therefore bounds
a crossing in that block by $\alpha/\{(m+1)(m+2)\}$.
These probabilities sum to $\alpha$, since
$1/\{(m+1)(m+2)\}=1/(m+1)-1/(m+2)$.
Consequently $\mathbb P(\exists j\geq1:|W_j|\geq U_j)\leq\alpha$.
Here $m$ has order $\log j$, so $U_j$ has order
$\sqrt{j\{\log\log j+\log(1/\alpha)\}}$ for large $j$.
The larger threshold used in the paper is chosen for the direct
summable calculations above. 
\end{remark}

\subsection{Martingale noise}
\label{subsec:martingale-noise}

We consider the conditional-noise setting in Remark~\ref{remark:martingale-noise}, using the filtrations in~\eqref{eq:stage-filtration}.  At a reached stage $t$, $X_t$ is $\mathcal F_{t-1}$-measurable and $\mu_t=f(X_t)$ stays fixed while observations accumulate. For observations indexed in the order collected, each sampling location and the decision to collect the next observation are measurable before its error is observed. With respect to this full observation history, we require the conditional mean and moment generating function bounds in Remark~\ref{remark:martingale-noise} for every actual observation, almost surely and for every $\lambda\in\mathbb R$, with the same known finite $v>0$. The errors are conditionally centered and sub-Gaussian with respect to this history, while their conditional distribution and variance may depend on the past.
Under this replacement of Assumption~\ref{assumption:regularity}.\ref{condition:noise}, Theorem~\ref{thm:power}, Theorem~\ref{thm:jump} and Proposition~\ref{prop:type1-error} hold with their respective assumptions on $f$, unchanged constants, and the same $K\geq\lceil\log_2(4/\Delta)\rceil$. The stage-finiteness statements of Proposition~\ref{prop:stage-finiteness} also hold. The proof below establishes these claims.

\begin{proof}[Proof of Remark~\ref{remark:martingale-noise} and the extensions stated above]
On the event $A_t$ that stage $t$ is reached, we extend its errors by zero after $N_t$ and define
\begin{equation}\label{eq:martingale-stopped-row}
W_{j,t}=\sum_{i=1}^{j\wedge N_t}\varepsilon_{i,t},\qquad
L_{j,t}(\lambda)=\exp\{\lambda W_{j,t}-\lambda^2v^2(j\wedge N_t)/2\},\quad j\geq0.
\end{equation}
The event $\{N_t\geq j\}$ is measurable before observation $j$. Conditional on the preceding within-stage history, the multiplicative increment of $L_{j,t}(\lambda)$ has expectation at most one by the conditional moment generating function assumption, and equals one after stopping. Thus $L_{j,t}(\lambda)$ is a nonnegative supermartingale starting at one for every fixed $\lambda\in\mathbb R$. This is the sub-Gaussian exponential construction in~\cite[Section~2, Definition~1 and the following discussion]{howard_time-uniform_2021}. At a random stage start, we condition on each possible finite starting time and apply the same conditional bound.

Since $j\wedge N_t\leq j$, conditional expectation and Chernoff's inequality {yield}, on $A_t$,
\begin{equation}\label{eq:martingale-row-tail}
\begin{aligned}
\mathbb{E}(e^{\lambda W_{j,t}}\mid\mathcal F_{t-1})
&\leq e^{j\lambda^2v^2/2},\\
\mathbb P(|W_{j,t}|>T(j,\alpha_t)\mid\mathcal F_{t-1})
&\leq2e^{-T(j,\alpha_t)^2/(2jv^2)}
=2(j+1)^{-(K+t)}.
\end{aligned}
\end{equation}
For the event $E$ of simultaneous noise control in Appendix~\ref{appendix:proof:thm:power}, the tower property and a union bound over reached stages and row lengths yield
\begin{equation*}
\mathbb P(E^c)\leq
2\sum_{t=1}^{\infty}\sum_{j=1}^{\infty}(j+1)^{-(K+t)}
\leq2^{1-K}(1+2/K)\leq2^{2-K}\leq\Delta.
\end{equation*}
This replaces the independent-error moment generating function factorization by iterated conditional expectation. On $E$, the rest of the rate proof is deterministic. Every completed sign test preserves $\theta$ in the updated interval. If a stage receives $m\geq2$ observations, its preceding sum has not crossed, so
\begin{equation*}
(m-1)|\mu_t|
\leq |S_{m-1,t}|+|W_{m-1,t}|
\leq2T(m-1,\alpha_t).
\end{equation*}
The interval-width and largest-stage-count arguments in Appendix~\ref{appendix:proof:thm:power} then {establish} both bounds of Theorem~\ref{thm:power}, including their starting-budget conditions. The same event {establishes} Proposition~\ref{prop:type1-error}. Under Assumption~\ref{assumption:jump point}, the deterministic gap argument in Appendix~\ref{appendix:proof:thm:jump} proves Theorem~\ref{thm:jump}. These statements count all observations, including an unfinished stage, and hold simultaneously over the reporting budgets. A finite-budget implementation observes an initial segment of this continuing experiment.

For Proposition~\ref{prop:stage-finiteness}, we replace {almost-sure convergence of i.i.d.\ sample averages} by a summable martingale tail bound. For each $u>0$, the same Chernoff argument {shows that}
\begin{equation*}
\sum_{j=1}^{\infty}
\mathbb P(|W_{j,t}|>uj\mid\mathcal F_{t-1})
\leq2\sum_{j=1}^{\infty}e^{-u^2j/(2v^2)}<\infty.
\end{equation*}
By Borel--Cantelli and a countable intersection over positive rational $u$, we have $W_{j,t}/j\to0$ almost surely. On $\{N_t=\infty\}$, this {implies} $S_{j,t}/j\to\mu_t$. If $\mu_t\ne0$, the relation $T(j,\alpha_t)/j\to0$ forces a finite crossing, a contradiction. At $\mu_t=0$, a finite crossing entails $|W_{j,t}|>T(j,\alpha_t)$ for some $j$, so the union bound above {establishes} the stated lower bound on $\mathbb P(N_t=\infty\mid\mathcal F_{t-1})$.
\end{proof}

These error and coverage guarantees use the conditional moment generating function bound and predictable sampling. The normal limits and efficiency analysis in Section~\ref{section:normality-efficiency} also use variance and distributional assumptions. 

%% file: appendix_new/CLT_proof.tex
\section{Randomly stopped averages}
\label{appendix:proof:thm:RSA}

Again, all one-stage arguments below are carried out conditionally on
$\mathcal F_{t-1}$. In particular, the random stage mean becomes a
fixed number under this conditional {distribution}, while the errors are i.i.d. distributed. Every approximation
has a deterministic error bound that tends to zero uniformly over
$0<|\mu_t|\leq \C_{\star}$. 

\begin{lemma}
\label{lemma:condition}
Suppose that for every continuity point $z$ of a distribution function
$F$, the conditional probabilities
$\mathbb P(A_t\leq z\mid\mathcal F_{t-1})$ converge in probability
to $F(z)$. Then $A_t$ converges in distribution to $F$.
\end{lemma}
\begin{proof}[Proof of Lemma~\ref{lemma:condition}]
\label{proof:lemma:condition}
We fix a continuity point $z$, and abbreviate
$Q_t=\mathbb P(A_t\leq z\mid\mathcal F_{t-1})$.
For every $\eta>0$, $0\leq Q_t\leq1$ and $0\leq F(z)\leq1$
imply
\begin{equation*}
 \mathbb{E}|Q_t-F(z)|
 \leq\eta+\mathbb P(|Q_t-F(z)|>\eta).
\end{equation*}
The second term tends to zero by the assumed convergence in
probability. Taking the limit superior and then letting
$\eta\downarrow0$ proves convergence in $L^1$. The tower property
now {yields}
\begin{equation*}
 \bigl|\mathbb P(A_t\leq z)-F(z)\bigr|
 =\bigl|\mathbb{E} Q_t-F(z)\bigr|
 \leq\mathbb{E}|Q_t-F(z)|\longrightarrow0.
\end{equation*}
This holds at every continuity point of $F$, which is the distribution
function criterion for convergence in distribution.
\label{proof-end:lemma:condition}
\end{proof}

\begin{proof}[Proof of Theorem~\ref{thm:RSA}]
\label{proof:thm:RSA}
We use the temporary notation $r=\mathfrak n_t, m=\lfloor r\rfloor, b=b_t,
  W_j=\sum_{i=1}^j\varepsilon_{it}, W_0=0.$
The conditionally fixed quantities satisfy
\begin{equation*}
 \mu_t^2r=b\log(r+1),\qquad
 \inf_{0<|\mu_t|\leq \C_{\star}}r\longrightarrow\infty,
 \qquad |m/r-1|\leq1/r.
\end{equation*}
Proposition~\ref{prop:stage-finiteness} ensures $N_t<\infty$
almost surely. {It follows from Lemma~\ref{lemma:stopping-localization} that}
$N_t/r\to1$ uniformly in conditional probability. We use the relative $L^1$ bound~\eqref{eq:stopping-relative-L1}, which also {establishes} the expectation asymptotic in Proposition~\ref{thm:EN}. 

We fix $u>0$ and $0<\eta<1/4$ and set
$L=\lceil2\eta r+2\rceil$. For all sufficiently large $r$,
$L\leq m$ and $L\leq2\eta r+3$.
The variables
\begin{equation*}
 W_{m+k}-W_m=\sum_{i=m+1}^{m+k}\varepsilon_{it},
 \qquad 1\leq k\leq L,
\end{equation*}
are partial sums of $L$ independent centered variables, each with
variance $\sigma^2$. Kolmogorov's maximal inequality \citep[Theorem~2.5.5]{durrett2019probability} therefore {implies}
\begin{equation*}
 \mathbb P\bigl(
   \max_{1\leq k\leq L}|W_{m+k}-W_m|>u\sqrt r
       \,\bigm|\,\mathcal F_{t-1}\bigr)
 \leq\frac{L\sigma^2}{u^2r}.
\end{equation*}
For the indices before $m$, the reversed list
$\varepsilon_{mt},\varepsilon_{m-1,t},\ldots,
\varepsilon_{m-L+1,t}$ is likewise a finite list of independent
centered variables with variance $\sigma^2$. Its first $k$ partial
sum is $W_m-W_{m-k}$. Applying the same maximal inequality to this
explicit list and taking a union bound yields
\begin{align*}
 &\mathbb P\bigl(
 \max_{\substack{j\geq0\\|j-m|\leq L}}
 |W_j-W_m|>u\sqrt r\,\bigm|\,\mathcal F_{t-1}\bigr)\\
 &\hspace{2cm}\leq\frac{2L\sigma^2}{u^2r}
 \leq\frac{4\sigma^2\eta}{u^2}
       +\frac{6\sigma^2}{u^2r}.
\end{align*}
The conditional horizons $m$ and $L$ are deterministic, which is
why the independence hypotheses of both applications hold.

On $|N_t/r-1|\leq\eta$, we have
$|N_t-m|\leq\eta r+1\leq L$. It follows that
\begin{align*}
 &\mathbb P\bigl(
   |W_{N_t}-W_m|>u\sqrt r\,\bigm|\,\mathcal F_{t-1}\bigr)\\
 &\quad\leq\mathbb P\bigl(
        |N_t/r-1|>\eta\,\bigm|\,\mathcal F_{t-1}\bigr)
       +\frac{4\sigma^2\eta}{u^2}
       +\frac{6\sigma^2}{u^2r}.
\end{align*}
{We take the supremum over the means and then let} $t\to\infty$.
{The ratio bound and the diverging lower bound} for $r$ leave at most
$4\sigma^2\eta/u^2$. Since $\eta$ can be arbitrarily small,
\begin{equation}
 \frac{W_{N_t}-W_m}{\sqrt r}=o_P(1)
 \label{eq:RSA-random-index-error}
\end{equation}
uniformly under the conditional laws. We control the dependence of $N_t$ on the current errors by bounding the maximum over all indices in a deterministic window.
{This deterministic-window argument is also used in the }i.i.d.{\ random-sum central limit theorem of \cite[Theorem~1]{renyi_random_sum_1957}.}

{We write $F_q$ for} the distribution function of
$q^{-1/2}\sum_{i=1}^q\varepsilon_i$ for $q$ independent errors
with the common distribution specified in Section~\ref{subsection:RSA}.
They have mean zero and variance $\sigma^2\in(0,\infty)$.
The i.i.d.\ finite-variance central limit theorem \citep[Theorem~3.4.1]{durrett2019probability} therefore {establishes}
$F_q(z)\to\Phi_\sigma(z)$ for every real $z$, where
$\Phi_\sigma$ is the distribution function of
$\mathcal{N}(0,\sigma^2)$.
For each fixed $z$, convergence of this numerical sequence implies
\begin{equation*}
 \sup_{q\geq q_0}|F_q(z)-\Phi_\sigma(z)|\longrightarrow0
 \qquad(q_0\to\infty).
\end{equation*}
The conditional distribution of $W_m/\sqrt m$ is precisely $F_m$,
and $\inf m\to\infty$. Thus this central limit theorem is uniform
in the conditionally fixed mean. Moreover,
$W_m/\sqrt r=\sqrt{m/r}\,W_m/\sqrt m$ and $m/r\to1$ uniformly.
The latter deterministic scaling keeps the limit. One can
see this directly by bounding its distribution function between
$F_m(z-h)$ and $F_m(z+h)$ once $r$ is large, for any $h>0$,
and using continuity of $\Phi_\sigma$.

For use with random perturbations, the needed distribution-function
argument is the following elementary inequality. For any real random
variables $U,V$, real $z$, and $h>0$,
\begin{equation*}
 \mathbb P(U\leq z-h)-\mathbb P(|U-V|>h)
 \leq\mathbb P(V\leq z)
 \leq\mathbb P(U\leq z+h)+\mathbb P(|U-V|>h).
\end{equation*}
It holds also under each conditional {distribution}, by the event inclusions
$\{U\leq z-h,|U-V|\leq h\}\subseteq\{V\leq z\}$ and
$\{V\leq z,|U-V|\leq h\}\subseteq\{U\leq z+h\}$.
It proves that a uniform $o_P(1)$ perturbation preserves a continuous
uniform distributional limit: we take uniform limits at $z-h,z+h$
and then let $h\downarrow0$.

We also need uniform tightness. Chebyshev's inequality {shows that}
\begin{equation*}
 \mathbb P\bigl(|W_m|/\sqrt r>A\,\bigm|\,
                     \mathcal F_{t-1}\bigr)
 \leq\frac{m\sigma^2}{rA^2}\leq\frac{\sigma^2}{A^2}.
\end{equation*}
The ratio bounds {show} that $\sqrt{r/N_t}\to1$ uniformly in conditional
probability. Together with~\eqref{eq:RSA-random-index-error}, the
identity
\begin{align*}
 \frac{W_{N_t}}{\sqrt{N_t}}-\frac{W_m}{\sqrt r}
 &=\frac{W_{N_t}-W_m}{\sqrt r}\sqrt{\frac r{N_t}}
   +\frac{W_m}{\sqrt r}
                \Bigl(\sqrt{\frac r{N_t}}-1\Bigr)
\end{align*}
shows that the left side is uniformly $o_P(1)$. To justify each
product, split a tight factor at a deterministic level $A$: for
any variables $U,V$,
$\mathbb P(|UV|>h)\leq\mathbb P(|U|>A)
+\mathbb P(|V|>h/A)$. First make the first probability small using
tightness, and then let $t$ increase in the second probability.
The perturbation inequality now proves
\begin{equation}
 \frac{W_{N_t}}{\sqrt{N_t}}
 =\frac{W_m}{\sqrt r}+o_P(1)
 \ \Longrightarrow\ \mathcal{N}(0,\sigma^2)
 \label{eq:RSA-Anscombe}
\end{equation}
uniformly conditionally. Since
$\sqrt{N_t}(M_t-\mu_t)=W_{N_t}/\sqrt{N_t}$, this is exactly
\eqref{eq:RSA2}.

We will also use the same convergence after multiplication by
${\textsf{sign}}(\mu_t)$. This sign is fixed under the
conditional {distribution} and takes only two possible values. The errors and
their negatives are both i.i.d.\ centered variables of variance
$\sigma^2$. Their deterministic-horizon limits are both
$\mathcal{N}(0,\sigma^2)$, and the absolute-value probability bounds
used to prove~\eqref{eq:RSA-random-index-error} apply to both. Taking the larger of the two deterministic
error bounds therefore proves the uniform limit after reflection,
even if the predictable sign changes with $t$.

The function $x\mapsto\sqrt{bx\log(x+1)}$ is increasing on
$[0,\infty)$, because the derivative of $x\log(x+1)$ is
$\log(x+1)+x/(x+1)>0$ for $x>0$. {We define the normalized \textit{overshoot}}
\begin{equation*}
 R_t=\frac{|S_{N_t,t}|-T(N_t,\alpha_t)}{\sqrt{N_t}}\geq0,
 \qquad S_{0t}=T(0,\alpha_t)=0.
\end{equation*}
{For $N_t>1$, it follows from the definition of the first crossing that}
$|S_{N_t-1,t}|\leq T(N_t-1,\alpha_t)$, and this holds as well when
$N_t=1$ with the zero conventions. The triangle inequality and
monotonicity of the threshold imply
\begin{align*}
 |S_{N_t,t}|-T(N_t,\alpha_t)
 &\leq |S_{N_t-1,t}|+|Y_{N_t,t}|-T(N_t,\alpha_t)\\
 &\leq T(N_t-1,\alpha_t)+|Y_{N_t,t}|-T(N_t,\alpha_t)\\
 &\leq|Y_{N_t,t}|\leq {\C_{\star}}+|\varepsilon_{N_t,t}|.
\end{align*}
Hence $R_t\leq({\C_{\star}}+|\varepsilon_{N_t,t}|)/\sqrt{N_t}$.

The event $r/2\leq N_t\leq2r$ has conditional probability tending
to one uniformly, {by the ratio bound}. {We fix $h>0$.} On that event, and
for all sufficiently large $r$ that ${\C_{\star}}\leq(h/2)\sqrt{r/2}$,
$R_t>h$ implies
\begin{equation*}
 \max_{1\leq j\leq\lceil2r\rceil}|\varepsilon_{jt}|
 >\frac h{2\sqrt2}\sqrt r.
\end{equation*}
\begingroup\color{black}\hypersetup{linkcolor=black,citecolor=black}
Finite variance implies $x^2\mathbb P(|\varepsilon|>x)\leq\mathbb E[\varepsilon^2\mathbf1_{\{|\varepsilon|>x\}}]\to0$. {Taking a union bound, we obtain}
\begin{align*}
 &\mathbb P\left(\max_{1\leq j\leq\lceil2r\rceil}|\varepsilon_{jt}|>
 \frac{h\sqrt r}{2\sqrt2}\,\middle|\,\mathcal F_{t-1}\right)\\
 &\qquad\leq\lceil2r\rceil\mathbb P\left(|\varepsilon|>\frac{h\sqrt r}{2\sqrt2}\right)\longrightarrow0
\end{align*}
uniformly above any diverging lower bound for $r$. This proves $R_t=o_P(1)$ uniformly conditionally. 
\endgroup

{We denote by $G_t$ the event that} the crossing sign equals the drift sign.
By Lemma~\ref{lemma:false sign},
\begin{equation*}
 \mathbb P(G_t^c\mid\mathcal F_{t-1})
 {\leq\C(K+t)^{-p/2}\longrightarrow0,\qquad p=2+\min\{\delta,1\}>2}
\end{equation*}
uniformly. On $G_t$,
\begin{equation*}
 |S_{N_t,t}|={\textsf{sign}}(\mu_t)S_{N_t,t}
 =|\mu_t|N_t+{\textsf{sign}}(\mu_t)W_{N_t}.
\end{equation*}
Substitution in the definition of $R_t$, followed by division by
$\sqrt{N_t}$, {yields} the exact identity
\begin{equation}
 |\mu_t|\sqrt{N_t}-\sqrt{b\log(N_t+1)}
 =-{\textsf{sign}}(\mu_t)\frac{W_{N_t}}{\sqrt{N_t}}+R_t.
 \label{eq:RSA-boundary-identity}
\end{equation}
An equality restricted to $G_t$ may be used for convergence in
probability because its complement has probability tending to zero. No bound on either side on $G_t^c$ is required for that use.

{For $x>0$, we define}
\begin{equation*}
 h_t(x)=|\mu_t|\sqrt x-\sqrt{b\log(x+1)}.
\end{equation*}
The centering equation implies $h_t(r)=0$. Direct differentiation
and $b/\mu_t^2=r/\log(r+1)$ {yield}
\begin{align*}
 h_t'(x)
 &=\frac{|\mu_t|}{2\sqrt x}
   -\frac{\sqrt b}{2(x+1)\sqrt{\log(x+1)}},\\
 \frac{h_t'(x)}{|\mu_t|/(2\sqrt r)}
 &=\sqrt{\frac r x}
   -\frac{r}{(x+1)\sqrt{\log(r+1)\log(x+1)}}.
\end{align*}
We verify the uniform approximation of this ratio by one. For
$0<\eta<1/2$ and $(1-\eta)r\leq x\leq(1+\eta)r$,
\begin{equation*}
 \bigl|\sqrt{r/x}-1\bigr|
 \leq\max\bigl\{(1-\eta)^{-1/2}-1,
                      1-(1+\eta)^{-1/2}\bigr\}.
\end{equation*}
The positive second term in the derivative ratio is bounded above by
\begin{equation*}
 \frac1{(1-\eta)
      \sqrt{\log(r+1)\log((1-\eta)r+1)}}\longrightarrow0
\end{equation*}
uniformly above the diverging lower bound for $r$.
The first displayed maximum tends to zero as $\eta\downarrow0$.
Thus, on neighborhoods whose relative radius tends to zero, the
normalized derivative tends uniformly to one.

To avoid an implicit choice of a random intermediate point in the
mean value theorem, we define the difference quotient explicitly:
\begin{equation*}
 D_t=
 \begin{cases}
 \displaystyle
 \frac{h_t(N_t)-h_t(r)}{N_t-r}\frac{2\sqrt r}{|\mu_t|},
       &N_t\neq r,\\[6pt]
 \displaystyle h_t'(r)\frac{2\sqrt r}{|\mu_t|},&N_t=r.
 \end{cases}
\end{equation*}
The fundamental theorem of calculus writes the first expression as
the average of the normalized derivatives over the segment joining
$r$ and $N_t$. On $|N_t/r-1|\leq\eta$, that segment lies in the
window just considered. For every fixed error tolerance, we choose
$\eta$ to control its first term and then increase $t$ to control
the second. {The ratio bound} controls the probability of leaving the
window. This proves $D_t\to1$ uniformly in conditional probability.
By definition, including the case $N_t=r$,
\begin{equation*}
 h_t(N_t)=D_t\frac{|\mu_t|(N_t-r)}{2\sqrt r}
 =D_t\frac{N_t-r}{2\sqrt{r/\mu_t^2}}.
\end{equation*}

On $G_t\cap\{|D_t-1|\leq1/2\}$, the derivative factor is
nonzero and~\eqref{eq:RSA-boundary-identity} {implies}
\begin{equation*}
 \frac{N_t-r}{2\sqrt{r/\mu_t^2}}
 =D_t^{-1}\Bigl\{
 -{\textsf{sign}}(\mu_t)\frac{W_{N_t}}{\sqrt{N_t}}+R_t
 \Bigr\}.
\end{equation*}
The event on which this equality is used has conditional probability
tending to one uniformly. On that event
$|D_t^{-1}-1|\leq2|D_t-1|\to0$ in probability. The stopped
normalized sum is uniformly tight by~\eqref{eq:RSA-Anscombe}, and $R_t=o_P(1)$ by the overshoot bound above. Applying the tight-factor product argument used for~\eqref{eq:RSA-Anscombe} therefore proves
\begin{equation*}
 \frac{N_t-\mathfrak n_t}{2\sqrt{\mathfrak n_t/\mu_t^2}}
 =-{\textsf{sign}}(\mu_t)\frac{W_{N_t}}{\sqrt{N_t}}
   +o_P(1).
\end{equation*}
Combining this expansion with~\eqref{eq:RSA-Anscombe} and the distribution-function perturbation inequality {establishes}~\eqref{eq:RSA1}, with variance $\sigma^2$.

For each fixed real $z$, the two preceding distributional limits
mean the following. For each statistic $A_t$ in
\eqref{eq:RSA1} and~\eqref{eq:RSA2}, there is a deterministic
sequence $d_t(z)\downarrow0$ (replacing an error sequence by its
nonincreasing tail supremum if necessary) such that
\begin{equation*}
 \bigl|\mathbb P(A_t\leq z\mid\mathcal F_{t-1})
             -\Phi_\sigma(z)\bigr|\leq d_t(z)
 \quad\hbox{almost surely}
\end{equation*}
for all sufficiently large $t$, for every predictable mean satisfying
$0<|\mu_t|\leq {\C_{\star}}$. To see that the error bounds are deterministic,
{the ratio bounds hold uniformly as functions of $t$},
the maximal bound is uniform above the deterministic lower bound for
$r$, and the only CLT error is the tail supremum of the common i.i.d.
error distribution. The overshoot and derivative calculations have
the same uniform structure.
\\Taking the countable intersection of
the probability-one sets where the conditional-{distribution} representation
holds for all stages proves conditional almost-sure convergence for
each $z$. If a single probability-one set for all $z$ is desired,
we first intersect over rational $z$ and then place any real $z$
between rational points, using monotonicity of distribution functions
and continuity of $\Phi_\sigma$.
Lemma~\ref{lemma:condition} now {establishes} both unconditional limits.

At a crossing, $S_{N_t,t}$ is nonzero because its absolute value
strictly exceeds the positive threshold. Thus
\begin{align*}
 &\sqrt{N_t}\Bigl\{M_t-
       {\textsf{sign}}(S_{N_t,t})
                   \frac{T(N_t,\alpha_t)}{N_t}\Bigr\}\\
 &\qquad={\textsf{sign}}(S_{N_t,t})
       \frac{|S_{N_t,t}|-T(N_t,\alpha_t)}{\sqrt{N_t}}
 ={\textsf{sign}}(S_{N_t,t})R_t=o_P(1).
\end{align*}
Subtracting this negligible difference from
$\sqrt{N_t}(M_t-\mu_t)$ and using the perturbation inequality
proves the same centered Gaussian limit for the signed threshold
average, conditionally and unconditionally.
\label{proof-end:thm:RSA}
\end{proof}

\begin{corollary}
\label{corollary:ratio_in_p}
Under the conditions of Theorem~\ref{thm:RSA},
$N_t/\mathfrak n_t\to1$ in probability. With our choice of
$\alpha_t$, this convergence holds almost surely.
\end{corollary}
\begin{proof}[Proof of Corollary~\ref{corollary:ratio_in_p}]
\label{proof:corollary:ratio_in_p}
For any fixed $\epsilon>0$, uniform conditional $L^1$ convergence
in~\eqref{eq:stopping-relative-L1} and Markov's inequality {yield}
\begin{equation*}
 \mathbb P\bigl(|N_t/\mathfrak n_t-1|>\epsilon
                       \,\bigm|\,\mathcal F_{t-1}\bigr)
 \leq\epsilon^{-1}
 \mathbb{E}\bigl[|N_t/\mathfrak n_t-1|\mid\mathcal F_{t-1}\bigr]
 \longrightarrow0
\end{equation*}
with a deterministic uniform upper bound. Taking expectations proves
the probability statement. For almost-sure convergence, the explicit
early and late tail probabilities are summable with
$\alpha_t=e^{-(K+t)}$, as checked in the proof of
Lemma~\ref{lemma:almost surely}. Borel--Cantelli there proves the
second statement.
\label{proof-end:corollary:ratio_in_p}
\end{proof}

%% file: appendix_new/minimax_proof.tex
\section{Adaptive minimax lower bound}
\label{SM:proof:thm:minimax}
We prove the adaptive minimax lower bound for the class $\mathcal F_\gamma^{\loc}(\c_0,\underline{\eta})$ defined in Section~\ref{subsec:minimax}. We redefine it here for the sake of convenience for readers:
\begin{equation}
\begin{aligned}
    \mathcal F_\gamma^{\loc}(\c_0,\underline{\eta}) :&=
    \Big\{f:\mathcal{I}\to \mathbb{R}:\text{ $f$ satisfies Assumption~\ref{assumption:regularity}.\ref{condition:sign}, Assumption~\ref{assumption:regularity}.\ref{condition:upper_envelope}, and }\\
    &\qquad|f(x)|\geq\c_0|x-\theta(f)|^\gamma
 \quad\text{for }x\in\mathcal I\text{ with }|x-\theta(f)|\leq\underline{\eta}.\Big\}
\end{aligned}
\end{equation}
and let the noise satisfy Assumption~\ref{assumption:regularity}.\ref{condition:noise}. The proof from a high level, proceeds with Le Cam's two point method~\citep{tsybakov2009introduction} with a delicate choice of the regression functions $f_0,f_1$ as shown in Figure~\ref{fig:f0-f1}. From the perspective technicality, an analogous cancellation of adaptive kernels is used for binary-label active learning by~\cite{castro_minimax_2008}.
\begin{theorem}[\textbf{Theorem~\ref{thm:minimax}} restatement]
\label{thm:minimax_re}
Suppose the errors are independent $\mathcal{N}(0,\sigma^2)$ random variables, with $\sigma>0$. For every $\gamma>0$ there is a constant $\c_\star=\c_\star(\gamma)>0$ such that, for all $n\geq1$ and $\c_0,\underline{\eta}>0$,
\begin{align}
    \inf_{\A\in \mathcal{A}_n} \sup_{f\in \mathcal{F}_\gamma^{\loc}(\c_0,\underline{\eta})}
    \bbP\Big(\big|\widehat{\theta}_n(\A)- \theta(f)\big|\geq \c_\star\Big\{R\wedge \Big(\frac{\sigma^2}{\c_0^2 2^\gamma n}\Big)^{1/(2\gamma)}\Big\}\Big)\geq \frac{1}{4}.
\end{align}
Consequently, we have
\begin{align}
 \inf_{\A\in\mathcal A_n}\sup_{f\in\mathcal F_\gamma^{\loc}(\c_0,\underline{\eta})}
 \mathbb{E}\big|\widehat\theta_n(\A)-\theta(f)\big|^2
\geq \frac{\c_\star^2}{4}
 \Big\{R^2\wedge\Big(\frac{\sigma^2}{\c_0^2 2^\gamma n}\Big)^{1/\gamma}\Big\}.
\end{align}
\end{theorem}
\begin{proof}[Proof of Theorem~\ref{thm:minimax}]
\label{proof:thm:minimax-lower}

We fix a procedure $\A\in\mathcal A_n$ and choose a root separation $0<d\leq R/2$.
We first construct a pair of regression functions $(f_0,f_1)$ independently of the
procedure, then bound the divergence of their full observation laws.

\begin{figure}[!t]
  \normalfont
  \centering
  \resizebox{\linewidth}{!}{%
    \begin{tikzpicture}
      \FunctionPanel{0.5}{2}
      \begin{scope}[xshift=6.3cm]
        \FunctionPanel{1}{3.5}
      \end{scope}
      \begin{scope}[xshift=12.6cm]
        \FunctionPanel{2}{10}
      \end{scope}
    \end{tikzpicture}%
  }
  \caption{The functions $f_0$ and $f_1$ for
    $\gamma\in\{1/2,1,2\}$, shown in normalized coordinates.}
  \label{fig:f0-f1}
\end{figure}

{Figure~\ref{fig:f0-f1} illustrates the two alternatives used in Theorem~\ref{thm:minimax}. Their separated roots and controlled statistical distance {establish} the Gaussian minimax lower bound{. These} are deterministic proof constructions.}
We define \begin{equation}\label{eq:minimax-roots}
 m=\frac{a_0+b_0}{2},\qquad
 \theta_0=m-d/2,\qquad\theta_1=m+d/2.
\end{equation}
Since $d\leq R/2$, both roots lie in $[a_0+R/4,b_0-R/4]\subseteq \mathcal{I}=(a_0,b_0)$. We define continuous piecewise linear maps on
$\mathbb R$ by
\begin{equation}\label{eq:minimax-warps}
 h_0(z)=
 \begin{cases}
 z-d,&z\leq-d,\\
 2z,&-d\leq z\leq0,\\
 z,&z\geq0,
 \end{cases}
 \qquad
 h_1(z)=
 \begin{cases}
 z-d,&z\leq d,\\
 2(z-d),&d\leq z\leq2d,\\
 z,&z\geq2d.
 \end{cases}
\end{equation}
We denote  $\phi_\gamma(u)={\textsf{sign}}(u)|u|^\gamma$, with convention that
$\phi_\gamma(0)=0$, and define
\begin{equation}\label{eq:minimax-alternatives}
 {f_j(x)=\c_0\phi_\gamma(h_j(x-\theta_0)),}
 \qquad x\in\mathcal I,\quad j=0,1.
\end{equation}
The map $\phi_\gamma$ is continuous and strictly increasing
for every $\gamma>0$. The functions $f_j$ therefore have these same
properties. The unique zeros of $h_0,h_1$ are $0,d$, respectively, as shown in Figure~\ref{fig:f0-f1}, so
$f_j$ has unique root $\theta_j$. The inequalities below verify the local separation condition for every $\underline{\eta}>0$ on its intersection with $\mathcal I$.

For every real $z$, we have
\begin{equation}\label{eq:minimax-warp-bounds}
 |z|\leq|h_0(z)|\leq2|z|,
 \qquad |z-d|\leq|h_1(z)|\leq2|z-d|.
\end{equation}
For the first inequality, if $z\leq-d$, then
$|h_0(z)|=|z|+d\in[|z|,2|z|]$, on $[-d,0]$ it equals $2|z|$, and on $[0,\infty)$ it equals $|z|$.
For the second inequality, the first two pieces of $h_1$ {imply} equality
with $|z-d|$ and $2|z-d|$, respectively. On $z\geq2d$,
$|h_1(z)|=z=(z-d)+d\in[z-d,2(z-d)]$.

Both alternatives belong to $\mathcal F_\gamma^{\loc}(\c_0,\underline{\eta})$ and are bounded on $\mathcal I$. Their continuous strict sign changes also satisfies Assumption~\ref{assumption:regularity}.\ref{condition:sign}.

{We reparametrize with $z=x-\theta_0$.} The alternatives coincide when $z\leq-d$ or
$z\geq2d$. On $-d\leq z\leq0$,
\begin{equation*}
 f_0(x)-f_1(x)
 =\c_0\{(d-z)^\gamma-(-2z)^\gamma\}
 \in[0,\c_0(2d)^\gamma],
\end{equation*}
because $0\leq-2z\leq d-z\leq2d$. On $d\leq z\leq2d$,
\begin{equation*}
 f_0(x)-f_1(x)
 =\c_0\{z^\gamma-[2(z-d)]^\gamma\}
 \in[0,\c_0(2d)^\gamma].
\end{equation*}
On $0\leq z\leq d$, the difference is
$\c_0\{z^\gamma+(d-z)^\gamma\}$.
If $\gamma\geq1$, we use $u^\gamma\leq u$ for $u\in[0,1]$ to bound
this by $\c_0d^\gamma$. {If $0<\gamma<1$, it follows from concavity that}
\begin{equation*}
 z^\gamma+(d-z)^\gamma
 \leq2(d/2)^\gamma=2^{1-\gamma}d^\gamma.
\end{equation*}
We next compare the full observation laws under independent $\mathcal{N}(0,\sigma^2)$ errors. We denote $P_j$ as the joint {distribution} of $(\H_n,\widehat\theta_n)$ under $f_j$, and use the Gaussian density
\begin{equation*}
 p_j(y\mid x)=\frac1{\sqrt{2\pi}\sigma}
       \exp\Bigl\{-\frac{(y-f_j(x))^2}{2\sigma^2}\Bigr\}.
\end{equation*}
{By successive conditioning, we obtain the joint {distribution}}
\begin{equation}\label{eq:minimax-history-density}
P_j(dh_n,dz)=\Big\{\prod_{i=1}^n K_i(dx_i| h_{i-1})\,p_j(y_i| x_i)\,dy_i\Big\}K_{n+1}(dz| h_n),
\end{equation}
for $j\in \{0,1\}$.
The product denotes successive integration against the displayed kernels. 

The Gaussian densities are positive, so $P_0$ and $P_1$ are mutually absolutely
continuous, and
\begin{equation}\label{eq:minimax-likelihood}
 \log\frac{dP_0}{dP_1}
 =\sum_{i=1}^n
  \frac{(Y_i-f_1(X_i))^2-(Y_i-f_0(X_i))^2}{2\sigma^2}.
\end{equation}
Under $P_0$, we set $D_i=f_0(X_i)-f_1(X_i)$. The $i$th summand
equals $(D_i^2+2D_i\varepsilon_i)/(2\sigma^2)$.
Conditional on $(\H_{i-1},X_i)$, $D_i$ is fixed and $\mathbb{E}_0(\varepsilon_i\mid\H_{i-1},X_i)=0$. With $K_\gamma=2^{\gamma\vee(1-\gamma)}$, the KL divergence satisfies 
\begin{align}
 L:=D_{\textsf{ KL}}(P_0\Vert P_1)
 &=\frac1{2\sigma^2}\mathbb{E}_0
            \sum_{i=1}^n\{f_0(X_i)-f_1(X_i)\}^2\notag\\
 &\overset{(\heartsuit)}{\leq }\frac{nK_\gamma^2\c_0^2d^{2\gamma}}{2\sigma^2},
 \label{eq:minimax-kl}
\end{align}
by the preceding conditional Gaussian calculation, where the inequality $(\heartsuit)$ uses the previously established 
\begin{align}
    {0\leq f_0(x)-f_1(x)\leq K_\gamma\c_0d^\gamma,\qquad K_\gamma=2^{\gamma\vee(1-\gamma)}.}
\end{align}
This calculation holds for every measurable adaptive algorithm in
$\mathcal A_n$. The remaining arguments are standard: we use the two-point testing argument of~\cite[Section~2.3 and Theorem~2.2]{tsybakov2009introduction} for the full adaptive histories. We choose densities $p_0,p_1$ of $P_0,P_1$ with respect to a common
dominating measure $\lambda$, and write
\begin{equation*}
 \normalfont{\textsf{TV}}(P_0,P_1)=\frac12\int|p_0-p_1|\,d\lambda,
 \qquad \rho=\int\sqrt{p_0p_1}\,d\lambda.
\end{equation*}
The first quantity is their total variation distance. For a test that
selects alternative $1$ with conditional probability $\varphi\in[0,1]$,
the pointwise inequality
$\varphi p_0+(1-\varphi)p_1\geq\min(p_0,p_1)$ {implies}
\begin{equation}\label{eq:minimax-testing-error}
  \mathbb{E}_0\varphi+ \mathbb{E}_1(1-\varphi)
 \geq\int\min(p_0,p_1)\,d\lambda=1-
 \normalfont{\textsf{TV}}(P_0,P_1).
\end{equation}
We use Pinsker's inequality~\citep{tsybakov2009introduction}, which we verify below in the present notation: $\normalfont{\textsf{TV}}(P_0,P_1)\leq\sqrt{L/2}$. Precisely, we set $\widetilde{E}=\{p_0\geq p_1\}$, $a=P_0(\widetilde{E})$, and $b=P_1(\widetilde{E})$.
{If $P_0=P_1$, the Pinsker inequality is immediate. Otherwise, $\normalfont{\textsf{TV}}(P_0,P_1)=a-b$ and mutual absolute continuity imply $0<b<a<1$.}
{Conditional Jensen's inequality on $\widetilde E$ {implies}}
\begin{equation*}
 {\int_{\widetilde E}p_0\log(p_0/p_1)\,d\lambda}
 =-a {\mathbb{E}_0[\log(p_1/p_0)\mid\widetilde E]}
 \geq-a\log(b/a).
\end{equation*}
Applying the same calculation on $\widetilde{E}^c$, we obtain
\begin{equation*}
 L\geq a\log(a/b)+(1-a)\log\{(1-a)/(1-b)\}=:g_b(a).
\end{equation*}
The function $g_b$ satisfies $g_b(b)=g_b'(b)=0$ and
$g_b''(u)=1/[u(1-u)]\geq4$ for $u\in(0,1)$.
{Integrating twice from $b$ to $a$, we obtain}
$g_b(a)\geq2(a-b)^2$. Hence
\begin{equation}\label{eq:minimax-pinsker}
 \normalfont{\textsf{TV}}(P_0,P_1)\leq\sqrt{L/2}.
\end{equation}
We take $\widehat\varphi=\mathbf1\{\widehat\theta_n\geq m\}$.
Under $P_0$, a testing error entails
{$|\widehat\theta_n-\theta_0|\geq d/2$.} Under $P_1$, a testing error
entails $|\widehat\theta_n-\theta_1|>d/2$. Therefore
\begin{align}
 \max_{j=0,1}\mathbb P_j
       \{|\widehat\theta_n-\theta_j|\geq d/2\}
 &\geq\frac12\{P_0(\widehat\varphi=1)
                         +P_1(\widehat\varphi=0)\}\notag\\
 &\geq\frac{1-
 \normalfont{\textsf{TV}}(P_0,P_1)}{2}.
 \label{eq:minimax-testing-reduction}
\end{align}
We choose
\begin{equation*}
 d=\frac R2\wedge\Bigl(\frac{\sigma^2}{nK_\gamma^2\c_0^2}\Bigr)^{1/(2\gamma)}.
\end{equation*}
Then $L\leq1/2$, so~\eqref{eq:minimax-pinsker} {implies} $\normalfont{\textsf{TV}}(P_0,P_1)\leq1/2$. Equation~\eqref{eq:minimax-testing-reduction} therefore {establishes} error probability at least $1/4$ at radius $d/2$. Taking
\begin{equation*}
 \c_\star=\min\{1/4,2^{-1/2}K_\gamma^{-1/\gamma}\}
\end{equation*}
makes the radius in Theorem~\ref{thm:minimax} no larger than $d/2$. The supremum over the class is at least the maximum over these two alternatives, and the argument holds for every $\A\in\mathcal A_n$. Finally, by Markov's inequality, $\mathbb{E} Z^2\geq r^2\mathbb P(|Z|\geq r)$.\par This concludes the proof of the lower bound.
\end{proof}  

%% file: appendix_new/jump_proof.tex
\section{Discontinuous crossings}
\label{appendix:proof:thm:jump}

\begin{theorem}[\textbf{Theorem~\ref{thm:jump}} restatement]
Under Assumptions~\ref{assumption:regularity} and~\ref{assumption:jump point}, let $\Delta\in(0,1)$ and $K\geq\lceil\log_2(4/\Delta)\rceil$, $b=\inf_{x\in\mathcal I\setminus\{\theta\}}|f(x)|>0$, and $\kappa=(b\log 2)/(16\sqrt2\,v).$ The starting sample size $n_0$  the smallest integer $n\geq3$ satisfying 
\begin{equation}
n/\log(n+1)\geq(256v^2(K+1)^2)/b^2
\end{equation}and 
\begin{equation}
    n\log(n+1)\geq b^2/4v^2.
\end{equation}
For Algorithm~\ref{alg:stb} with output~\eqref{eq:uncapped-output}, simultaneously for all $n\geq n_0$,
\begin{equation}
 |\widehat\theta_n-\theta|
 \leq R\,\exp\big\{-\kappa\sqrt{n/\log(n)}\big\},
\end{equation}
with probability at least $1-\Delta$.
\end{theorem}
\begin{proof}[Proof of Theorem~\ref{thm:jump}]
\label{proof:thm:jump}
By the local two-sided limits in Assumption~\ref{assumption:jump point}, we have $|f(x)|\geq\min(\mu_-,\mu_+)/2$ in a sufficiently small neighborhood of $\theta$. Assumption~\ref{assumption:regularity}.\ref{condition:sign} bounds $|f|$ away from zero on the remaining compact search region. Thus $b=\inf_{x\in\mathcal I\setminus\{\theta\}}|f(x)|>0$. We work on the event $E$ in Appendix~\ref{appendix:proof:thm:power}, whose probability is at least $1-\Delta$.

As stated in the theorem, let $T=\lfloor (b/(8v))\sqrt{n/\log(n+1)}\rfloor$. The first starting condition in Theorem~\ref{thm:jump} implies
\begin{equation}
\frac b{8v}\sqrt{\frac n{\log(n+1)}}\geq2(K+1),\qquad T\geq K,\qquad T\geq\frac b{16v}\sqrt{\frac n{\log(n+1)}}.
\end{equation}
The second condition {implies} $n\geq(b/(2v))\sqrt{n/\log(n+1)}\geq4T$. Both starting conditions hold for larger $n$.

\textbf{Case 1.} If at least $T$ stages complete, projection onto the current interval {ensures} error at most $R2^{-T}$.

\textbf{Case 2.} Otherwise at most $T$ stages receive observations, and the selected stage has $m_{t_n}(n)\geq n/T$. Equation~\eqref{eq:uncapped-signal-bound} then {implies}
\begin{equation*}{|f(X_{t_n})|\leq4v\sqrt{T(K+T)\log(n+1)/n}\leq b/\sqrt2<b.}
\end{equation*}
The gap implies $X_{t_n}=\theta$, so its projection has zero error. In either case, using $\log(n+1)\leq2\log n$ for $n\geq3$, we obtain
\begin{equation*}
|\widehat\theta_n-\theta|\leq R2^{-T}\leq R\,\exp\!\Bigl\{-\frac{b\log2}{16\sqrt2\,v}\sqrt{\frac n{\log n}}\Bigr\}.
\end{equation*}
This proves the stated constants. The event $E$ guarantees that the same failure probability $\Delta$ throughout this range.
\end{proof}

%% file: appendix_new/anytime_validity_proof.tex
\section{Simultaneous interval coverage}
\label{appendix:proof:thm:type1}
We formalize the coverage of the automatic confidence sequences below.
\begin{proposition}[\textbf{Proposition~\ref{prop:type1-error}} restatement]
\label{prop:type1-error_re}
Under Assumption~\ref{assumption:regularity}.\ref{condition:noise} and \ref{assumption:regularity}.\ref{condition:sign}, and with an initial interval containing $\theta$, Algorithm~\ref{alg:stb} uses
$\alpha_t=e^{-(K+t)}$, where
$K\geq\lceil\log_2(4/\Delta)\rceil$ and $\Delta\in(0,1)$. Then the intervals $\mathcal{I}_t = [X_{\ell,t},X_{r,t}]$ form a $(1-\Delta)$ confidence sequence for $\theta$:
\begin{equation}
\mathbb P\bigl(\exists\text{ reached stage }t:
\mathcal I_t\not\owns \theta\bigr)\leq\Delta.
\end{equation}
\end{proposition}
\begin{proof}[Proof of Proposition~\ref{prop:type1-error}]
\label{proof:prop:type1-error}
We write $A_t$ for the event that stage $t$ is reached. It is determined by
the completed history before that stage. Conditional on this history and
on $A_t$, the errors $\varepsilon_{i,t}$ are independent and satisfy the moment generating function bound~\eqref{eq:noise-mgf}. The conditional Chernoff calculation in Section~\ref{subsec:stage} {implies}
\begin{equation*}
 \mathbb P\bigl(A_t\cap
 \bigl\{\exists j\geq1:
 |S_{j,t}-jf(X_t)|>T(j,\alpha_t)\bigr\}\bigr)
 \overset{(\ddagger)}{\leq }\mathbb P(A_t)\,2\sum_{j=1}^\infty(j+1)^{-(K+t)}.
\end{equation*}
Summing over the countably many
indices, using $\mathbb P(A_t)\leq1$ and~\eqref{eq:stage-error-sum},
shows that the following event has probability at least $1-\Delta$:
\begin{equation}\label{eq:cs-good-event}
|S_{j,t}-jf(X_t)|\leq T(j,\alpha_t)
\quad\text{for every reached }t\text{ and every }j\geq1.
\end{equation}
The bound above is conditional on the history selecting each location. Taking expectations and summing these bounds therefore allows adaptive, dependent locations.

We work on~\eqref{eq:cs-good-event} and suppose inductively that
$\theta\in[X_{\ell,t},X_{r,t}]$ at a reached stage. If a positive
crossing occurs, then
\begin{equation*}
 N_tf(X_t)=S_{N_t,t}-(S_{N_t,t}-N_tf(X_t))
 \geq S_{N_t,t}-T(N_t,\alpha_t)>0.
\end{equation*}
Thus $f(X_t)>0$. If $X_t<\theta$, Assumption~\ref{assumption:regularity}.\ref{condition:sign} would imply
$f(X_t)<0$, a contradiction. Hence $X_t\geq\theta$, and
$\theta\in[X_{\ell,t},X_t]$ after the right-endpoint replacement.
For a negative crossing,
\begin{equation*}
 N_tf(X_t)\leq S_{N_t,t}+T(N_t,\alpha_t)<0,
\end{equation*}
so $X_t\leq\theta$ and
$\theta\in[X_t,X_{r,t}]$ after the left-endpoint replacement.
When $X_t=\theta$, either half includes $\theta$ as a closed endpoint,
whatever the model's specified value of $f(\theta)$. In the continuous
case $f(\theta)=0$, neither strict crossing can occur on the good event.
If the observation budget is reached before a crossing, the interval is
unchanged. The initial interval contains $\theta$ by assumption, so
induction proves inclusion in every reached and reported interval.

Within a stage the reported interval is constant. At a completed stage
it changes to the half selected by the sign test. Hence the same event guarantees
coverage after every individual observation. We allow $\tau$ to be any
data-dependent selection of an observation time at which an interval
is reported. Pathwise, failure of coverage for that selected interval
implies failure of coverage for at least one reported interval. Its
probability is therefore at most $\Delta$. If the algorithm has already
stopped, the last interval {ensures} the same conclusion for later
reporting times. This proves both~\eqref{eq:stb-cs} and the statement about arbitrary observation budgets and reporting times.
\label{proof-end:prop:type1-error}
\end{proof}
\begin{remark}
\label{remark:tradeoff}
Recall the discussion of Assumption~\ref{assumption:regularity}.\ref{condition:noise}. If $v$ is chosen larger than the sub-Gaussian norm $\|\varepsilon\|_{\psi_2}$, the right-hand side of $(\ddagger)$ admits a tighter bound, with the improvement depending on the ratio $v/\|\varepsilon\|_{\psi_2}>1$. This {explains} the improvement in coverage described earlier. 
\end{remark}

\subsection{Interval length}
\label{subsec:interval-length-full}

\begin{corollary}[Full version of \textbf{Corollary}~\ref{cor:interval-length}]
\label{cor:interval-length-full}
Under Assumption~\ref{assumption:regularity}, run Algorithm~\ref{alg:stb} on $\mathcal I=[a_0,b_0]\owns\theta$ with its known noise bound $v>0$, fixed $\Delta\in(0,1)$ and $K\geq\lceil\log_2(4/\Delta)\rceil$. Let $R=b_0-a_0$ and let $X_t^\circ$ be the bisection midpoints obtained using the true signs of $f$. On one event of probability at least $1-\Delta$, $\theta\in\mathcal I(n)$ for every total observation budget $n$, including during an unfinished stage, and the following bounds hold simultaneously.
\begin{enumerate}
\item\label{condition:interval-power}
Suppose Assumption~\ref{assumption:power} holds with $\c_0,\gamma>0$ and $0<\eta<\min\{b_0-\theta,\theta-a_0\}$, and $|X_t^\circ-\theta|\geq\lambda R2^{-(t-1)}$ for all $t\geq1$, for fixed $\lambda\in(0,1/2)$. Set $m_\eta=\inf_{x\in\mathcal I:\,|x-\theta|\geq\eta}|f(x)|>0$, $T_n=\lceil\log_2(n+1)/(2\gamma)\rceil+1$ and $B_n(v)=4v\{T_n(K+T_n)\log(n+1)/n\}^{1/2}$. For $n\geq4T_n$ with $B_n(v)<m_\eta$,
\begin{equation}\label{eq:interval-power-explicit}
 |\mathcal I(n)|\leq R(n+1)^{-1/(2\gamma)}
 \vee\lambda^{-1}\{B_n(v)/\c_0\}^{1/\gamma}.
\end{equation}
Let $\normalfont{\textsf{S}}_\gamma=(2\gamma\log2)^{-1}+2/\log3$, $\C_\gamma=\{4\sqrt{8\normalfont{\textsf{S}}_\gamma(1\vee\normalfont{\textsf{S}}_\gamma)}\}^{1/\gamma}$, and let $n_0$ be the smallest integer $n\geq100$ satisfying
\begin{equation}\label{eq:supp-interval-power-start}
 n\geq4\normalfont{\textsf{S}}_\gamma\log(n+1),\qquad
 \frac{16v^2\normalfont{\textsf{S}}_\gamma(1\vee\normalfont{\textsf{S}}_\gamma)
 \log^2(n+1)\{K+\log(n+1)\}}{n}<m_\eta^2.
\end{equation}
These are the choices in~\eqref{eq:power-start-explicit}. For every $n\geq n_0$,
\begin{equation}\label{eq:supp-interval-power-rate}
 |\mathcal I(n)|\leq
 \bigl[R\vee\lambda^{-1}\C_\gamma(v/\c_0)^{1/\gamma}\bigr]
 \Bigl\{\frac{\log^2 n(K+\log n)}{n}\Bigr\}^{1/(2\gamma)}.
\end{equation}
Here $\C_\gamma$ depends only on $\gamma$, $n_0$ depends on $\gamma,v,K,m_\eta$, and the factor $\lambda^{-1}$ is explicit. For fixed $K$, the logarithmic factor is $(\log n)^{3/(2\gamma)}$.

\item\label{condition:interval-jump}
Suppose Assumption~\ref{assumption:jump point} holds with limits $-\mu_-<0<\mu_+$, and $X_t^\circ\neq\theta$ for every $t\geq1$. Set $b=\inf_{x\in\mathcal I\setminus\{\theta\}}|f(x)|>0$ and $\kappa=b\log2/(16\sqrt2\,v)$. Let $n_0$ be the smallest integer $n\geq3$ satisfying~\eqref{eq:supp-interval-jump-start}, as in Theorem~\ref{thm:jump}. Then~\eqref{eq:supp-interval-jump-rate} holds for every $n\geq n_0$:
\begin{align}
 \frac{n}{\log(n+1)}&\geq\frac{256v^2(K+1)^2}{b^2},\qquad
 n\log(n+1)\geq\frac{b^2}{4v^2},\label{eq:supp-interval-jump-start}\\
 |\mathcal I(n)|&\leq R\exp\{-\kappa\sqrt{n/\log n}\}.\label{eq:supp-interval-jump-rate}
\end{align}
Here $n_0$ depends only on $b/v$ and $K$. The midpoint condition holds if $(\theta-a_0)/R$ is not a dyadic rational.
\end{enumerate}
In both cases, choosing $K=\lceil\log_2(4/\Delta)\rceil$ specifies the confidence dependence. 
\end{corollary}

\begin{proof}[Proof of Corollaries~\ref{cor:interval-length} and~\ref{cor:interval-length-full}]
We work on the simultaneous noise event~\eqref{eq:cs-good-event}. Its probability is at least $1-\Delta$, and the proof of Proposition~\ref{prop:type1-error} {establishes} coverage at every observation time. Every completed test has the true sign, so each reached location agrees with $X_t^\circ$. If $J$ tests have completed after $n$ observations, then
\begin{equation}\label{eq:interval-update-count}
 |\mathcal I(n)|=R2^{-J}.
\end{equation}
At most $J+1$ stages have received observations, and their counts sum to $n$. If $J<T$ and $n\geq4T$ for an integer $T\geq1$, some stage $t\leq T$ has received $m=m_t(n)\geq n/T$ observations. The first $m-1$ observations did not cross the threshold, even if the last observation completed the test. Thus
\begin{equation*}
 (m-1)|f(X_t)|
 \leq |S_{m-1,t}|+|S_{m-1,t}-(m-1)f(X_t)|
 \leq2v\sqrt{2(m-1)(K+t)\log m}.
\end{equation*}
Since $m-1\geq n/(2T)$ and $m\leq n$, we obtain
\begin{equation}\label{eq:interval-budget-signal}
 |f(X_t)|\leq4v\sqrt{\frac{T(K+T)\log(n+1)}{n}}.
\end{equation}
This is the noncrossing argument in~\eqref{eq:uncapped-signal-bound}, applied pathwise to the actual counts, including an unfinished test.

For the power case, we use the midpoint condition
\begin{equation}\label{eq:interval-midpoint-separation}
 |X_t^\circ-\theta|\geq\lambda R2^{-(t-1)},\qquad t\geq1,
\end{equation}
and take $T=T_n$. If $J\geq T_n$,~\eqref{eq:interval-update-count} {establishes} the first term of~\eqref{eq:interval-power-explicit}. If $J<T_n$,~\eqref{eq:interval-budget-signal} and $B_n(v)<m_\eta$ imply $|X_t-\theta|<\eta$. The local signal bound and~\eqref{eq:interval-midpoint-separation}, with $t-1\leq J$, {imply}
\begin{equation*}
 \lambda|\mathcal I(n)|
 \leq\lambda R2^{-(t-1)}
 \leq|X_t-\theta|
 \leq\{B_n(v)/\c_0\}^{1/\gamma}.
\end{equation*}
This proves~\eqref{eq:interval-power-explicit}. For $n\geq2$, we have $T_n\leq\normalfont{\textsf{S}}_\gamma\log(n+1)$ and
\begin{equation*}
 B_n(v)^2\leq
 \frac{16v^2\normalfont{\textsf{S}}_\gamma(1\vee\normalfont{\textsf{S}}_\gamma)
 \log^2(n+1)\{K+\log(n+1)\}}{n}.
\end{equation*}
The two inequalities defining~\eqref{eq:supp-interval-power-start} imply $n\geq4T_n$ and $B_n(v)<m_\eta$. They continue to hold for larger $n$: $n/\log(n+1)$ increases, and $\log^2(n+1)\{K+\log(n+1)\}/n$ decreases for $n\geq100$. Applying $\log(n+1)\leq2\log n$ and $K+2\log n\leq2(K+\log n)$ {yields}~\eqref{eq:supp-interval-power-rate} with the displayed $\C_\gamma$.

For the jump case, the two-sided limits and Assumption~\ref{assumption:regularity}.\ref{condition:sign} {imply} $b>0$. We take $T=\lfloor(b/(8v))\sqrt{n/\log(n+1)}\rfloor$. The first inequality in~\eqref{eq:supp-interval-jump-start} {implies}
\begin{equation*}
 \frac b{8v}\sqrt{\frac n{\log(n+1)}}\geq2(K+1),\qquad
 T\geq K,\qquad
 T\geq\frac b{16v}\sqrt{\frac n{\log(n+1)}}.
\end{equation*}
{It follows from the second inequality that} $n\geq(b/(2v))\sqrt{n/\log(n+1)}\geq4T$. Both inequalities hold for every $n\geq n_0$. If $J<T$,~\eqref{eq:interval-budget-signal} therefore implies
\begin{equation*}
 |f(X_t)|\leq4v\sqrt{\frac{2T^2\log(n+1)}{n}}\leq\frac b{\sqrt2}<b,
\end{equation*}
contradicting $X_t=X_t^\circ\neq\theta$ and the definition of $b$. Hence $J\geq T$, and
\begin{equation*}
 |\mathcal I(n)|=R2^{-J}
 \leq R\exp\bigg\{-\frac{b\log2}{16v}\sqrt{\frac n{\log(n+1)}}\bigg\}
 \leq R\exp\{-\kappa\sqrt{n/\log n}\}.
\end{equation*}
We used $\log(n+1)\leq2\log n$ in the last step. Every argument holds on the same noise event for all stated budgets. The two supplementary rate bounds are precisely~\eqref{eq:interval-power-rate} and~\eqref{eq:interval-jump-rate}, which proves both corollaries.
\end{proof}

%% file: appendix_new/multivariate_proof.tex
\section{Multivariate level-set recovery}
\label{appendix:proof:multivariate}

\begin{proof}[Proof of the simultaneous radial bound~\eqref{eq:ray-uniform}]
\label{proof:thm:ray-finite}
{We recall from Section~\ref{subsec:ray-model} that $\mathcal G_i=\sigma(\bm{x}_0,(\bm X_j,Y_j)_{1\leq j\leq i})$, that the design $\bm X_i$ is $\mathcal G_{i-1}$-measurable, and that the centered errors satisfy the conditional bound $\mathbb E\{\exp(\lambda\varepsilon_i)\mid\mathcal G_{i-1}\}\leq\exp(\lambda^2v^2/2)$ for every real $\lambda$, as in $
\mathbb{E}\{\exp(\lambda\varepsilon_i)\mid\mathcal G_{i-1}\}
 \leq \exp(\lambda^2v^2/2),\qquad \lambda\in\mathbb R$. Algorithm~\ref{alg:ray-stb} allocates the deterministic budget $m_k$ and uses $K=\lceil\log_2(4M/\Delta)\rceil$ on each ray. We fix a ray $k$ and a reached stage $t$, and condition on the observation history at its starting time.} This starting time is a stopping time for $(\mathcal G_i)$ in Section~\ref{subsec:ray-model}. Within the stage the design point stays fixed. For the probability calculation, we set the stage errors to zero after its test or ray budget ends. The indicator that a further observation belongs to the stage is measurable before that observation. {For the extended row sum $W_{j,k,t}$ followed by Chernoff's inequality with the threshold $T(j,e^{-(K+t)})=v\sqrt{2j(K+t)\log(j+1)}$ from~\eqref{eq:boundary}, {implies}, for $j\geq1$, $\lambda\in\mathbb R$ and either sign,}
\begin{equation*}
{\begin{aligned}
\mathbb E\{\exp(\lambda W_{j,k,t})\mid\text{history at stage start}\}
&\leq\exp(j\lambda^2v^2/2),\\
\mathbb P\{\pm W_{j,k,t}>T(j,e^{-(K+t)})\mid\text{history at stage start}\}
&\leq(j+1)^{-(K+t)}.
\end{aligned}}
\end{equation*}
{Conditioning on each possible starting time {yields} the same bounds at the random stage start.}
We sum over both signs, all $j\geq1$, and all reached stages, applying the tower property at every start. The probability of any row error exceeding its threshold on ray $k$ is at most
\begin{equation*}
2\sum_{t=1}^{\infty}\sum_{j=1}^{\infty}(j+1)^{-(K+t)}
\leq2^{2-K}\leq\Delta/M.
\end{equation*}
A union bound over rays {establishes} a single event of probability at least $1-\Delta$ on which these inequalities hold for every observed row length. {Assumption~\ref{ass:star}.\ref{condition:radial-crossing} states that $g_{u_k}(r)=f(\bm x_0+ru_k)$ has its unique crossing at $\rho_\star(u_k)\in[r_-,r_+]\subset(0,R)$ and that $(r-\rho_\star(u_k))g_{u_k}(r)>0$ off the root. Thus a correct completed sign specifies which half of $[0,R]$ contains the root, and the simultaneous noise bounds preserve it after every interval update.} The argument permits dependence between rays and between consecutive errors subject to the conditional moment generating function bound.

On this event, the deterministic argument of Theorem~\ref{thm:power} applies to each ray at its budget $m_k$. Indeed, we put $T_j=\lceil\log_2(j+1)/(2\gamma)\rceil+1$. If at least $T_j$ stages have completed by budget $j$, the current interval has width at most $R(j+1)^{-1/(2\gamma)}$. Otherwise, at most $T_j$ stages have observations, and the stage with the largest observed count has at least $j/T_j$ of them. Its sum had not crossed at the preceding observation. Subtracting its noise sum, whose absolute value is bounded by the same threshold, {yields} the following bound at the selected scalar radius $X_{t_j}$:
\begin{equation*}
|g_{u_k}(X_{t_j})|
\leq4v\Bigl\{\frac{T_j(K+T_j)\log(j+1)}j\Bigr\}^{1/2},\qquad j\geq4T_j.
\end{equation*}
{Here Assumption~\ref{ass:star}.\ref{condition:radial-separation} {requires the global bound} $|g_u(r)|\geq\c|r-\rho_\star(u)|^\gamma$ with the same $\c,\gamma>0$ for every direction and every $r\in[0,R]$. Hence $|r-\rho_\star(u)|\leq\{|g_u(r)|/\c\}^{1/\gamma}$, converting the signal bound into a root-error bound without a direction-dependent neighborhood condition.} Projection onto the current interval can only reduce that error because the interval contains the root. The counts sum to $j$ even when the last stage is unfinished. Thus~\eqref{eq:power-start-explicit}, with its local-separation starting condition omitted, {implies}
\begin{equation*}
|\widehat\rho_k-\rho_\star(u_k)|
\leq\widetilde{\C}_{1}\Bigl\{\frac{\log^2m_k(K+\log m_k)}{m_k}\Bigr\}^{1/(2\gamma)},
\qquad m_k\geq m_\gamma.
\end{equation*}
For $j\geq30$, the function $\log^2j(K+\log j)/j$ decreases, since its logarithmic derivative with respect to $\log j$ is $2/\log j+1/(K+\log j)-1<0$. Also $j/\log(j+1)$ increases, so the defining inequality for $m_\gamma$ holds for all $j\geq m_\gamma$. Since $m_k\geq m$, the last bound implies~\eqref{eq:ray-uniform}. Both $m_\gamma$ and $\widetilde{\C}_{1}$ are independent of $K$, hence of $M$ and $\Delta$.
\label{proof-end:thm:ray-finite}
\end{proof}

\begin{lemma}[Local polynomial reconstruction and radial sets]
\label{lem:radial-regression}
For the grids, charts and constants in Section~\ref{subsec:ray-reconstruction}, every $H_a(z)$ has smallest eigenvalue at least $\lambda_0/2$. Its intercept weights have sum of absolute values at most $\widetilde{\C}_{3}$. For arbitrary responses $a_1,\ldots,a_M\in[0,R]$, let $a_h$ be the fitted function obtained from~\eqref{eq:ray-regression}--\eqref{eq:ray-chart-weights} with these responses. Then $a_h$ is continuous, takes values in $[0,R]$, and, under~\eqref{eq:ray-holder},
\begin{equation}\label{eq:ray-radial-comparison}
\|a_h-\rho_\star\|_\infty\leq\widetilde{\C}_{3}\max_k|a_k-\rho_\star(u_k)|+\widetilde{\C}_{4}Lh^{\mathfrak b}.
\end{equation}
The filled radial set defined by $a_h$ is compact, its boundary is its radial graph, and the two Hausdorff distances in~\eqref{eq:ray-hausdorff} are at most this uniform function error.
\end{lemma}
\begin{proof}[Proof of Lemma~\ref{lem:radial-regression}]
{Recall the construction in Section~\ref{subsec:ray-reconstruction}: $\mathfrak d=d-1$, $p=\lceil\mathfrak b\rceil-1$, $v_p(w)=(w^\iota)_{|\iota|\leq p}$ contains $P=\binom{\mathfrak d+p}{p}$ monomials, and $W(w)=\prod_{i=1}^{\mathfrak d}(1-|w_i|)_+$. The matrix $H_0$ in~\eqref{eq:ray-kernel} is their weighted integral Gram matrix, with $\lambda_0=\lambda_{\min}(H_0)$. For a nonzero coefficient vector $c$, the polynomial $c^\top v_p(w)$ cannot vanish on an open cube.} As $W$ is strictly positive on $(-1,1)^{\mathfrak{d}}$, its squared integral with weight $W$ is positive. Hence $H_0$ is positive definite. Each entry $W(w)w^{\iota+\iota'}$, extended by zero outside $[-1,1]^{\mathfrak{d}}$, is globally Lipschitz in each coordinate with constant at most $1+2p$. For a shifted lattice of mesh $1/Q$, we compare its Riemann sum with the integral, using cubes of side $1/Q$ centered at the lattice points. The relevant cubes are contained in $[-2,2]^{\mathfrak{d}}$, and the variation within a cube is at most $\mathfrak{d}(1+2p)/Q$. Thus every entry has approximation error at most $4^{\mathfrak{d}} \mathfrak{d}(1+2p)/Q$. The operator norm error is at most $P$ times this bound.

For $z\in[-A,A]^{\mathfrak{d}}$ and $h\leq1$, all grid points with $W((z_k-z)/h)>0$ are in $D$. The finite sum defining $H_a(z)$ is therefore the full shifted-lattice sum. {The choice in~\eqref{eq:ray-grid-constants} ensures $P4^{\mathfrak d}\mathfrak d(1+2p)/Q\leq\lambda_0/2$, and therefore}
\begin{equation*}
\|H_a(z)-H_0\|_{\mathrm{op}}\leq\lambda_0/2,\qquad
 \|H_a(z)^{-1}\|_{\mathrm{op}}\leq2/\lambda_0,
\end{equation*}
uniformly over $a,z,h$. {Writing $w_k(z)=(z_k-z)/h$ and $e_0$ for the constant-monomial coordinate vector, the intercept weights in~\eqref{eq:ray-regression} are $l_{a,k}(z)=Q^{-\mathfrak d}W(w_k(z))e_0^\top H_a(z)^{-1}v_p(w_k(z))$. We use} $\|v_p(w)\|\leq\sqrt P$ on the support of $W$ and at most $(2Q+1)^{\mathfrak{d}}$ nonzero summands to obtain
\begin{equation*}
\sum_{k\in I_a}|l_{a,k}(z)|\leq
 \frac{2\sqrt P}{\lambda_0}Q^{-\mathfrak{d}}(2Q+1)^{\mathfrak{d}}\leq\widetilde{\C}_{3}.
\end{equation*}
In each coordinate the grid has at most $2(A+2)Q/h+1$ points, proving $M\leq\widetilde{\C}_{2}h^{-\mathfrak{d}}$. The number of grid points in the fixed cube $D$ is also bounded below by a positive constant times $h^{-\mathfrak{d}}$, proving $M\asymp h^{-\mathfrak{d}}$. All constants are independent of the responses.

{We next use the precise smoothness condition in Assumption~\ref{ass:star}.\ref{condition:radial-holder}: for $F_a=\rho_\star\circ\Phi_a$ and $\vartheta=\mathfrak b-p\in(0,1]$, its derivatives of total order $p$ satisfy $|D^\iota F_a(z)-D^\iota F_a(z')|\leq L\|z-z'\|_\infty^\vartheta$ on $D$, as stated in~\eqref{eq:ray-holder}. The same condition is H\"older continuity of $F_a$ when $p=0$, and Lipschitz continuity of the order-$p$ derivatives when $\vartheta=1$.} For $p\geq1$, we use $T_{a,z}$ to denote the Taylor polynomial of $F_a$ at $z$ of degree $p$. Applying one-dimensional Taylor's formula to $t\mapsto F_a(z+t(z_k-z))$ and subtracting its order-$p$ term {yields} an integral of differences of the order-$p$ derivatives. The H\"older bound and the identity $\sum_{|\iota|=p}1/\iota!=\mathfrak{d}^p/p!$ imply
\begin{equation*}
|F_a(z_k)-T_{a,z}(z_k)|
 \leq\frac{\mathfrak{d}^p}{p!}L\|z_k-z\|_\infty^{p+\vartheta}
 \leq\frac{\mathfrak{d}^p}{p!}Lh^{\mathfrak b}
\end{equation*}
for every point with nonzero kernel weight. The integral factor is bounded by $1/p!$ after the multinomial coefficients are included. For $p=0$, we set $T_{a,z}=F_a(z)$ and the same inequality follows directly from H\"older continuity, with $\mathfrak{d}^0/0!=1$.

Weighted least squares reproduces every polynomial of degree at most $p$, since its normal matrix is invertible. In this proof, $\widehat F_a$ denotes the fit to the responses $a_k$. The exact decomposition is
\begin{equation}\label{eq:ray-regression-decomposition}
\begin{aligned}
\widehat F_a(z)-F_a(z)
&=\sum_{k\in I_a}l_{a,k}(z)\{a_k-F_a(z_k)\}\\
&\quad+\sum_{k\in I_a}l_{a,k}(z)\{F_a(z_k)-T_{a,z}(z_k)\}.
\end{aligned}
\end{equation}
The first term is bounded by $\widetilde{\C}_{3}\max_k|a_k-\rho_\star(u_k)|$ and the second by $\widetilde{\C}_{4}Lh^{\mathfrak b}$. Since $F_a(z_a(u))=\rho_\star(u)$ wherever $\psi_a(u)>0$, the same bound holds after averaging the chart fits. Projection onto $[0,R]$ does not increase error about $\rho_\star(u)\in[0,R]$. This proves~\eqref{eq:ray-radial-comparison} for arbitrary response errors, including their biases.

The local fits are continuous because $W$ is continuous and their Gram matrices are uniformly invertible. They are bounded uniformly in $z$ by $\widetilde{\C}_{3}\max_k|a_k|$. Consequently each chart term tends to zero when its weight tends to zero, and the combined, projected function $a_h$ is continuous. For any continuous $a:\mathbb S^{d-1}\to[0,R]$, the associated filled radial set is the continuous image of $\mathbb S^{d-1}\times[0,1]$ under $(u,t)\mapsto \bm{x}_0+ta(u)u$, and is compact. At positive radius, a point below the radial graph is interior by continuity, and one above it is exterior. A point on the graph is on the boundary. The center is interior if $\min_u a(u)>0$, and is a boundary point on the radial graph otherwise. This proves the boundary identity, including possible zero radii.

We match points on two radial graphs in the same direction to bound their distance by $\|a-\rho_\star\|_\infty$. For the filled sets, we match a radius $r$ to the smaller of $r$ and the other set's radius in that direction. Applying this construction in both directions proves the bounds on both Hausdorff distances.
\end{proof}

\begin{proof}[Proof of Theorem~\ref{thm:ray-reconstruction}]
\label{proof:thm:ray-reconstruction}
{Recall that~\eqref{eq:ray-chart-weights} forms $\widehat\rho_h$ by averaging the local fits with nonnegative weights $\psi_a$ summing to one and projecting onto $[0,R]$. Assumption~\ref{ass:star}.\ref{condition:radial-crossing} puts the true radius in $[r_-,r_+]\subset[0,R]$, so this projection cannot increase its error. The output of each ray also belongs to $[0,R]$.} \par On the event~\eqref{eq:ray-uniform}, Lemma~\ref{lem:radial-regression} with $a_k=\widehat\rho_k$ {establishes}~\eqref{eq:ray-hausdorff}. {The simultaneous bound applies to the root error on every ray.} Random stage counts and the unfinished last test have already been included in that event, while each ray receives its deterministic budget $m_k$.

{For the bandwidth in~\eqref{eq:ray-bandwidth}, recall that $\ell_{n,\Delta}=\log^2(n+1)\log(4n/\Delta)$ and $h=(\ell_{n,\Delta}/n)^{1/(2\gamma\mathfrak b+\mathfrak d)}$. The grid satisfies $M\leq\widetilde{\C}_2h^{-\mathfrak d}$, and $
m_\gamma=\min\{j\in\mathbb N:j\geq30,\ j\geq4A_\gamma\log(j+1)\}.$is the common starting budget for the radial bound. The first condition $\ell_{n,\Delta}\leq n$ {implies} $0<h\leq1$.} {It follows from the second condition that} $n/M\geq nh^{\mathfrak{d}}/\widetilde{\C}_{2}\geq2m_\gamma$. Consequently
\begin{equation*}
m=\lfloor n/M\rfloor\geq\frac{n}{2M}
 \geq\frac{nh^{\mathfrak{d}}}{2\widetilde{\C}_{2}}\geq m_\gamma.
\end{equation*}
The construction uses exactly $n$ observations and satisfies the per-ray starting condition. {Since $M,m\leq n$ and $0<\Delta<1$, the choice $K=\lceil\log_2(4M/\Delta)\rceil$ and~\eqref{eq:ray-uniform} {yield}}
\begin{equation*}
{\begin{aligned}
K+\log m&\leq1+\frac{\log(4n/\Delta)}{\log2}+\log n
 \leq(2+1/\log2)\log(4n/\Delta),\\
\epsilon_{m,M,\Delta}&\leq\widetilde{\C}_{1}
 \Bigl\{\frac{2\widetilde{\C}_{2}(2+1/\log2)\ell_{n,\Delta}}{nh^{\mathfrak d}}\Bigr\}^{1/(2\gamma)}.
\end{aligned}}
\end{equation*}
The identity $h^{2\gamma\mathfrak b+\mathfrak{d}}=\ell_{n,\Delta}/n$ makes this bound equal to $\widetilde{\C}_{1}\{2\widetilde{\C}_{2}(2+1/\log2)\}^{1/(2\gamma)}h^{\mathfrak b}$. Multiplying by $\widetilde{\C}_{3}$ and adding $\widetilde{\C}_{4}Lh^{\mathfrak b}$ proves~\eqref{eq:ray-global-rate}, with all dependence on $\Delta$ explicit.

For fixed $\Delta$ and fixed model parameters, $\ell_{n,\Delta}\asymp\log^3 n$, $h\to0$, and $nh^{\mathfrak{d}}\to\infty$, so the budget conditions eventually hold. The same finite-sample statement applies to any sequence $\Delta_n$ satisfying the displayed conditions. In particular, $\log(1/\Delta_n)=o(n/\log^2 n)$ ensures $h\to0$ and eventual feasibility. The probability statement is for each selected confidence level.
\label{proof-end:thm:ray-reconstruction}
\end{proof}

%% file: appendix_new/pj_proof.tex
\section{PJ averaging under signal separation}
\label{SM:pj-rates}

\begin{figure}[!t]
\normalfont
\centering
\resizebox{\linewidth}{!}{\input{fig/rate_exponents.tex}}
\caption{Absolute-error exponents, omitting logarithmic factors. Coordinates compress $\gamma$ and $r$ by $z\mapsto z/(1+z)$. PJ uses fixed $s=0.6,0.8$, or $s=1$ with $a=R=\c_f=1$. Its curves are upper-bound exponents. SA uses $a_n=1/n$ under the conditions below. The separate STB point at $\gamma=0$ denotes a pointwise super-polynomial rate and has no finite exponent.}
\label{fig:rate-exponents}
\end{figure}

Figure~\ref{fig:rate-exponents} compares the polynomial exponents of the bounds in Theorems~\ref{thm:pj-upper} and~\ref{thm:power}, together with the minimax exponent in Theorem~\ref{thm:minimax}. The jump behavior follows the separate super-polynomial result in Theorem~\ref{thm:jump}. The plotted bounds organize the theoretical regimes.

We specify Polyak--Juditsky (PJ) averaging~\citep{polyak_acceleration_1992}. For a deterministic $X_1\in\mathcal I$, $a>0$ and $s>0$, we define
\begin{equation}\label{eq:pj-projected}X_{k+1}=\Pi_{\mathcal I}\{X_k-a_kY_k\},\qquad
 a_k=a k^{-s},\qquad Y_k=f(X_k)+\varepsilon_k,\qquad
 \overline X_n=\frac1n\sum_{k=1}^n X_k.
\end{equation}
Each iteration uses one fresh observation. For each fixed regression function $f$ satisfying Assumptions~\ref{assumption:regularity} and~\ref{assumption:power}, there is a constant $\c_f>0$ such that
\begin{equation}\label{eq:pj-global-signal}
 (x-\theta)f(x)\geq \c_f|x-\theta|^{\gamma+1},\qquad x\in\mathcal I.
\end{equation}
Indeed, if we denote $m_\eta=\inf_{x\in\mathcal I:|x-\theta|\geq\eta}|f(x)|$, we may take $\c_f=\min\{\c_0,m_\eta/R^\gamma\}$. If the exterior set is empty, $\c_0$ is sufficient. Assumption~\ref{assumption:regularity}.\ref{condition:sign} implies $m_\eta>0$. Under the jump assumption, the same argument {establishes}~\eqref{eq:pj-global-signal} with $\gamma=0$, interpreted as $(x-\theta)f(x)\geq \c_f|x-\theta|$. We use $\gamma=0$ for this jump case (Section~\ref{subsection:discontinuous regression function}) below. The constants may depend on $f$, in particular on its separation away from the root.

\begin{theorem}\label{thm:pj-upper}
Suppose Assumption~\ref{assumption:regularity} and the global signal inequality~\eqref{eq:pj-global-signal} hold. For the update~\eqref{eq:pj-projected} 
\begin{enumerate}
\item 
If $a>0$ and $0<s<1$, the mean absolute error satisfies
\begin{equation}\label{eq:pj-upper-rates}
\mathbb{E}|\overline X_n-\theta|=\begin{cases}
 O(n^{-s/(\gamma+1)}),&0\leq\gamma\leq1,\\
 O\!\bigl(n^{-\{s/(\gamma+1)\wedge(1-s)/(\gamma-1)\}}\bigr),&\gamma>1.
 \end{cases}
\end{equation}
    \item If $s=1$ and $\gamma>1$, the mean absolute error is $O((\log(n+1))^{-1/(\gamma-1)})$.
    \item If $s=1$ and $0\leq\gamma\leq1$, we set $\lambda=2a\c_fR^{\gamma-1}$. Then
\begin{equation}\label{eq:pj-harmonic-upper}
\mathbb{E}|\overline X_n-\theta|^{\gamma+1}
 =\begin{cases}
 O(n^{-\lambda}),&\lambda<1,\\
 O(n^{-1}\log^2(n+1)),&\lambda=1,\\
 O(n^{-1}\log(n+1)),&\lambda>1.
 \end{cases}
\end{equation}

\end{enumerate}
\end{theorem}

\begin{remark}
 For $\gamma>1$, the largest upper-bound exponent is $1/(2\gamma)$, obtained at $s_\star=(\gamma+1)/(2\gamma)$. Selecting this step-size exponent requires knowledge of the signal strength $\gamma$. 
\end{remark}

Theorem~\ref{thm:pj-upper} then {implies} exponent $s/(\gamma+1)$ for $0\leq\gamma\leq1$, and $s/(\gamma+1)\wedge(1-s)/(\gamma-1)$ for $\gamma>1$, when $s<1$. For $s=1$ and $0\leq\gamma\leq1$, the mean absolute-error upper-bound exponent is $(\lambda\wedge1)/(\gamma+1)$, with the logarithmic factors in~\eqref{eq:pj-harmonic-upper}. The plotted choice $a=R=\c_f=1$ {corresponds to} $\lambda=2$ and exponent $1/(\gamma+1)$. Beyond this point, exponent zero represents the logarithmic decay $(\log n)^{-1/(\gamma-1)}$. The nonzero-noise lower bound in Proposition~\ref{prop:pj-flat-lower} matches the flat-root upper bound when $s>(\gamma+1)/(2\gamma)$. 

For SA, Theorem~2.1 of~\cite{kaniovski_non-standard_1995}, under the conditions stated in Section~\ref{subsection:differentiable-regimes}, {establishes} exponent $1/(1+\gamma)$ for $1/2<\gamma<1$. We show the positive-slope value $1/2$ for $f(x)=x-\theta$, and the logarithmic flat-root rate under Assumption~8.1 and Theorem~8.2 of~\cite{ljung_stochastic_1992}.No SA curve is drawn for $0<\gamma\leq1/2$, which lies outside that theorem. At $\gamma=0$, Theorem~3.1 of~\cite{lim_convergence_2011} {implies} $\mathbb{E}|X_{n+1}-\theta|=O(n^{-1})$.

The \textsc{STB} exponents follow from Theorems~\ref{thm:power} and~\ref{thm:jump}. For the minimax comparison in the Gaussian experiment, we impose $|f(x)|\geq\c_0|x-\theta|^\gamma$ throughout the common interval and $\sup_{\mathcal I}|f|\leq\C_\star$, with fixed $\c_0>0$ and $\C_\star\geq\c_0(2R)^\gamma$. The bound in Theorem~\ref{thm:power} holds uniformly over this class. The alternatives in~\eqref{eq:minimax-alternatives} belong to this class by~\eqref{eq:minimax-warp-bounds}, so the lower-bound proof applies. Hence, the minimax rate over this class has polynomial exponent $1/(2\gamma)$ for every $\gamma>0$, up to logarithmic factors. 

Before we present the proof, we recall (the expectation version of) Minkowski's inequality~\citep{durrett2019probability}.
\begin{lemma}[Minkowski's inequality]\label{lemma:mink}
Let $n\geq 1$ be an integer, and let $1\leq p<\infty$.
Suppose that $\normalfont{\textsf{X}}_1,\ldots,\normalfont{\textsf{X}}_n$ are real-valued random variables satisfying
$\mathbb{E}|\normalfont{\textsf{X}}_i|^p<\infty$ for every
$i=1,\ldots,n$. Then
\begin{equation}\label{eq:minkowski}
\left(
    \mathbb{E}\left|\sum_{i=1}^{n}\normalfont{\textsf{X}}_i\right|^p
\right)^{1/p}
\leq
\sum_{i=1}^{n}\left(\mathbb{E}|\normalfont{\textsf{X}}_i|^p\right)^{1/p}.
\end{equation}
\end{lemma}

\begin{proof}[Proof of Theorem~\ref{thm:pj-upper}]
We write $e_k=X_k-\theta$, $u_k=\mathbb{E} e_k^2$ and $M=\C_\star^2+\sigma^2$. Since projection cannot increase distance to $\theta$, we have
\begin{equation*}
u_{k+1}\leq u_k-2a_k\mathbb{E}(e_kY_k)+a_k^2\mathbb{E}Y_k^2.
\end{equation*}
{From independence of $\varepsilon_k$ from the past, centering and $Y_k=f(X_k)+\varepsilon_k$, we obtain} $\mathbb{E}Y_k^2\leq M$ and, by~\eqref{eq:pj-global-signal},
\begin{equation*}
\mathbb{E}(e_kY_k)=\mathbb{E}\{e_kf(X_k)\}\geq\c_f\mathbb{E}|e_k|^{\gamma+1}.
\end{equation*}
{Combining these bounds, we obtain}
\begin{equation}\label{eq:pj-energy}
 u_{k+1}\leq u_k-2\c_fa_k\mathbb{E}|e_k|^{\gamma+1}+Ma_k^2.
\end{equation}
\textbf{Case 1: $\gamma>1$.}\\ Jensen's inequality bounds the middle moment below by $u_k^p$, where $p=(\gamma+1)/2>1$. If $0<s<1$, we put $r=\min\{s/p,(1-s)/(p-1)\}$. Comparison with $Dk^{-r}$ {shows that}
\begin{equation}\label{eq:pj-second-moment}
 u_k=O(k^{-r}).
\end{equation}
Here are the details of the comparison. For large $k$, the map $u\mapsto u-2\c_f ak^{-s}u^p$ is increasing on the bounded range of $u_k$. The inequalities $s+rp\leq2s$ and $s+rp\leq r+1$, together with
\begin{equation*}
 D\{k^{-r}-(k+1)^{-r}\}\leq rDk^{-r-1},
\end{equation*}
show that the drift at $Dk^{-r}$ dominates both this decrement and $Ma^2k^{-2s}$ when $D$ is large enough. We can start the induction at a large $N$ with $D=R^2N^r$, since $u_N\leq R^2$ and this choice makes $D$ arbitrarily large. The comparison sequence then stays within $[0,R^2]$. If $s=1$, the same comparison with $D(\log(k+1))^{-1/(p-1)}$ works: its decrement is asymptotic to $D(p-1)^{-1}k^{-1}(\log k)^{-p/(p-1)}$, the drift has the same logarithmic power, and the $k^{-2}$ term is smaller. Minkowski's inequality (Lemma~\ref{lemma:mink}) now {implies}
\begin{equation*}
 \{\mathbb{E}|\overline X_n-\theta|^2\}^{1/2}
 \leq\frac1n\sum_{k=1}^n\sqrt{u_k}.
\end{equation*}
For $0<s<1$, $r/2=\{s/(\gamma+1)\}\wedge\{(1-s)/(\gamma-1)\}<1$, therefore $n^{-1}\sum_{k=1}^n\sqrt{u_k}=O(n^{-r/2})$ by~\eqref{eq:pj-second-moment}, which proves~\eqref{eq:pj-upper-rates}. For $s=1$, splitting the sum at $\lfloor\sqrt n\rfloor$ {produces} an $O(n^{-1/2})$ contribution from the first part and $O((\log n)^{-1/(\gamma-1)})$ from the second.

\textbf{Case 2: $0\leq\gamma\leq1$.}\\
We first note that $|e_k|^{\gamma+1}\geq R^{\gamma-1}e_k^2$. With $b=\c_f R^{\gamma-1}$, we therefore have
\begin{equation*}
 u_{k+1}\leq(1-2ba_k)u_k+Ma_k^2.
\end{equation*}
For $0<s<1$, comparison with $Dk^{-s}$ {shows that} $u_k=O(k^{-s})$. For $s=1$, iteration of this linear inequality, starting beyond the finite set where its coefficient might be negative, {implies} $u_k=O(k^{-\lambda})$ for $\lambda<1$, $O(k^{-1}\log(k+1))$ for $\lambda=1$, and $O(k^{-1})$ for $\lambda>1$. To see this directly, the product of the coefficients from $j$ to $k$ is bounded by $\C(j/k)^\lambda$, and the noise terms sum as $\C k^{-\lambda}\sum_{j\leq k}j^{\lambda-2}$.

Dividing~\eqref{eq:pj-energy} by $a_k$ and summing {yields}
\begin{equation}\label{eq:pj-ergodic-energy}
 2\c_f \sum_{k=1}^n\mathbb{E}|e_k|^{\gamma+1}
 \leq\frac{u_1}{a_1}
 +\sum_{k=2}^n u_k(a_k^{-1}-a_{k-1}^{-1})
 +M\sum_{k=1}^n a_k.
\end{equation}
We have dropped the nonpositive terminal term. For $0<s<1$, the middle sum is $O(\log(n+1))$ and the last sum is $O(n^{1-s})$. Convexity of $x\mapsto|x|^{\gamma+1}$ then {implies} $\mathbb{E}|\overline X_n-\theta|^{\gamma+1}=O(n^{-s})$. For $s=1$, the middle sum is $a^{-1}\sum_{k=2}^n u_k$ and the last sum is $O(\log(n+1))$. Substitution of the three bounds for $u_k$ proves~\eqref{eq:pj-harmonic-upper}.
\end{proof}

\begin{proposition}[Positive-slope setting]\label{prop:pj-smooth}
Under Assumption~\ref{assumption:regularity}, suppose additionally that $f(\theta)=0$ and $f(\theta+h)=\beta h+O(h^2)$ for some $\beta>0$. For the update~\eqref{eq:pj-projected} with any $a>0$ and $1/2<s<1$,
\begin{equation}\label{eq:pj-projected-clt}
 \sqrt n(\overline X_n-\theta)
 =-\frac1{\beta\sqrt n}\sum_{k=1}^n\varepsilon_k+o_P(1)
 \ \overset{d}{\longrightarrow}\ \mathcal{N}(0,\sigma^2/\beta^2).
\end{equation}
\end{proposition}

\begin{remark}
This is the classical PJ limit~\citep[Theorem~2]{polyak_acceleration_1992}. For completeness and to facilitate comparison with the proof of our CLT result, we include a brief proof of Proposition~\ref{prop:pj-smooth} here. A subtle but important point is that this argument requires a second-order remainder bound, as stated in the proposition.
\end{remark}

\begin{proof}[Proof of Proposition~\ref{prop:pj-smooth}]
We use $e_k=X_k-\theta$ and $u_k=\mathbb{E}e_k^2$ as in the preceding proof.
The local expansion and Assumption~\ref{assumption:regularity}.\ref{condition:sign} {imply} $e f(\theta+e)\geq \c e^2$ on $\mathcal I$ for some $\c>0$. The proof of Theorem~\ref{thm:pj-upper} therefore {establishes} $u_k=O(k^{-s})$. We first justify removal of projection. With $\mathcal G_k=\sigma(X_1,Y_1,\ldots,Y_k)$, the deterministic inputs fixed, and $X_k$ measurable with respect to $\mathcal G_{k-1}$, the conditional version of~\eqref{eq:pj-energy} shows that $
 e_k^2+M\sum_{j=k}^{\infty}a_j^2\geq 0$
is a nonnegative supermartingale with respect to $(\mathcal G_{k-1})_{k\geq1}$ and hence converges almost surely. Summing its expected decreases also {shows that} $\sum_k a_ke_k^2<\infty$ almost surely. Since $\sum_k a_k=\infty$, we conclude that $e_k\to0$ almost surely. These uses of supermartingale convergence are covered by~\cite[Theorem~4.2.12]{durrett2019probability}. Moreover, $a_k\varepsilon_k\to0$ almost surely by Chebyshev's inequality and the Borel--Cantelli lemma, because $\sum_k a_k^2<\infty$. The root is interior and $f$ is bounded, so projection is inactive after an almost surely finite time.

We write $r_k=f(X_k)-\beta e_k$ and $p_k=X_{k+1}-X_k+a_k(f(X_k)+\varepsilon_k)$. Boundedness of $f$ outside the neighborhood of its expansion {ensures} a finite constant $L_f$ with $|r_k|\leq L_fe_k^2$ throughout $\mathcal I$. {Summing the update identity, we obtain}
\begin{equation}\label{eq:pj-linear-decomposition}
 \beta\sum_{k=1}^n e_k
 =-\sum_{k=1}^n\varepsilon_k
 +\frac{e_1}{a_1}-\frac{e_{n+1}}{a_n}
 +\underbrace{\sum_{k=2}^n e_k(a_k^{-1}-a_{k-1}^{-1})}_{:=\normalfont{\textsf{T}}_1}
 -\underbrace{\sum_{k=1}^n r_k}_{:=\normalfont{\textsf{T}}_2}+\underbrace{\sum_{k=1}^n\frac{p_k}{a_k}}_{:=\normalfont{\textsf{T}}_3}.
\end{equation}
The two boundary terms and $\normalfont{\textsf{T}}_1$ have total $L^2$ norm $O(n^{s/2})$. Indeed, $u_k=O(k^{-s})$ and $a_k^{-1}-a_{k-1}^{-1}=O(k^{s-1})$ {imply}
\begin{equation*}
\bigg\|\frac{e_1}{a_1}-\frac{e_{n+1}}{a_n}+\normalfont{\textsf{T}}_1\bigg\|_2
 \leq \C\bigg(1+n^{s/2}+\sum_{k=2}^n k^{s-1}k^{-s/2}\bigg)=O(n^{s/2}).
\end{equation*}
The remainder satisfies $\mathbb{E}|\normalfont{\textsf{T}}_2|\leq L_f\sum_{k=1}^n u_k=O(n^{1-s})$. The projection term $\normalfont{\textsf{T}}_3$ is eventually a finite random constant because $p_k=0$ after an almost surely finite time. Dividing by $\sqrt n$, the endpoint terms and $\normalfont{\textsf{T}}_1$ vanish in $L^2$, $\normalfont{\textsf{T}}_2$ vanishes in $L^1$, and $\normalfont{\textsf{T}}_3$ vanishes almost surely, since $1/2<s<1$.

The common sub-Gaussian bound implies $\sup_k\mathbb{E}\varepsilon_k^4<\infty$. Thus, for every $\delta>0$,
\begin{equation*}
\frac1n\sum_{k=1}^n\mathbb{E}\{\varepsilon_k^2\mathbf1_{\{|\varepsilon_k|>\delta\sqrt n\}}\}
 \leq\frac{\sup_k\mathbb{E}\varepsilon_k^4}{\delta^2n}\longrightarrow0.
\end{equation*}
The independent errors therefore satisfy the Lindeberg condition and have total variance $n\sigma^2$. The Lindeberg--Feller theorem~\citep[Theorem~3.4.10]{durrett2019probability} proves~\eqref{eq:pj-projected-clt}.
\end{proof}

\begin{proposition}\label{prop:pj-flat-lower}
The parameters satisfy $\gamma>1$, $a>0$ and $s_{\star}<s\leq1$, where $s_{\star}=(\gamma+1)/(2\gamma)$. There are a compact interval, a deterministic positive initialization, a function $f(x)=\beta{\textsf{sign}}(x)|x|^\gamma$ with root zero and $\beta>0$, and i.i.d.\ centered bounded errors of positive variance satisfying Assumption~\ref{assumption:regularity}, for which~\eqref{eq:pj-projected} satisfies, for all sufficiently large $n$,
\begin{equation}\label{eq:pj-flat-lower}
 \mathbb P\{\overline X_n\geq \c h_n\}\geq p_0>0,
 \qquad\mathbb{E}|\overline X_n|\geq \c p_0h_n,
 \qquad h_n=\bigl(1+\sum_{k=1}^{n-1}a_k\bigr)^{-1/(\gamma-1)}.
\end{equation}
Here $\c ,p_0$ are fixed positive constants and $h_n\asymp n^{-(1-s)/(\gamma-1)}$ for $s<1$ and $h_n\asymp(\log n)^{-1/(\gamma-1)}$ for $s=1$.
\end{proposition}

\begin{remark}Our main goal is to show that PJ averaging is not adaptive across all $\gamma>0$. To this end, we focus on the case $s_\star<s\leq 1$. In this regime, the lower bound establishes the tightness of our analysis of PJ averaging in the presence of nonzero noise. For the remaining regimes, we conjecture that the bounds in Theorem~\ref{thm:pj-upper} are also tight.
\end{remark}

\begin{proof}[Proof of Proposition~\ref{prop:pj-flat-lower}]
We take $\theta=0$ and errors $\varepsilon_k=\sigma\xi_k$, where the $\xi_k$ are independent Rademacher variables and the positive constant $\sigma$ will be chosen sufficiently small. We put $r=1/(\gamma-1)$. The sequence $h_n$ in~\eqref{eq:pj-flat-lower} satisfies
\begin{equation}\label{eq:pj-barrier-speed}
 \frac{h_n-h_{n+1}}{a_nh_n^\gamma}\longrightarrow r.
\end{equation}
We choose $x_1>0$ small enough that $a\beta x_1^{\gamma-1}<1/2$. The noiseless iteration $x_{n+1}=x_n-a_n\beta x_n^\gamma$ then stays positive and decreases to zero. Taylor expansion of $x_{n+1}^{-(\gamma-1)}-x_n^{-(\gamma-1)}$, followed by summation and division by $\sum_{k<n}a_k$, {yields}
\begin{equation*}
 \frac{x_n}{h_n}\longrightarrow [\beta(\gamma-1)]^{-1/(\gamma-1)}.
\end{equation*}
The expansion error is $O(a_n^2x_n^{\gamma-1})=o(a_n)$, which justifies this summation for both $s<1$ and $s=1$.

The noise tails $R_n=\sum_{k=n}^{\infty}a_k\varepsilon_k$ exist almost surely and in $L^2$. The martingale maximal inequality~\citep[Theorem~4.4.4]{durrett2019probability} implies, for dyadic blocks with $m\geq1$,
\begin{equation*}
\mathbb{E}\max_{m\leq n\leq2m}|R_n|^2
 \leq \c \sigma^2\sum_{k=m}^{\infty}a_k^2
 \leq \c \sigma^2m^{1-2s}.
\end{equation*}
We obtain the first inequality by writing $R_n=R_m-\sum_{k=m}^{n-1}a_k\varepsilon_k$ and bounding the two terms separately. Hence, for any fixed $\epsilon>0$ and $N\geq1$,
\begin{equation}\label{eq:pj-tail-event}
 \mathbb P\{\exists n\geq N:|R_n|>\epsilon h_n\}
 \leq \frac{\c \sigma^2}{\epsilon^2}
 \sum_{j=0}^{\infty}\frac{(2^jN)^{1-2s}}{h_{2^{j+1}N}^2}.
\end{equation}
For $s<1$, the series is finite precisely when $1-2s+2(1-s)/(\gamma-1)<0$, which is $s>s_{\star}$. For $s=1$, it is bounded by a constant times $\sum_j(2^jN)^{-1}(\log(2^{j+1}N+1))^{2r}$ and is also finite.

We choose $0<d<D$ with $d$ below and $D$ above the noiseless limit, and then choose $0<\epsilon<d$ so that
\begin{equation*}
 \beta(d+\epsilon)^\gamma<\tfrac12rd,
 \qquad \beta(D-\epsilon)^\gamma>2rD.
\end{equation*}
Such choices are possible by first taking $d$ sufficiently small and $D$ sufficiently large. We also take $\epsilon$ small enough that the noiseless limit lies in $(d+2\epsilon,D-2\epsilon)$. We choose $N$ sufficiently large for~\eqref{eq:pj-barrier-speed}, and for $x_N/h_N$ to lie strictly between $d+2\epsilon$ and $D-2\epsilon$. We choose the compact interval large enough to contain the noiseless initial segment and $[-(D+\epsilon)h_N,(D+\epsilon)h_N]$ in its interior. For sufficiently small $\sigma>0$, every finite sequence of Rademacher signs up to time $N-1$ perturbs this initial segment by as little as desired. We can therefore ensure that projection is inactive on that segment and that $X_N-R_N\in[d h_N,D h_N]$ whenever $|R_N|\leq\epsilon h_N$. {\hypersetup{linkcolor=black,citecolor=black}By reducing $\sigma$ again,~\eqref{eq:pj-tail-event} ensures that the event $E=\{|R_n|\leq\epsilon h_n\text{ for all }n\geq N\}$ has probability at least $1/2$.}

On $E$, we write $z_n=X_n-R_n$. If $z_n\in[d h_n,D h_n]$, then $X_n=z_n+R_n>0$. The unprojected candidate $\widetilde X_{n+1}=X_n-a_n(\beta X_n^\gamma+\varepsilon_n)$ satisfies
\begin{equation*}
\widetilde X_{n+1}-R_{n+1}=z_n-a_n\beta(z_n+R_n)^\gamma,
\end{equation*}
because $R_n=a_n\varepsilon_n+R_{n+1}$.
For large $N$, the map on the right is increasing on this interval, since $a_nh_n^{\gamma-1}\to0$. Its value at $d h_n$ is at least $d h_{n+1}$, and its value at $D h_n$ is at most $D h_{n+1}$, by the two strict inequalities above and~\eqref{eq:pj-barrier-speed}. Thus $\widetilde X_{n+1}-R_{n+1}\in[d h_{n+1},D h_{n+1}]$, and $\widetilde X_{n+1}$ lies in $[-(D+\epsilon)h_N,(D+\epsilon)h_N]\subset\mathcal I$. Projection is therefore inactive at this step, so $X_{n+1}=\widetilde X_{n+1}$. Starting from the established bounds at $N$, induction proves both the barriers and inactivity of projection on $E$. We conclude that $X_n\geq(d-\epsilon)h_n>0$ on $E$ for $n\geq N$. The finite initial segment contributes $O(n^{-1})$ to the average, and $n^{-1}\sum_{k=N}^n h_k\geq(1-o(1))h_n$ and $n^{-1}=o(h_n)$. This proves~\eqref{eq:pj-flat-lower} with $p_0=1/2$ after reducing $\c $.
\end{proof}

%% file: fig/rate_exponents.tex
\begingroup
\color{black}
\colorlet{rateSTB}{blue!30!white}
\colorlet{rateSA}{red!70!white}
\colorlet{ratePJone}{green!70!white}
\colorlet{ratePJsix}{magenta!70!white}
\colorlet{ratePJeight}{orange!70!white}
\begin{tikzpicture}[font=\normalfont\fontsize{10}{12}\selectfont]
\fill[rateSTB] (.22cm,9.32cm) circle (2.5pt);
\node[anchor=west] at (.44cm,9.32cm)
  {\normalfont{\textsf{STB at $\gamma=0$: super-polynomial}}};

\begin{axis}[
  at={(0cm,2.1cm)},anchor=south west,
  scale only axis,width=13cm,height=6.5cm,
  xmin=0,xmax=1,ymin=0,ymax=1,
  axis lines=left,axis line style={black!65},
  axis background/.style={fill=white},
  tick align=outside,tick style={black!65},
  xtick={0,.3333333333,.5,.6666666667,.8333333333,1},
  xticklabels={\normalfont{\textsf{$0$}},\normalfont{\textsf{$1/2$}},
    \normalfont{\textsf{$1$}},\normalfont{\textsf{$2$}},
    \normalfont{\textsf{$5$}},\normalfont{\textsf{$\gamma\to\infty$}}},
  ytick={0,.3333333333,.5,.6666666667,1},
  yticklabels={\normalfont{\textsf{$0$}},\normalfont{\textsf{$1/2$}},
    \normalfont{\textsf{$1$}},\normalfont{\textsf{$2$}},\normalfont{\textsf{$\infty$}}},
  xlabel={\normalfont{\textsf{Signal strength $\gamma$}}},
  ylabel={\normalfont{\textsf{Exponent $r$ in $n^{-r}$}}},
  tick label style={font=\normalfont\fontsize{10}{12}\selectfont},
  label style={font=\normalfont\fontsize{10}{12}\selectfont},
  ymajorgrids,grid style={black!9,line width=.3pt},
  clip=false,enlargelimits=false
]
\addplot[rateSTB,line width=2.7pt,domain=0:1,samples=251]
  {(1-x)/(1+x)};
\addplot[black,line width=1.5pt,densely dotted,domain=0:1,samples=251]
  {(1-x)/(1+x)};
\addplot[ratePJone,line width=2.5pt,domain=0:.5,samples=151]
  {(1-x)/(2-x)};
\addplot[ratePJone,line width=2.5pt,domain=.50001:1,samples=2] {0};
\addplot[rateSA,line width=1.2pt,dashed,domain=.3333333333:.5,samples=101]
  {(1-x)/(2-x)};
\addplot[rateSA,line width=1.2pt,dashed,domain=.50001:1,samples=2] {0};
\addplot[ratePJsix,line width=1.4pt,dash pattern=on 3pt off 2pt,
  domain=0:.5,samples=151] {.6*(1-x)/(1+.6*(1-x))};
\addplot[ratePJsix,line width=1.4pt,dash pattern=on 3pt off 2pt,
  domain=.500001:1,samples=401]
  {min(.6*(1-x),.4*(1-x)/(2*x-1))/(1+min(.6*(1-x),.4*(1-x)/(2*x-1)))};
\addplot[ratePJeight,line width=1.4pt,dash pattern=on 5pt off 2pt on 1pt off 2pt,
  domain=0:.5,samples=151] {.8*(1-x)/(1+.8*(1-x))};
\addplot[ratePJeight,line width=1.4pt,dash pattern=on 5pt off 2pt on 1pt off 2pt,
  domain=.500001:1,samples=401]
  {min(.8*(1-x),.2*(1-x)/(2*x-1))/(1+min(.8*(1-x),.2*(1-x)/(2*x-1)))};
\addplot[only marks,mark=o,mark size=3.8pt,rateSA,line width=1.2pt]
  coordinates {(0,.5)};
\addplot[only marks,mark=+,mark size=3pt,ratePJone,line width=1.4pt]
  coordinates {(0,.5)};
\addplot[only marks,mark=triangle*,mark size=3pt,ratePJsix]
  coordinates {(0,.375)};
\addplot[only marks,mark=square*,mark size=2.5pt,ratePJeight]
  coordinates {(0,.4444444444)};
\addplot[only marks,mark=o,mark size=2.8pt,rateSA,line width=1.1pt,
  mark options={fill=white}] coordinates {(.3333333333,.4) (.5,0)};
\addplot[only marks,mark=o,mark size=2.5pt,black,line width=.8pt,
  mark options={fill=white}] coordinates {(0,1) (1,0)};
\addplot[only marks,mark=diamond*,mark size=3.4pt,black,
  mark options={fill=white,line width=1pt}] coordinates {(.5,.3333333333)};
\draw[black!50,thin] (axis cs:.515,.352) -- (axis cs:.60,.49);
\node[anchor=west,fill=white,inner sep=2pt] at (axis cs:.60,.50)
  {\normalfont{\textsf{PJ: positive-slope}}};
\end{axis}

\begin{scope}[yshift=-.55cm]
\draw[rateSTB,line width=2.7pt] (0cm,.92cm)--(.65cm,.92cm);
\node[anchor=west] at (.82cm,.92cm) {\normalfont{\textsf{STB}}};
\draw[rateSA,dashed,line width=1.2pt] (4.45cm,.92cm)--(5.10cm,.92cm);
\draw[rateSA,fill=white,line width=1.1pt] (4.775cm,.92cm) circle (2.6pt);
\node[anchor=west] at (5.27cm,.92cm) {\normalfont{\textsf{SA, $a_n=1/n$}}};
\draw[black,densely dotted,line width=1.pt] (9cm,.92cm)--(9.65cm,.92cm);
\node[anchor=west] at (9.82cm,.92cm) {\normalfont{\textsf{Minimax-optimal}}};
\draw[ratePJone,line width=2.5pt] (0cm,.24cm)--(.65cm,.24cm);
\draw[ratePJone,line width=1.2pt] (.325cm,.15cm)--(.325cm,.33cm);
\node[anchor=west] at (.82cm,.24cm) {\normalfont{\textsf{PJ, $s=1$}}};
\draw[ratePJsix,line width=1.4pt,dash pattern=on 3pt off 2pt]
  (4.45cm,.24cm)--(5.10cm,.24cm);
\fill[ratePJsix] (4.775cm,.33cm)--(4.69cm,.18cm)--(4.86cm,.18cm)--cycle;
\node[anchor=west] at (5.27cm,.24cm) {\normalfont{\textsf{PJ, $s=0.6$}}};
\draw[ratePJeight,line width=1.4pt,dash pattern=on 5pt off 2pt on 1pt off 2pt]
  (9cm,.24cm)--(9.65cm,.24cm);
\fill[ratePJeight] (9.255cm,.17cm) rectangle (9.395cm,.31cm);
\node[anchor=west] at (9.82cm,.24cm) {\normalfont{\textsf{PJ, $s=0.8$}}};
\end{scope}
\end{tikzpicture}
\endgroup

%% file: appendix_new/weighted_inference_proof.tex
\section{Proofs for weighted-section and augmented \textsc{STB}}
\label{supp:weighted-inference}
We prove the all-observation results of Section~\ref{section:normality-efficiency}. All constants below may depend on the fixed regression function, {common error {distribution}, threshold scale and initial interval}, but are independent of $n$, $K$ and the stage. The algorithms, their output and their information conditions are given in Section~\ref{sec:weighted-stb}.

\subsection{{Stopped means and localization}}\label{supp:weighted-rsa}
\begingroup
\begin{proof}[Proof of Proposition~\ref{prop:weighted-pooled}]
Reindex the threshold with $K_*=2$ and $t^*=K+t-2$, so $K_*+t^*=K+t$. Lemma~\ref{lemma:stopping-localization} then implies $N_t/\mathfrak n_t\to1$ conditionally in probability, and Theorem~\ref{thm:RSA} proves the stopped-mean CLT, uniformly over $0<|\mu_t|\le\C_\star$. Since
\begin{equation*}
 \mu_t^2\mathfrak n_t=2v^2(K+t)\log(\mathfrak n_t+1)\overset{d}{\rightarrow}\infty,
\end{equation*}
we have $|\mu_t|\sqrt{N_t}\to\infty$. Dividing the tight quantity $\sqrt{N_t}(M_t-\mu_t)$ by $\mu_t\sqrt{N_t}$ proves relative consistency. These arguments apply to the observations that determine the stopping time itself.
\end{proof}

\begin{lemma}[Localization to bounded regression functions]\label{lem:weighted-localization}
Suppose Assumption~\ref{assumption:regularity}.\ref{condition:sign}, Assumption~\ref{ass:weighted-efficient}, and~\eqref{eq:weighted-K-conditions} hold. For sufficiently large fixed $M$, put
\begin{equation}
 f_M(x)=\normalfont{\textsf{clip}}_{[-M,M]}(f(x)).\label{eq:weighted-truncated-function}
\end{equation}
Relate the runs under $f$ and $f_M$ using the same errors and perturbations. Let $\mathcal O_n$ and $\mathcal O_n^{(M)}$ denote their ordered lists of the first $n$ location--response pairs. Then
\begin{equation}
 \lim_{M\to\infty}\limsup_{n\to\infty}
 \mathbb P\{\mathcal O_n\ne\mathcal O_n^{(M)}\}=0.
 \label{eq:weighted-localization}
\end{equation}
The truncated function has the same root, slope and local quadratic remainder as $f$. For each fixed $M$, there are positive finite constants, used only in the proof, such that
\begin{equation}
 0<\c_M\le\frac{f_M(x)}{x-\theta}\le\C_M\quad(x\ne\theta),\qquad
 |f_M(x)-\beta(x-\theta)|\le\C_M|x-\theta|^2\quad(x\in\mathcal I).
 \label{eq:weighted-bounded-comparison}
\end{equation}
\end{lemma}
\begin{proof}
Choose $\rho$ as in~\eqref{eq:weighted-global-separation} and $M>\sup_{|x-\theta|\le\rho}|f(x)|$. Clipping then changes nothing in this neighborhood. Outside it, the global lower bound in~\eqref{eq:weighted-global-separation} implies
$|f_M(x)|\ge\min\{\c_f,M/R\}|x-\theta|$, where $R=b_0-a_0$. The upper bound follows from $|f_M|\le M$ and $|x-\theta|\ge\rho$. The same argument extends the local quadratic remainder to this complement, proving~\eqref{eq:weighted-bounded-comparison}.

Put $p=2+\min\{\delta,1\}>2$. By~\eqref{eq:noise-moment-envelope}, a union bound over fresh error rows makes every test decision correct with probability at least $1-\C K^{1-p/2}$. On this event all intervals contain $\theta$. Choose a fixed integer $T$ such that $R(3/4)^{T-1}\le\rho$. Every position at stage $t\ge T$ then belongs to the neighborhood where $f_M=f$.

Write $W_t=X_{r,t}-X_{\ell,t}$. At any stage, conditional on the history before $Z_t$, the density of $X_t$ is at most $4/W_t$. Independently of whether previous decisions were correct, clipping implies $W_t\ge R4^{-(t-1)}$, so this density is at most $4^t/R$. Let $B_M=\{x\in\mathcal I:|f(x)|>M\}$ and let $|B_M|$ denote its Lebesgue measure. Before the first position in $B_M$, the responses, stopping times and updated means agree exactly. On the event of correct decisions, a first difference can therefore occur only at one of the stages $t<T$. Consequently
\begin{equation}
 \mathbb P\{\mathcal O_n\ne\mathcal O_n^{(M)}\}
 \le\C K^{1-p/2}+\frac1R\sum_{t<T}4^t|B_M|.
 \label{eq:weighted-localization-probability}
\end{equation}
Because $f$ is finite and measurable on a finite interval, $|B_M|\to0$. First let $n\to\infty$, then $M\to\infty$. This proves~\eqref{eq:weighted-localization}. The coupling uses infinite independent error rows and perturbations to continue each run beyond its budget. Each test ends almost surely because its mean is nonzero, the noise averages tend to zero, and $T(j,\alpha_t)/j\to0$.

Any statistic of the first $n$ pairs agrees on this coupling event. Thus a common limiting distribution established for every fixed $f_M$ transfers to $f$ by bounded continuous test functions. Bounds in probability transfer as well, by first choosing $M$ to make the coupling error small.
\end{proof}
\endgroup

\subsection{Sampling geometry and slope information}\label{supp:weighted-geometry}
\begingroup
Throughout this subsection and Section~\ref{supp:weighted-expansion}, first fix $M$ in Lemma~\ref{lem:weighted-localization} and, to simplify notation, write $f$ for $f_M$. Thus~\eqref{eq:weighted-bounded-comparison} holds and $|f|\le\C_\star=M$, while the root, slope and error distribution are unchanged. We assume the separation condition, Assumption~\ref{ass:weighted-efficient}, and~\eqref{eq:weighted-K-conditions}. For the proof, continue the algorithm beyond $n$ with the same $K(n)$. Each stage exposes a fresh independent error row only up to its sign-test stopping time. In observation order the exposed errors have the original i.i.d. distribution. Lemma~\ref{lem:weighted-localization} will transfer the resulting expansion, consistency and probability bounds to the original, possibly unbounded, function.

\begin{lemma}\label{lem:weighted-good}
There are events $E_n$ with $\mathbb P(E_n)\to1$ on which every interval contains $\theta$ and the following statements hold for all continuing stages. Let $W_t=X_{r,t}-X_{\ell,t}$, $R=b_0-a_0$, $e_t=X_t-\theta$ and $L_t=K+t$. Put $p=2+\min\{\delta,1\}>2$. Then

\begin{align}R4^{-(t-1)} & \le W_t\le R(3/4)^{t-1}, & W_t e^{-(K+t)} & \le |e_t|\le W_t,\label{eq:weighted-geometry} \\
\c\frac{G_t}{e_t^2} & \le N_t\le\C\frac{G_t}{e_t^2}, & G_t & =L_t\log\bigl(2+L_t/e_t^2\bigr),\label{eq:weighted-stage-cost} \\
|S_{j,t}/j-\mu_t| & \le\frac{\eta_n}{1-\eta_n}|\mu_t|\quad(j\ge N_t), & \eta_n & =K^{-(p-2)/(4p)}.\label{eq:weighted-all-records}\end{align}

The quantity $\eta_n$ and the lower bound for $|e_t|$ are only proof devices. The probability of any incorrect sign decision is at most $\C K^{1-p/2}$, with a constant depending on the common error {distribution} and $v$.
\end{lemma}
\begin{proof}[Proof of Lemma~\ref{lem:weighted-good}]
We let $W_{j,t}^{\varepsilon}=\sum_{k=1}^j\varepsilon_{k,t}$ and define the noise event by
$|W_{j,t}^{\varepsilon}|\le\eta_n T(j,\alpha_t)$ for every $j,t\ge1$.

Equation~\eqref{eq:noise-moment-envelope}, applied conditionally to each fresh row, and a union bound imply
\begin{equation*}
 \mathbb P(\text{noise event fails})
 \leq\C\eta_n^{-p}\sum_{t\geq1}(K+t)^{-p/2}
 \leq\C'K^{-(p-2)/4}\overset{d}{\rightarrow}0.
\end{equation*}
The choice of $\eta_n$ works for every $p>2$ and changes only this proof event.

For large $n$, $\eta_n<1/2$. On the noise event a crossing must have the correct sign. At its first occurrence,

\begin{equation}
|\mu_t|N_t>(1-\eta_n)T(N_t,\alpha_t).\label{eq:weighted-drift-cross}
\end{equation}

The ratio $T(j,\alpha_t)/j$ decreases, so for every $j\ge N_t$,

\begin{equation*}
|W_{j,t}^{\varepsilon}|/j
 \le\eta_n T(N_t,\alpha_t)/N_t
 <\frac{\eta_n}{1-\eta_n}|\mu_t|.
\end{equation*}

This proves~\eqref{eq:weighted-all-records}, including the endpoint mean at $j=N_t$.

We establish the uniform count comparison used in~\eqref{eq:weighted-stage-cost}. For $d\to\infty$, the solution to $r=d\log(r+1)$ lies between $d\log d$ and $2d\log d$ for all sufficiently large $d$. The lower inequality~\eqref{eq:weighted-drift-cross} implies $N_t/\log(N_t+1)>2v^2(1-\eta_n)^2L_t/\mu_t^2$. Conversely, the crossing has occurred by any integer $j$ for which $|\mu_t|j>(1+\eta_n)T(j,\alpha_t)$. Applying the same inverse comparison to these two inequalities and using~\eqref{eq:weighted-bounded-comparison} proves~\eqref{eq:weighted-stage-cost}. The minimum of $L_t/\mu_t^2$ diverges because $L_t\ge K$ and $|\mu_t|\le\C_\star$. Integer rounding is absorbed in the constants. This is a simultaneous quantitative version of the stopping-time comparison in Lemma~\ref{lemma:stopping-localization}.

Before drawing $Z_t$, the conditional density of $X_t$ is at most $4/W_t$. Hence

\begin{equation*}
\mathbb P\{|X_t-\theta|<W_t e^{-(K+t)}\mid\text{history before }Z_t\}
 \le8e^{-(K+t)}.
\end{equation*}

A union bound makes all the lower distance bounds hold with probability at least $1-\C e^{-K}$. Every completed interval update contracts its width by a factor in $[1/4,3/4]$, regardless of the sign. On the noise event the root remains inside, proving~\eqref{eq:weighted-geometry}. We take $E_n$ to be the intersection of these two events.

For the final sign-error statement we use the noise threshold itself. An incorrect crossing requires $|W_{j,t}^{\varepsilon}|>T(j,\alpha_t)$ at some $j$. By~\eqref{eq:noise-moment-envelope},

\begin{equation*}
\mathbb P(\text{any incorrect sign})
 \leq\C\sum_{t\geq1}(K+t)^{-p/2}
 \leq\C'K^{1-p/2}.
\end{equation*}

The conditional bounds allow every adaptively chosen position.
\end{proof}

\begin{lemma}\label{lem:weighted-design}
On events with probability tending to one, there are fixed positive constants bounding the following quantities:
\begin{align}\c\log n & \le J_n\le\C\log n,\label{eq:weighted-J} \\
\c(\log n)^3 & \le A_n:=\sum_{t\le J_n}r_t(n)e_t^2\le\C(\log n)^3,\label{eq:weighted-A} \\
B_n:=\sum_{t\le J_n}r_t(n)|e_t| & \le\C\sqrt n\log n,\qquad
 \sum_{t\le J_n}r_t(n)|e_t|^3\le\C K^2,\label{eq:weighted-B} \\
\c(\log n)^3 & \le D_n=A_n-n(\bar X_n-\theta)^2\le\C(\log n)^3.\label{eq:weighted-D}\end{align}

The current stage location tends to $\theta$ in probability. Every fixed number of stages finishes within budget with probability tending to one. The empirical fraction of observations outside any fixed neighborhood of $\theta$ tends to zero.
\end{lemma}
\begin{proof}[Proof of Lemma~\ref{lem:weighted-design}]
We work on $E_n$. Equations~\eqref{eq:weighted-geometry}--\eqref{eq:weighted-stage-cost} imply

\begin{equation}
G_t\le\C(K+t)^2,
 \qquad G_t\ge\c(K+t)t\quad(t\ge t_0),
 \qquad N_t\le\C(K+t)^2e^{2K+\C t},\label{eq:weighted-G-bounds}
\end{equation}

where $t_0$ is a fixed constant. For the first bound we use
$|e_t|\ge R4^{-(t-1)}e^{-(K+t)}$. For the second we use
$|e_t|\le R(3/4)^{t-1}$. Substitution of the first distance bound also establishes the bound for $N_t$.

Every completed stage has $r_t(n)=N_t$. The last stage has $1\le r_{J_n}(n)\le N_{J_n}$, including a test truncated by the budget. Summing the last inequality in~\eqref{eq:weighted-G-bounds} yields

\begin{equation*}
n\le\C(K+J_n)^2e^{2K+\C J_n}.
\end{equation*}

For $J_n\ge2$, the completed preceding stage and the lower count bound yield
$n\ge N_{J_n-1}\ge\c(4/3)^{2(J_n-2)}$. These two inequalities and $K=o(\log n)$ prove~\eqref{eq:weighted-J}. In detail, the latter inequality first implies $J_n=O(\log n)$. Substituting this in the former shows that $J_n$ is bounded below by a positive multiple of $\log n$. For any fixed $t$, the upper bound in~\eqref{eq:weighted-G-bounds} is $e^{o(\log n)}=o(n)$, which proves completion of any fixed initial set of stages.

We next use $r_t(n)e_t^2\le N_te_t^2\le\C G_t$ and sum~\eqref{eq:weighted-G-bounds}. The completed stages contribute the lower sum $\c\sum_{t=t_0}^{J_n-1}(K+t)t$. This proves~\eqref{eq:weighted-A}. Multiplying the upper summands by $|e_t|\le R(3/4)^{t-1}$ yields

\begin{equation*}
\sum_{t\le J_n}r_t(n)|e_t|^3
 \le\C\sum_{t\ge1}(K+t)^2(3/4)^{t-1}\le\C K^2.
\end{equation*}

To bound $B_n$, we group the completed stages according to the sign of $e_t$. Successive negative positions move the left endpoint toward the root. If $-a$ is its displacement before the next negative update, the next negative displacement has absolute value at most $3a/4$, since the new point is at least a quarter of the interval width to its right. The same argument applies to positive positions. Thus the absolute displacements on each side decrease by a factor at most $3/4$ at each update on that side.

We let $\tau$ be the last completed stage on one of the two sides. Along the stages on that side, $G_t$ is nondecreasing, since $K+t$ increases and $|e_t|$ decreases. Summing the geometric bounds for $1/|e_t|$ backward from $\tau$ yields

\begin{equation*}
\sum_{\substack{t<J_n\\ e_te_\tau>0}} N_t|e_t|
 \le\C\frac{G_\tau}{|e_\tau|}
 \le\C\sqrt{nG_\tau}\le\C\sqrt n\log n.
\end{equation*}

The middle inequality uses $n\ge N_\tau\ge\c G_\tau/e_\tau^2$. If there is no completed stage on a side its sum is zero. The partial last stage contributes at most

\begin{equation*}
r_{J_n}(n)|e_{J_n}|
 \le\sqrt n\sqrt{r_{J_n}(n)e_{J_n}^2}
 \le\C\sqrt{nG_{J_n}}\le\C\sqrt n\log n.
\end{equation*}

This proves the first bound in~\eqref{eq:weighted-B}. Since
$|\bar X_n-\theta|\le B_n/n$, its squared contribution to $D_n$ is only $O((\log n)^2)$. Equation~\eqref{eq:weighted-A} now proves~\eqref{eq:weighted-D}.

Finally, $W_{J_n}\le R(3/4)^{J_n-1}\to0$ on these events. For a fixed neighborhood of the root, all stages after a fixed index lie within it by~\eqref{eq:weighted-geometry}. Their complement uses $o(n)$ observations by the fixed-stage bound above. This proves the remaining statements.
\end{proof}

\begin{proposition}[The interpolation denominator]\label{prop:weighted-denominator}
With probability tending to one, both current endpoint means exist. If $\widehat f_{\ell,n}$ and $\widehat f_{r,n}$ denote these means and $W(n)$ the current interval width, then

\begin{equation}
\frac{\widehat f_{r,n}-\widehat f_{\ell,n}}{W(n)}
 =\beta+o_P(1).\label{eq:weighted-denominator}
\end{equation}

Thus the interpolation denominator estimates a positive slope times the interval width. The midpoint choice for missing or incorrectly signed means, followed by clipping and perturbation, defines the sampling rule before this property holds.
\end{proposition}
\begin{proof}[Proof of Proposition~\ref{prop:weighted-denominator}]
On $E_n$, a still-original endpoint would keep the interval width above either $\theta-a_0$ or $b_0-\theta$. Since its width tends to zero, both original endpoints are eventually replaced. The current means originate in completed tests. Equation~\eqref{eq:weighted-all-records} therefore applies to both means. The true endpoint values satisfy

\begin{equation*}
\frac{f(X_r)-f(X_\ell)}{X_r-X_\ell}=\beta+O(W(n)),
 \qquad |f(X_r)|+|f(X_\ell)|\le\C_f W(n).
\end{equation*}

We used~\eqref{eq:weighted-remainder} and the fact that $\theta$ lies between the endpoints. Their relative mean errors contribute $O(\eta_n)$ after division by $W(n)$. Both terms vanish. The simultaneous estimate justifies evaluating the errors at these adaptively selected means. 
\end{proof}

\subsection{Asymptotic expansion}\label{supp:weighted-expansion}

\begin{proof}[Proof of Theorem~\ref{thm:normality}]
We first fix $M$ in Lemma~\ref{lem:weighted-localization} and write $f$ for $f_M$. Thus $|f|\le M$, the root and slope are unchanged, and~\eqref{eq:weighted-bounded-comparison} holds. The conclusion for the original, possibly unbounded function follows from the coupling at the end of the proof. Throughout, constants may depend on this fixed $M$, the regression function, the error distribution, $v$ and the initial interval, but not on $n$, $K$ or the stage. The algorithm is continued beyond its budget with the same deterministic $K=K(n)$ for purposes of analysis.

Put $e_t=X_t-\theta$, $\bar e_n=\bar X_n-\theta$ and $R_i=R_f(X_{t(i)})$. Define the following shorthands
\[
A_n=\sum_{i=1}^n e_{t(i)}^2,\qquad
D_n=A_n-n\bar e_n^2,\qquad
Q_n=\sum_{i=1}^n e_{t(i)}\varepsilon_i.
\]
Substitution of $Y_i=\beta e_{t(i)}+R_i+\varepsilon_i$ into~\eqref{eq:weighted-LS} gives, when $D_n>0$, the exact slope identity
\begin{equation}
\widehat\beta_n^{\normalfont{\textsf{LS}}}-\beta
 =\frac{\displaystyle
   Q_n-\bar e_n\sum_i\varepsilon_i
   +\sum_i e_{t(i)}R_i-\bar e_n\sum_iR_i}{D_n}.
\label{eq:weighted-slope-identity}
\end{equation}
The root estimator has the exact three-term expansion
\begin{equation}
\sqrt n(\widehat\theta_n-\theta)
 =-\frac{\sum_i\varepsilon_i}{\widehat\beta_n^{\normalfont{\textsf{LS}}}\sqrt n}
   +\frac{\sqrt n\,\bar e_n(\widehat\beta_n^{\normalfont{\textsf{LS}}}-\beta)}{\widehat\beta_n^{\normalfont{\textsf{LS}}}}
   -\frac{\sum_iR_i}{\widehat\beta_n^{\normalfont{\textsf{LS}}}\sqrt n}.
\label{eq:weighted-root-identity}
\end{equation}
We show that the slope denominator tends to $\beta>0$, the first term has the asserted normal limit, and the last two terms vanish. Theorem~\ref{thm:RSA} will provide the lower information bound needed to control the slope error.

\smallskip
\noindent\textit{Applying the stopped-time theorem.}
Conditional on $\mathcal F_{t-1}$, the mean $\mu_t=f(X_t)$ is fixed, and the errors at the continuing stage are fresh i.i.d. variables with mean zero, variance $\sigma^2$ and the finite moment in~\eqref{eq:weighted-noise-moment}. The independent perturbation makes $\mu_t\ne0$ almost surely. To apply Theorem~\ref{thm:RSA} when $K$ depends on $n$, use its fixed value $K_*=2$ and index $t^*=K(n)+t-2$. Then $K_*+t^*=K(n)+t$, so the boundary is exactly the one used by the algorithm. The theorem is uniform over $0<|\mu_t|\le M$, which permits substitution of the adaptive mean without conditioning on future observations.

Let $\mathfrak n_t$ be the conditional deterministic center in~\eqref{eq:RSA-centering}, and put $L_t=K+t$. The first limit in Theorem~\ref{thm:RSA} implies uniform conditional tightness of
\[
\frac{N_t-\mathfrak n_t}{2\sqrt{\mathfrak n_t/\mu_t^2}}.
\]
Moreover,
\[
\mu_t^2\mathfrak n_t=2v^2L_t\log(\mathfrak n_t+1),
\qquad \inf_{0<|\mu_t|\le M}\mathfrak n_t\longrightarrow\infty
\quad(L_t\to\infty).
\]
For the latter statement, use monotonicity of $x/\log(x+1)$ and $2v^2L_t/\mu_t^2\ge2v^2L_t/M^2\to\infty$. On $N_t<\mathfrak n_t/2$, the displayed normalized variable has absolute value exceeding $|\mu_t|\sqrt{\mathfrak n_t}/4$, whose infimum tends to infinity. Thus there is a deterministic function $a(u)\to0$ such that, uniformly for $L_t\ge u$,
\[
\mathbb P\{N_t<\mathfrak n_t/2\mid\mathcal F_{t-1}\}\le a(u).
\]
The same argument proves $N_t/\mathfrak n_t\to1$ uniformly in conditional probability. The second limit of Theorem~\ref{thm:RSA} then implies $M_t/\mu_t\to1$, since $|\mu_t|\sqrt{N_t}\to\infty$. This describes the means used to update endpoints at their actual stopping counts.

\smallskip
\noindent\textit{Design information from actual continuing stages.}
We use only the geometry and upper stage-cost bounds of Lemma~\ref{lem:weighted-good}. Its events $E_n$ have probability tending to one, and on them, with $q=3/4$ and $R=b_0-a_0$,
\[
R4^{-(t-1)}\le W_t\le Rq^{t-1},\qquad
W_te^{-(K+t)}\le|e_t|\le W_t,\qquad
N_t e_t^2\le\C G_t,
\]
where $W_t=X_{r,t}-X_{\ell,t}$ and $G_t=L_t\log(2+L_t/e_t^2)$. These inequalities imply, for a fixed $c_1>0$,
\[
G_t\le\C(K+t)^2,\qquad
N_t\le\C(K+t)^2\exp(2K+c_1t).
\]
For completeness, the upper count follows because the test has stopped by any $j$ with $|\mu_t|j>(1+\eta_n)T(j,\alpha_t)$ on $E_n$. Inverting $j/\log(j+1)$, and using the bounded comparison~\eqref{eq:weighted-bounded-comparison}, gives the upper cost. \par The geometric lower bound on $|e_t|$ then gives the two displayed estimates. No sharp lower stage-cost bound is used below.

Choose $b>0$ so that $2b\log(1/q)>1$, and let $m_n=\lceil b\log n\rceil$. On $E_n$, $|\mu_{m_n}|\le\C Rq^{m_n-1}$, so
\[
\mathfrak n_{m_n}>
\frac{2v^2(K+m_n)}{\mu_{m_n}^2}
\ge\c(K+m_n)q^{-2(m_n-1)}\gg n.
\]
The stopped-time bound above implies
$\mathbb P(N_{m_n}<\mathfrak n_{m_n}/2)\to0$, by iterated conditional expectation. On its complement and $E_n$, stage $m_n$ cannot finish within budget $n$. Therefore $J_n\le m_n$ with probability tending to one. Since $r_t(n)\le N_t$, it follows that
\[
A_n\le\C\sum_{t\le m_n}(K+t)^2\le\C'(\log n)^3
\]
with probability tending to one.

To obtain the lower bound, we choose a fixed $a>0$ with $c_1a<1/2$, put $q_n=\lfloor a\log n\rfloor$, and take the deterministic block
$\mathcal B_n=\{\lceil q_n/2\rceil,\ldots,q_n\}$. On $E_n$,
\[
\sum_{t\le q_n}N_t
\le\C(K+q_n)^2\exp(2K+c_1q_n)
\le n^{1/2+o(1)}=o(n),
\]
so every stage in this block is completed. By the uniform stopped-time bound there are deterministic $a_n\to0$ such that
\[
\mathbb P(N_t<\mathfrak n_t/2\mid\mathcal F_{t-1})\le a_n
\qquad(t\in\mathcal B_n).
\]
If $H_n=\sum_{t\in\mathcal B_n}\mathbf1_{\{N_t<\mathfrak n_t/2\}}$, then
\[
\mathbb E H_n\le a_n|\mathcal B_n|,
\qquad
\mathbb P\{H_n>|\mathcal B_n|/2\}\le2a_n\longrightarrow0.
\]

On $E_n$, the bounded comparison and the definition of $\mathfrak n_t$ give
\[
\mathfrak n_t e_t^2
=2v^2L_t\frac{e_t^2}{\mu_t^2}\log(\mathfrak n_t+1)
\ge\c L_t\log(\mathfrak n_t+1).
\]
Since $\mathfrak n_t>2v^2L_t/\mu_t^2$ and $|\mu_t|\le\C Rq^{t-1}$, we have $\log(\mathfrak n_t+1)\ge\c't$ for large $t$. Thus $\mathfrak n_t e_t^2\ge\c''(\log n)^2$ throughout $\mathcal B_n$. With probability tending to one at least half of this block has $N_t\ge\mathfrak n_t/2$. These stages have finished, and hence
\[
A_n\ge\sum_{\substack{t\in\mathcal B_n\\N_t\ge\mathfrak n_t/2}}N_t e_t^2
\ge\c(\log n)^3.
\]
This is the lower information estimate obtained from Theorem~\ref{thm:RSA}.

To center the design, consider completed stages on each side of the root separately. Their absolute displacements decrease by a factor at most $q=3/4$. Indeed, from endpoints $\theta-a$ and $\theta+b$, a new negative point is at least $\theta-a+(a+b)/4$, so its distance is at most $3a/4$. The positive side is analogous. For $t\le J_n\le m_n$, the upper count bound and $r_t(n)\le n$ imply
\[
r_t(n)|e_t|\le\min\{n|e_t|,\C(\log n)^2/|e_t|\}.
\]
Split either side's geometric sequence at $(\log n)/\sqrt n$. Above this value, sum the reciprocals backward. Each contribution is at most $\C\sqrt n\log n$. The last, possibly unfinished, stage contributes at most
$\sqrt{n r_{J_n}(n)e_{J_n}^2}\le\C\sqrt n\log n$.
Consequently $|\bar e_n|=O_P((\log n)/\sqrt n)$ and $n\bar e_n^2=O_P((\log n)^2)$. Subtraction from the two bounds for $A_n$ proves
\[
\c(\log n)^3\le D_n\le\C(\log n)^3
\]
with probability tending to one. Also,
\[
\sum_{t\le J_n}r_t(n)|e_t|^3
\le\C\sum_{t\ge1}(K+t)^2q^{t-1}\le\C'K^2.
\]

\smallskip
\noindent\textit{Noise and nonlinear remainders.}
The global quadratic remainder bound for $f_M$ now gives
\begin{align}
\sum_{i=1}^n|R_i|&=O_P((\log n)^3),&
\sum_{i=1}^n|e_{t(i)}R_i|&=O_P(K^2),&
|\bar e_n|&=O_P((\log n)/\sqrt n).
\label{eq:weighted-remainders}
\end{align}
For the slope noise, $e_{t(i)}$ is measurable before observation $i$ and $\varepsilon_i$ is independent of that history. With $A_i=\sum_{j\le i}e_{t(j)}^2$, define for deterministic $d>0$
\[
Q_n^{(d)}=\sum_{i=1}^n e_{t(i)}\varepsilon_i\mathbf1_{\{A_i\le d\}}.
\]
The indicator is predictable, and the accepted coefficient-square sum is at most $d$. Orthogonality of centered increments gives $\mathbb E[(Q_n^{(d)})^2]\le\sigma^2d$. Since $Q_n^{(d)}=Q_n$ on $\{A_n\le d\}$, we have
\begin{equation}\label{eq:weighted-stopped-L2}
\mathbb P(|Q_n|>z,A_n\le d)\le\frac{\sigma^2d}{z^2},\qquad d,z>0.
\end{equation}
Together with the upper bound on $A_n$, this proves the first estimate below. Chebyshev's inequality gives the second:
\begin{equation}
Q_n=O_P((\log n)^{3/2}),\qquad
\sum_{i=1}^n\varepsilon_i=O_P(\sqrt n).
\label{eq:weighted-martingale}
\end{equation}

The term have respective orders $(\log n)^{3/2}$, $\log n$, $K^2$ and $(\log n)^4/\sqrt n$. Division by $D_n\asymp_P(\log n)^3$ proves, and in particular $\widehat\beta_n^{\normalfont{\textsf{LS}}}\to\beta>0$. Thus it has probability tending to one. More precisely,
\begin{align}
\sqrt n\,\bar e_n(\widehat\beta_n^{\normalfont{\textsf{LS}}}-\beta)
&=O_P\bigl((\log n)^{-1/2}+K^2/(\log n)^2+(\log n)^2/\sqrt n\bigr)=o_P(1),
\label{eq:weighted-slope-product}\\
\frac1{\sqrt n}\sum_iR_i
&=O_P((\log n)^3/\sqrt n)=o_P(1).
\label{eq:weighted-bias}
\end{align}
The condition $K=o(\log n)$ is used in the first bound. Slope consistency alone would not suffice for that product.

The stopped averages $M_t-\mu_t$ need not be centered, independent or identically distributed. However,~\eqref{eq:weighted-pooled-identity} is the exact identity
\[
\sum_{t\le J_n}r_t(n)\{\bar Y_{t,n}-\mu_t\}
=\sum_{t\le J_n}\sum_{j=1}^{r_t(n)}\varepsilon_{j,t}
=\sum_{i=1}^n\varepsilon_i.
\]
It includes the unfinished last test, and its normalization is the actual total $n$. In observation order, the next exposed error is independent of the previously revealed observations and has the original common distribution. Tower property therefore implies that the first $n$ exposed errors are i.i.d., even though the locations and stage counts are adaptive. The ordinary finite-variance CLT applies to this exact raw sum. \\
Combining it with~\eqref{eq:weighted-root-identity}--\eqref{eq:weighted-bias} and replacing the consistent slope denominator by $\beta$ proves~\eqref{eq:weighted-full-linear} for fixed $M$. Theorem~\ref{thm:RSA} has supplied the design information needed to make the slope term negligible, while the raw-error CLT identifies the terminal limiting distribution.

\smallskip
\noindent\textit{Variance estimation and removal of truncation.}
Let
$u_i=R_i-\bar R_n-(\widehat\beta_n^{\normalfont{\textsf{LS}}}-\beta)(e_{t(i)}-\bar e_n)$.
The fitted residual is $\varepsilon_i-\bar\varepsilon_n+u_i$. Since positions are bounded and $|R_i|\le\C e_{t(i)}^2$, we have $R_i^2\le\C'e_{t(i)}^2$. Hence
\[
\frac1n\sum_iR_i^2\le\C'A_n/n=o_P(1),\qquad
\frac1n\sum_iu_i^2
\le\frac2n\sum_iR_i^2+
2(\widehat\beta_n^{\normalfont{\textsf{LS}}}-\beta)^2D_n/n=o_P(1).
\]
Finite variance gives $n^{-1}\sum_i\varepsilon_i^2\to\sigma^2$ and $\bar\varepsilon_n\to0$ in probability. Cauchy--Schwarz makes the cross term between $u_i$ and $\varepsilon_i-\bar\varepsilon_n$ negligible. Therefore $\widehat\sigma_n^2\to\sigma^2$ and $\widehat V_n\to\sigma^2/\beta^2$. Slutsky's theorem proves~\eqref{eq:weighted-studentized}. 

Finally, Lemma~\ref{lem:weighted-localization} relates the original function and $f_M$ so that the first $n$ location--response pairs, their errors and all estimator outputs agree with probability whose iterated lower limit is one as $n\to\infty$ and then $M\to\infty$. The root and slope are unchanged. For every fixed tolerance, bound the probability of a failed remainder statement by the coupling error plus its probability under $f_M$, then take these two limits. For $O_P$, first choose $M$ to make the coupling error small and then a tightness constant for this fixed $M$. This transfers the expansion, the slope rate, the reciprocal information bound and consistency of the slope and both variance estimates to the original function. 
\end{proof}

\subsection{Regularity under local alternatives}\label{supp:weighted-regularity}
\begin{lemma}
\label{lem:weighted-local}
Suppose Assumptions~\ref{ass:weighted-efficient} and~\ref{assumption:efficiency} and~\eqref{eq:weighted-K-conditions} hold. Let $g_1,\ldots,g_p$ be continuous on $\mathcal I$, with Lipschitz continuous first derivatives on a neighborhood of $\theta$. For bounded $h\in\mathbb R^p$, set

\begin{equation*}
f_{n,h}=f+n^{-1/2}\sum_{j=1}^p h_jg_j.
\end{equation*}

For large $n$, these functions have unique crossing roots $\theta_{n,h}$ and slopes $\beta_{n,h}\to\beta$, uniformly on bounded $h$ sets. {The sampling rule in Algorithm~\ref{alg:weighted-section}, with unchanged inputs, followed by Algorithm~\ref{alg:weighted-estimator}, satisfies}

\begin{equation}
\sqrt n(\widehat\theta_n-\theta_{n,h})
 =-\frac1{\beta_{n,h}\sqrt n}\sum_{i=1}^n\varepsilon_i+o_P(1)
 \ \overset{d}{\rightarrow}\ \mathcal N(0,\sigma^2/\beta^2).\label{eq:weighted-local-expansion}
\end{equation}

{The statement also holds for row-wise i.i.d. centered errors, independent of the auxiliary randomization, whose laws vary with $n$, have uniformly bounded moments of some fixed order $2+\delta>2$, and have variances tending to $\sigma^2>0$.}
\end{lemma}
\begin{proof}[Proof of Lemma~\ref{lem:weighted-local}]
On a smaller interval about $\theta$, all perturbed derivatives are at least $\beta/4$ for large $n$. On its complement the unperturbed continuous function has absolute value bounded away from zero. Uniformly small perturbations preserve those signs. The intermediate value theorem and the derivative bound therefore {ensure} a unique crossing root with $\theta_{n,h}-\theta=O(n^{-1/2})$. The derivative at that root tends to $\beta$, and Taylor expansion {yields}

\begin{equation}
\theta_{n,h}-\theta
 =-\frac{\sum_jh_jg_j(\theta)}{\beta\sqrt n}+O(n^{-1}).\label{eq:weighted-root-derivative}
\end{equation}

{The Lipschitz constants of the perturbed derivatives are uniformly bounded on a common neighborhood. Assumption~\ref{assumption:efficiency} also makes $f$ and the permitted directions continuous on the compact interval, so all these perturbed functions are uniformly bounded. Their common local expansion and uniform separation away from their roots imply global comparisons of the form~\eqref{eq:weighted-bounded-comparison}, with constants independent of bounded $h$ and large $n$. Repeating Lemmas~\ref{lem:weighted-good}--\ref{lem:weighted-design} and identities~\eqref{eq:weighted-slope-identity}--\eqref{eq:weighted-root-identity} about $\theta_{n,h}$ proves the local expansion for the actual $n$ test observations.}

For these triangular error laws, let $p=2+\min\{\delta,1\}$. Their uniform $p$th moment bound makes~\eqref{eq:noise-maximal-moment}--\eqref{eq:noise-moment-envelope} and~\eqref{eq:weighted-stopped-L2} uniform in $n$. The characteristic functions satisfy
$\varphi_n(u/\sqrt m)=1-\sigma_n^2u^2/(2m)+O(m^{-p/2})$
for fixed $u$, uniformly as $m\to\infty$. This proves the full-sum triangular-array CLT and the fixed-index limit in Theorem~\ref{thm:RSA}. Its deterministic-window argument uses only a uniform variance bound{. Its} overshoot bound is uniform because $m\mathbb P(|\varepsilon_n|>c\sqrt m)\leq\C_c m^{1-p/2}\to0$. {Proposition~\ref{prop:weighted-pooled} then follows from the same stopped-mean argument.} Finally, uniform integrability of $\varepsilon_n^2$, truncation and {convergence in probability of averages of bounded independent variables} show $n^{-1}\sum_i\varepsilon_{n,i}^2-\sigma_n^2\to0$ in probability. Thus the expansion and variance consistency hold under these local alternatives. Together with the uniform function bounds, these facts prove the stated regularity.

\end{proof}

\subsection{Minimum variance with an unknown error {distribution}}\label{supp:weighted-minimum}
\begin{proof}[Proof of Proposition~\ref{prop:weighted-minimum}]
We let $P$ be the error {distribution}. For any bounded measurable $s$ with $\mathbb E_Ps(\varepsilon)=0$ and $I_s=\mathbb E_Ps(\varepsilon)^2>0$, define

\begin{equation*}
dP_t=(1+ts)\,dP,\qquad a_s=\mathbb E_P[\varepsilon s(\varepsilon)],\qquad m_t=ta_s.
\end{equation*}

For small $|t|$ this is a probability {distribution}. We let $Q_t$ be the {distribution} of $\varepsilon-m_t$ for $\varepsilon\sim P_t$, and put $f_t=f+m_t$. Thus $Q_t$ is centered and the conditional response at a chosen position has the {distribution} of $f(X)+\varepsilon$ with $\varepsilon\sim P_t$. 

\begingroup\color{black}\hypersetup{linkcolor=black,citecolor=black}
For sufficiently small $|t|$, boundedness of $s$ makes $1+ts$ positive. Moreover,
\begin{equation*}
 \mathbb E_{Q_t}|\varepsilon|^{2+\delta}
 \leq2^{1+\delta}\bigl\{(1+|t|\|s\|_\infty)\mathbb E_P|\varepsilon|^{2+\delta}
                +|ta_s|^{2+\delta}\bigr\},
\end{equation*}
so these centered laws have uniformly bounded moments of that order and their variances tend to $\sigma^2>0$. No moment-generating-function restriction is imposed in this unknown-{distribution} model. Assumption~\ref{assumption:efficiency} keeps $f_t=f+ta_s$ in the regression class and implies
\begin{equation*}
 \theta_t=\theta-\frac{a_s}{\beta}t+O(t^2).
\end{equation*}
\endgroup

For any predictable design, its sampling kernels and randomization are the same under both models. They cancel from the likelihood ratio. With $t=u/\sqrt n$, we obtain

\begin{equation}
\log\frac{d\mathbb P_{u/\sqrt n}}{d\mathbb P_0}
 =\sum_{i=1}^n\log\bigl(1+us(\varepsilon_i)/\sqrt n\bigr)
 =\frac u{\sqrt n}\sum_i s(\varepsilon_i)-\frac{u^2I_s}{2}+o_P(1).\label{eq:weighted-noise-LAN}
\end{equation}

The boundedness of $s$ controls the Taylor remainder by $O(n^{-1/2})$, and the i.i.d.\ {distribution} {establishes} the CLT for the score and {convergence of the sample second moment}. Thus this is a one-dimensional locally asymptotically normal experiment with information $I_s$ and root derivative $-a_s/\beta$.

When $a_s\ne0$, the nonzero root derivative permits local reparameterization by the root. For $a_s=0$, the resulting variance bound is zero. The convolution theorem of \cite[Section~2, theorem on pp.~324--325]{hajek1970characterization} now bounds the variance of a regular limit by $a_s^2/(\beta^2 I_s)$. Its likelihood condition is~\eqref{eq:weighted-noise-LAN}, and its regularity condition is the common centered limit under every fixed $u$. The theorem allows general experiments. Taking the supremum over bounded centered $s$ {yields}

\begin{equation*}
\sup_s\frac{a_s^2}{I_s}=\sigma^2.
\end{equation*}

The upper bound is Cauchy--Schwarz. For the reverse bound, we use

\begin{equation*}
s_M(\varepsilon)=\normalfont{\textsf{clip}}_{[-M,M]}(\varepsilon)-\mathbb E\normalfont{\textsf{clip}}_{[-M,M]}(\varepsilon).
\end{equation*}

These scores tend to $\varepsilon$ in $L^2(P)$ as $M\to\infty$.

Under $t=u/\sqrt n$, {the laws $Q_t$ are centered, have uniformly bounded $(2+\delta)$th moments, and have variances tending to $\sigma^2$.} Lemma~\ref{lem:weighted-local} applies to $f_t$ and these error laws, proving regularity and attainment. Finite collections of bounded scores and smooth directions are treated by their fixed linear combinations, with the same uniform bounds. This proves the theorem.
\end{proof}
\subsection{Efficiency}\label{supp:weighted-gaussian}
\begin{proof}[Proof of Proposition~\ref{prop:weighted-gaussian-efficiency}]
Lemma~\ref{lem:weighted-local} proves the required regularity in the full regression model. To obtain a lower bound valid for any predictable design, we use the vertical submodel

\begin{equation*}
f_{n,u}(x)=f(x)-\beta u/\sqrt n,\qquad
 \theta_{n,u}=\theta+u/\sqrt n+O(n^{-1}).
\end{equation*}

It stays in the model for large $n$, by Assumption~\ref{assumption:efficiency}. Under the baseline experiment, the conditional Gaussian likelihood {yields} exactly

\begin{equation}
\log\frac{d\mathbb P_{n,u}}{d\mathbb P_{n,0}}
 =u\mathcal Z_n-\frac{u^2I_f}{2},\qquad
 \mathcal Z_n=-\frac\beta{\sigma^2\sqrt n}\sum_{i=1}^n\varepsilon_i,
 \qquad I_f=\frac{\beta^2}{\sigma^2}.\label{eq:weighted-Gaussian-LAN}
\end{equation}

 Under every design, $\mathcal Z_n$ is exactly $\mathcal N(0,I_f)$. This verifies the likelihood condition of H\'ajek's convolution theorem~\citep[Section~2, pp.~324--325]{hajek1970characterization}, without a likelihood remainder. Regularity in the regression model implies regularity along this scalar submodel. Every regular limit is therefore a convolution of $\mathcal N(0,I_f^{-1})$ with another probability {distribution}, and its variance is at least $I_f^{-1}$. Theorem~\ref{thm:normality} attains this bound, with a centered Gaussian limit and no extra convolution component.
\end{proof}

%% file: section/Simulation_Study.tex
\section{Simulation study}\label{section:simulation}

We examine the finite-sample performance of STB and ASTB alongside SA, ASA and PJ, and connect the observed estimation errors, interval coverage and normalized empirical distribution to the corresponding theoretical results. The simulation code is available at \href{https://github.com/yueuy2/Sequential-Test-Bisection}{https://github.com/yueuy2/Sequential-Test-Bisection}.

\subsection{Experimental settings}\label{sim:settings}

All comparisons use independent Gaussian observation errors, $Y=f(X)+\varepsilon$, with $\varepsilon\sim\mathcal N(0,\sigma^2)$ and known sub-Gaussian scale $v=\sigma$. Unless specified otherwise below, the root is $\theta=0.37$, noise scale is $\sigma=0.05$, initial interval is $\mathcal I=[0,1], R = 1$, and the confidence error is $\Delta=0.05$. All methods receive the same total observation budget where $n\in \{10^2,10^3,\ldots,10^9\}$, with 200 repetitions per method.

\textit{Linear root estimation.}
The first experiment uses $f_1(x)=x-\theta$. We compare SA at three coefficients, full-history least-squares ASA, projected PJ averaging, STB and ASTB. All approximation methods start at the midpoint $1/2$. 
\begin{itemize}
    \item SA uses step-size $a_k= a/k$, with $a\in\{1/4,1,4\}$ for the small, optimally-tuned and large variants. 
    \item  PJ uses $4k^{-2/3}$, projects each update onto $[0,1]$, and takes averaging. 
    \item ASA uses $1/(k\widehat\beta_k)$ with its fitted slope clipped to $[0.1,10]$, initialized at 1, which is favorable to ASA.
    \item STB uses Algorithm~\ref{alg:stb} with $K=\lceil\log_2(4/\Delta)\rceil=7$.
    \item \textsc{ASTB}
\end{itemize} 

\textit{Flat cubic root estimation.}
The second experiment changes only the regression function and the independently tuned approximation coefficients. For $\gamma=3$, we use
\begin{equation}\label{sim:eq:function}
f_\gamma(x)={\textsf{sign}}(x-\theta)
\begin{cases}|x-\theta|^\gamma,&|x-\theta|\leq 1,\\
1 + \gamma(|x-\theta|-1),&|x-\theta|>1.
\end{cases}
\end{equation}

The flat cubic SA coefficients are $a\in\{8,32,128\}$, and PJ uses $4k^{-2/3}$.

\textit{Crossing-jump estimation.}
The third experiment uses $f_0(x)={\textsf{sign}}(x-\theta)$, with $f_0(\theta)=0$, and the same root, noise, interval, budgets and repetitions. The SA coefficients are $a=1/16,1/4,1$, and PJ uses $(1/4)k^{-1}$. The other algorithm parameters are unchanged. The jump gap is $b=1$. Theorem~\ref{thm:jump} gives a super-polynomial bound on the scale $\exp\{-\kappa\sqrt{n/\log n}\}$, where $\kappa=b\log 2/(16\sqrt2\,v)$. 

\textit{Confidence-sequence coverage and length.}
The three preceding STB experiments also provide confidence-sequence coverage and length. The additional quadratic experiment uses $f_2$ in~\eqref{sim:eq:function}, roots $\theta\in\{1/3,0.37\}$ and noise scales $\sigma=v\in\{0.02,0.2\}$, with all four combinations. It uses $[0,1]$, $\Delta=0.05$, $K=7$, 200 repetitions per budget.

\textit{Asymptotic distribution and variance.}
The linear ASTB experiment uses $f(x)=x-0.37$, $\sigma=v=0.05$, $\mathcal I=[0,1]$ and the ASTB parameters above. Its positive slope is $\beta=1$, and the theoretical asymptotic variance is $\sigma^2/\beta^2=0.0025$. We plot empirical distribution over 400 repetitions at $10^3$ through $10^6$ and 200 at $10^7$ through $10^9$. The histograms plot $\sqrt n(\widehat\theta_n-\theta)$ against $\mathcal N(0,\sigma^2/\beta^2)$. QQ plots use $(\beta/\sigma)\sqrt n(\widehat\theta_n-\theta)$ against standard normal quantiles. The figures display the results $n\in \{10^j\}_{4\leq j \leq 9}$.

We also use a different regression function $f(x)=x-\theta+2(x-\theta)^3$, $\theta=1/3$, $\sigma=v=0.5$, $\mathcal I=[0,1]$, $\Delta=0.05$. Its slope is $\beta=1$ and theoretical variance is $0.25$. It has 400 repetitions at $10^3$ through $10^6$ and 200 at $10^7$ through $10^9$.

\subsection{Estimation error}\label{sim:errors}
Tables~\ref{sim:tab:error1}--\ref{sim:tab:error0} compare the seven implementations using 200 repetitions at each budget. Figure~\ref{sim:fig:combined} displays the linear, flat cubic and crossing-jump curves, with mean plus or minus 1.96 Monte Carlo standard errors.

\begin{table}[H]
\centering\small
\setlength{\tabcolsep}{4pt}
\renewcommand{\arraystretch}{1.12}
\caption{Differentiable case: linear signal. Mean absolute errors at equal observation budgets.}
\label{sim:tab:error1}
\begin{tabular}{lcc}
\toprule
Method & $n=10^8$ & $n=10^9$ \\
\midrule
SA small    & $1.06\times10^{-3}$ & $5.99\times10^{-4}$ \\
SA optimal & $3.81\times10^{-6}$ & $1.27\times10^{-6}$ \\
SA large    & $5.96\times10^{-6}$ & $1.92\times10^{-6}$ \\
ASA         & $3.79\times10^{-6}$ & $1.30\times10^{-6}$ \\
PJ          & $3.81\times10^{-6}$ & $1.27\times10^{-6}$ \\
STB         & $1.42\times10^{-4}$ & $2.99\times10^{-5}$ \\
ASTB        & $4.23\times10^{-6}$ & $1.23\times10^{-6}$ \\
\bottomrule
\end{tabular}
\end{table}

At $10^9$ observations, ASTB has the smallest observed mean absolute error, $1.23\times10^{-6}$, close to SA optimal, ASA and PJ ($1.27\times10^{-6}$, $1.30\times10^{-6}$ and $1.27\times10^{-6}$). This is consistent with the positive-slope root-$n$ limits in Theorem~\ref{thm:normality} and Proposition~\ref{prop:pj-smooth}, and ASTB\textquotesingle s Gaussian efficiency in Proposition~\ref{prop:weighted-gaussian-efficiency}. STB\textquotesingle s larger finite-budget error, $2.99\times10^{-5}$, is compatible with its fixed-confidence bound in Theorem~\ref{thm:power}.

\begin{table}[H]
\centering\small
\setlength{\tabcolsep}{4pt}
\renewcommand{\arraystretch}{1.12}
\caption{Differentiable case: flat cubic signal. Mean absolute errors at equal observation budgets.}
\label{sim:tab:error3}
\begin{tabular}{lcc}
\toprule
Method & $n=10^8$ & $n=10^9$ \\
\midrule
SA small   & $5.06\times10^{-2}$ & $4.78\times10^{-2}$ \\
SA optimal & $2.73\times10^{-2}$ & $2.54\times10^{-2}$ \\
SA large   & $1.60\times10^{-2}$ & $1.45\times10^{-2}$ \\
ASA        & $5.59\times10^{-2}$ & $5.17\times10^{-2}$ \\
PJ         & $1.00\times10^{-2}$ & $7.01\times10^{-3}$ \\
STB        & $5.00\times10^{-3}$ & $5.00\times10^{-3}$ \\
ASTB       & $4.12\times10^{-2}$ & $2.75\times10^{-2}$ \\
\bottomrule
\end{tabular}
\end{table}
At $10^9$ observations, STB has the smallest observed mean absolute error, $5.00\times10^{-3}$, compared with $7.01\times10^{-3}$ for PJ and $2.75\times10^{-2}$ for ASTB. Theorem~\ref{thm:power} gives the eventual STB exponent $1/6$, while Theorem~\ref{thm:pj-upper} gives the same upper-bound exponent for PJ with $s=2/3$ in this setting. Thus the table shows a finite-sample advantage.

\begin{table}[H]
\centering\small
\setlength{\tabcolsep}{4pt}
\renewcommand{\arraystretch}{1.12}
\caption{Crossing jump. Mean absolute errors at equal observation budgets.}
\label{sim:tab:error0}
\begin{tabular}{lcc}
\toprule
Method & $n=10^8$ & $n=10^9$ \\
\midrule
SA small   & $2.74\times10^{-10}$ & $2.91\times10^{-11}$ \\
SA optimal & $1.06\times10^{-9}$ & $1.14\times10^{-10}$ \\
SA large   & $5.64\times10^{-9}$ & $5.38\times10^{-10}$ \\
ASA        & $1.42\times10^{-9}$ & $1.17\times10^{-10}$ \\
PJ         & $6.41\times10^{-10}$ & $6.41\times10^{-11}$ \\
STB        & $7.16\times10^{-22118}$ & $1.86\times10^{-64786}$ \\
ASTB       & $4.98\times10^{-4}$ & $3.01\times10^{-4}$ \\
\bottomrule
\end{tabular}
\end{table}

At $10^9$ observations, STB has the smallest observed mean absolute error, $1.86\times10^{-64786}$, compared with $2.91\times10^{-11}$ for SA small, $6.41\times10^{-11}$ for PJ and $3.01\times10^{-4}$ for ASTB. The pronounced STB contraction is consistent with the super-polynomial high-probability bound in Theorem~\ref{thm:jump}, whereas the PJ mean-error bound in Theorem~\ref{thm:pj-upper} is polynomial for this jump. 

\begin{figure}[H]\centering
\includegraphics[width=\linewidth,height=.69\textheight,keepaspectratio]{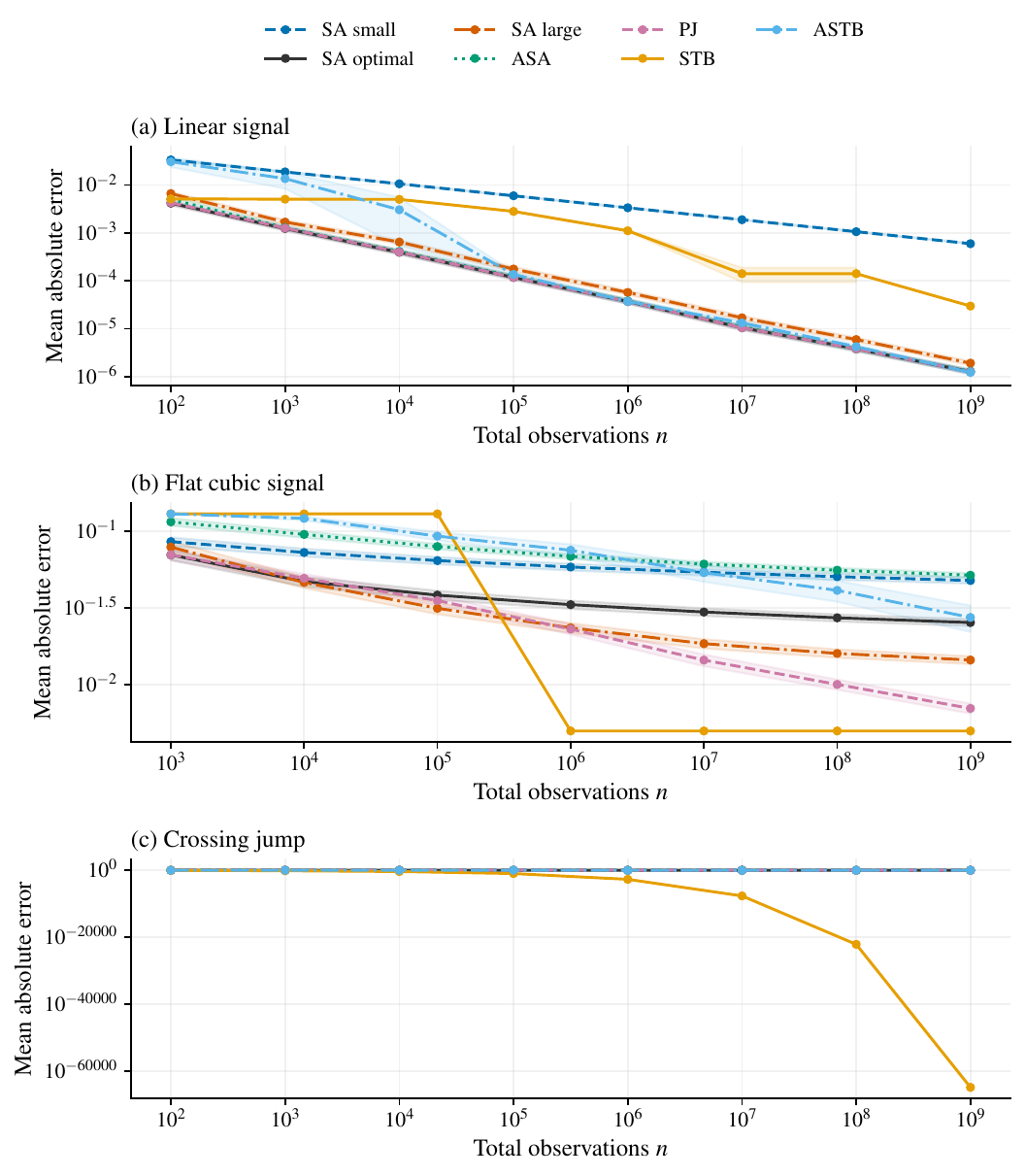}
\caption{{Mean absolute errors and pointwise Monte Carlo bands: (a) linear signal, (b) flat cubic signal, and (c) crossing jump.} All panels use $\theta=0.37$, $\sigma=v=0.05$, $\mathcal I=[0,1]$ and 200 repetitions per budget. Panel (b) displays $n\geq10^3$.}\label{sim:fig:combined}
\end{figure}

From $10^8$ to $10^9$, STB\textquotesingle s observed mean error decreases in panels (a) and (c), respectively. In panel (b), the STB mean is unchanged at $5.00\times10^{-3}$ over these two budgets. The first panel compares the positive-slope behavior in Theorem~\ref{thm:power}, Theorem~\ref{thm:normality} and Proposition~\ref{prop:pj-smooth} while the flat cubic panel corresponds to the {power-type results} in Theorems~\ref{thm:power} and~\ref{thm:pj-upper}, and the third panel illustrates the distinct jump scale in Theorem~\ref{thm:jump}. SA, ASA and PJ are competitive with ASTB in the linear experiment, whereas STB contracts much more strongly in the jump experiment. 

\subsection{The jump error on the super-polynomial scale}\label{sim:jump}

The jump comparison in Section~\ref{sim:errors} can also be viewed on the scale appearing in Theorem~\ref{thm:jump}.

\begin{figure}[H]\centering
\includegraphics[width=\linewidth]{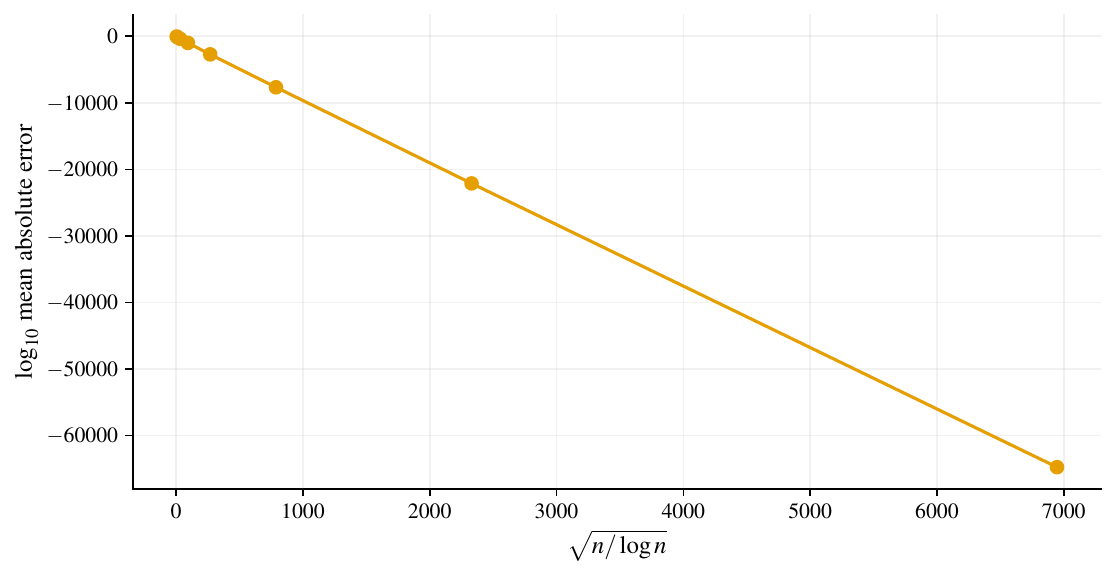}
\caption{Crossing jump. STB log mean absolute error against $\sqrt{n/\log n}$.}\label{sim:fig:jump-scale}
\end{figure}
The approximately linear decline on this $\sqrt{n/\log n}$ scale is consistent with Theorem~\ref{thm:jump}.

\subsection{Confidence sequences and asymptotic distribution}\label{sim:inference}

The interval and the all-observation fitted root quantify different aspects of inference. We first examine simultaneous interval coverage and contraction, then the positive-slope distribution and variance.

\subsubsection{Anytime-valid confidence sequence: coverage and interval length}
At every interval update, we check whether the root remains covered. Proposition~\ref{prop:type1-error} guarantees simultaneous coverage at least $1-\Delta=0.95$.  The following tables and plots report the observed coverage proportions and interval-length summaries from all 200 repetitions.

Note that Corollary~\ref{cor:interval-length} additionally requires $|X_t^\circ-\theta|\geq\lambda R2^{-(t-1)}$. Writing $u=(\theta-a_0)/R$, the relative midpoint gap along correct updates is $|\{2^{t-1}u\}-1/2|$. For $u=0.37$, the first two gaps are $0.13$ and $0.24$, and the subsequent periodic orbit has minimum gap $0.02$, so $\lambda=0.02$ is valid. For the quadratic setting with $u=1/3$, one may use $\lambda=1/6$. The jump root $0.37$ is \textit{not} a dyadic midpoint.

\begin{table}[H]
\centering\small
\setlength{\tabcolsep}{4pt}
\renewcommand{\arraystretch}{1.12}
\caption{Differentiable case: linear signal. STB coverage and median interval length.}
\label{sim:tab:cs1}
\begin{tabular}{lccc}
\toprule
$n$ & Pointwise & Simultaneous & Length \\
\midrule
$10^{8}$ & $9.95\times10^{-1}$ & $9.95\times10^{-1}$ & $1.95\times10^{-3}$ \\
$10^{9}$ & $9.95\times10^{-1}$ & $9.95\times10^{-1}$ & $2.44\times10^{-4}$ \\
\bottomrule
\end{tabular}
\end{table}

Table~\ref{sim:tab:cs1} gives empirical coverage $0.995$ at both displayed budgets, above the $0.95$ guarantee in Proposition~\ref{prop:type1-error}. Its median length decreases from $1.95\times10^{-3}$ to $2.44\times10^{-4}$. This contraction is compatible with the eventual $\widetilde O(n^{-1/2})$ length bound in Corollary~\ref{cor:interval-length}.

\begin{figure}[H]\centering
\includegraphics[width=\linewidth]{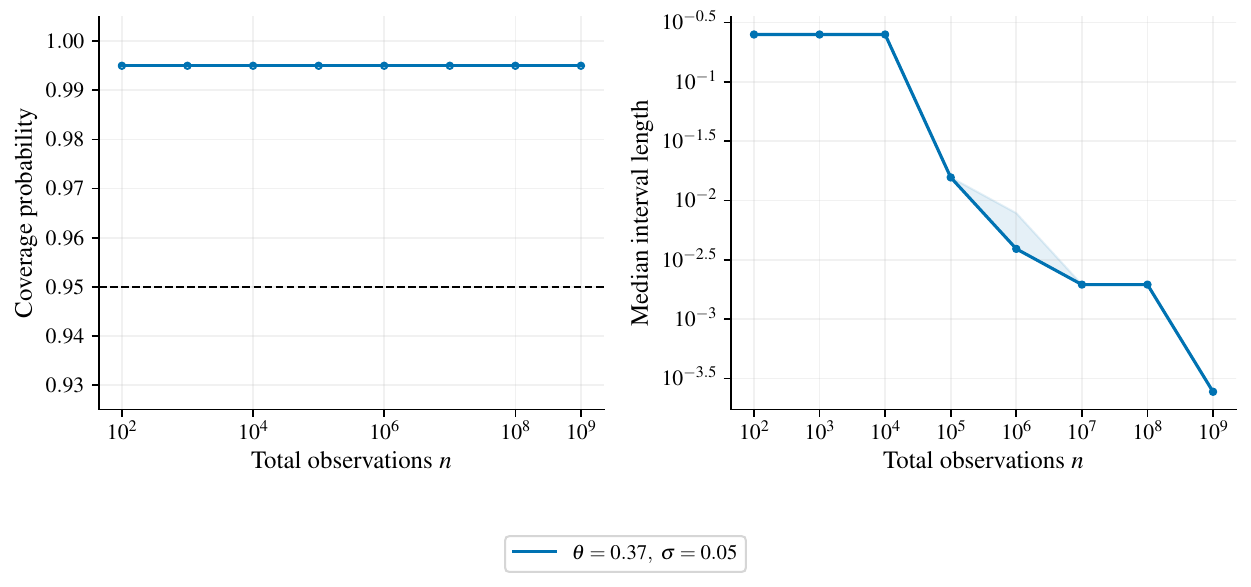}
\caption{Linear signal. Left: empirical pointwise and simultaneous coverage (open circles), with target $0.95$. Right: median interval length and 10th--90th percentiles.}\label{sim:fig:cs1}
\end{figure}

Figure~\ref{sim:fig:cs1} shows how coverage is maintained as the intervals contract over the full budget grid. The plateaus occur while a sign test is unfinished, followed by halving at completed updates. This stagewise behavior is consistent with Proposition~\ref{prop:type1-error} and Corollary~\ref{cor:interval-length}, which concern coverage and eventual contraction rather than smooth decrease at every budget.

\begin{table}[H]\centering\small
\setlength{\tabcolsep}{4pt}\renewcommand{\arraystretch}{1.12}
\caption{ STB coverage and median interval length.}\label{sim:tab:cs2}
\begin{tabular}{lccccc}\toprule
$\theta$ & $\sigma$ & $n$ & Pointwise & Simultaneous & Length \\ \midrule
$1/3$ & $2.00\times10^{-2}$ & $10^{8}$ & $1.00\times10^{0}$ & $1.00\times10^{0}$ & $3.13\times10^{-2}$ \\
$1/3$ & $2.00\times10^{-2}$ & $10^{9}$ & $1.00\times10^{0}$ & $1.00\times10^{0}$ & $1.56\times10^{-2}$ \\
{$0.37$} & $2.00\times10^{-2}$ & $10^{8}$ & $1.00\times10^{0}$ & $1.00\times10^{0}$ & $2.50\times10^{-1}$ \\
{$0.37$} & $2.00\times10^{-2}$ & $10^{9}$ & $1.00\times10^{0}$ & $1.00\times10^{0}$ & $1.56\times10^{-2}$ \\
$1/3$ & $2.00\times10^{-1}$ & $10^{8}$ & $1.00\times10^{0}$ & $1.00\times10^{0}$ & $6.25\times10^{-2}$ \\
$1/3$ & $2.00\times10^{-1}$ & $10^{9}$ & $1.00\times10^{0}$ & $1.00\times10^{0}$ & $6.25\times10^{-2}$ \\
{$0.37$} & $2.00\times10^{-1}$ & $10^{8}$ & $1.00\times10^{0}$ & $1.00\times10^{0}$ & $2.50\times10^{-1}$ \\
{$0.37$} & $2.00\times10^{-1}$ & $10^{9}$ & $1.00\times10^{0}$ & $1.00\times10^{0}$ & $2.50\times10^{-1}$ \\
\bottomrule\end{tabular}
\end{table}

All four settings in Table~\ref{sim:tab:cs2} have empirical simultaneous coverage one at $10^9$, consistent with Proposition~\ref{prop:type1-error}. Corollary~\ref{cor:interval-length} gives the eventual exponent $1/4$ for these quadratic crossings, but the table also shows substantial finite-budget dependence on noise and root geometry.

\begin{figure}[H]\centering
\includegraphics[width=\linewidth]{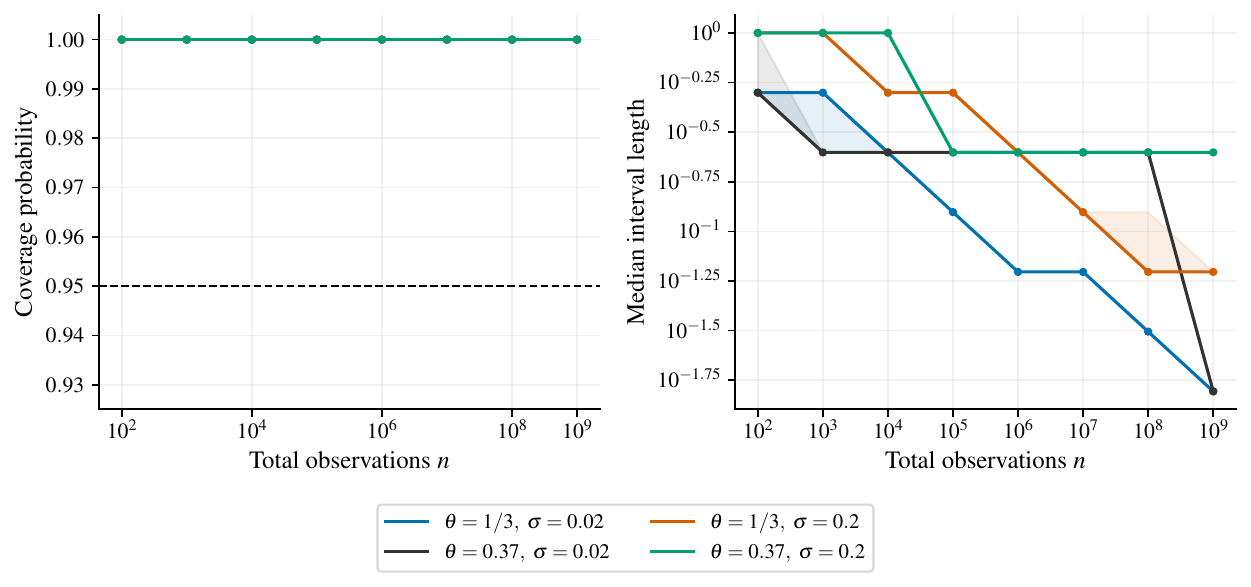}
\caption{Left: empirical pointwise and simultaneous coverage (open circles), with target $0.95$. Right: median interval length and 10th--90th percentiles.}\label{sim:fig:cs2}
\end{figure}

In Figure~\ref{sim:fig:cs2}, the median lengths at $10^9$ range from $1.56\times10^{-2}$ to $2.50\times10^{-1}$. Higher noise and a test point close to the root can delay an interval update, explaining the different plateaus without changing the exponent in Corollary~\ref{cor:interval-length}. Empirical coverage one in 200 runs is compatible with Proposition~\ref{prop:type1-error}.

\begin{table}[H]\centering\small
\setlength{\tabcolsep}{4pt}\renewcommand{\arraystretch}{1.12}
\caption{Differentiable case: flat cubic signal. STB coverage and median interval length.}\label{sim:tab:cs3}
\begin{tabular}{lccccc}\toprule
$\theta$ & $\sigma$ & $n$ & Pointwise & Simultaneous & Length \\ \midrule
{$0.37$} & $5.00\times10^{-2}$ & $10^{8}$ & $1.00\times10^{0}$ & $1.00\times10^{0}$ & $2.50\times10^{-1}$ \\
{$0.37$} & $5.00\times10^{-2}$ & $10^{9}$ & $1.00\times10^{0}$ & $1.00\times10^{0}$ & $2.50\times10^{-1}$ \\
\bottomrule\end{tabular}
\end{table}

Table~\ref{sim:tab:cs3} shows empirical pointwise and simultaneous coverage one at both displayed budgets, consistent with Proposition~\ref{prop:type1-error}. The median length remains $0.25$, illustrating that reliable coverage can coexist with slow finite-budget contraction near a flat root. Corollary~\ref{cor:interval-length} gives the eventual length exponent $1/6$, with a model-dependent starting budget.

\begin{figure}[H]\centering
\includegraphics[width=\linewidth]{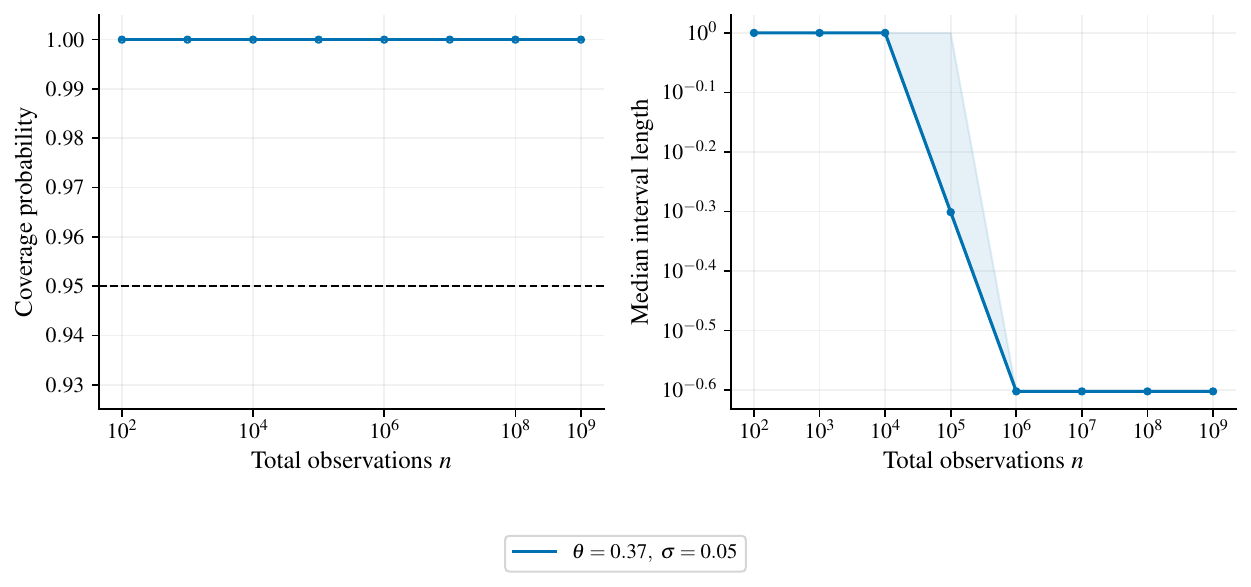}
\caption{Flat cubic signal. Left: empirical pointwise and simultaneous coverage (open circles), with target $0.95$. Right: median interval length and 10th--90th percentiles.}\label{sim:fig:cs3}
\end{figure}

Figure~\ref{sim:fig:cs3} makes the final-decade plateau visible. This behavior reflects the weak cubic signal near the root and is compatible with the eventual bound in Corollary~\ref{cor:interval-length}.

\begin{table}[H]\centering\small
\setlength{\tabcolsep}{4pt}\renewcommand{\arraystretch}{1.12}
\caption{Crossing jump. STB coverage and median interval length.}\label{sim:tab:cs0}
\begin{tabular}{lccccc}\toprule
$\theta$ & $\sigma$ & $n$ & Pointwise & Simultaneous & Length \\ \midrule
{$0.37$} & $5.00\times10^{-2}$ & $10^{8}$ & $1.00\times10^{0}$ & $1.00\times10^{0}$ & $1.06\times10^{-22117}$ \\
{$0.37$} & $5.00\times10^{-2}$ & $10^{9}$ & $1.00\times10^{0}$ & $1.00\times10^{0}$ & $2.70\times10^{-64786}$ \\
\bottomrule\end{tabular}
\end{table}

Table~\ref{sim:tab:cs0} shows empirical coverage one at $10^9$, with median completed-stage count 215,213. Proposition~\ref{prop:type1-error} provides the coverage guarantee, while the jump part of Corollary~\ref{cor:interval-length} allows super-polynomial contraction.

\begin{figure}[H]\centering
\includegraphics[width=\linewidth]{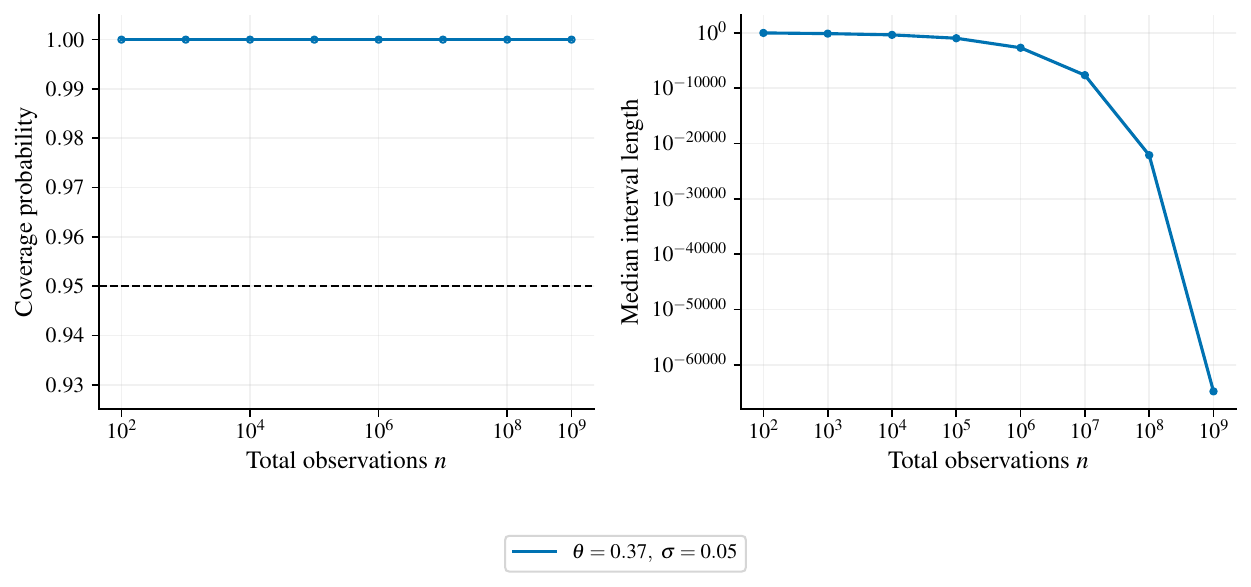}
\caption{Crossing jump. Left: empirical pointwise and simultaneous coverage (open circles), with target $0.95$. Right: median interval length and 10th--90th percentiles.}\label{sim:fig:cs0}
\end{figure}
In Figure~\ref{sim:fig:cs0}, the median length reaches $2.70\times10^{-64786}$ at $10^9$, decreasing by 42668.594 base-ten orders of magnitude over the final budget decade. This contrasts with the flat cubic plateau and illustrates the stronger contraction permitted by Corollary~\ref{cor:interval-length} for a separated jump.

\subsubsection{The linear all-observation fit}

{All moments and QQ quantiles use every simulated output, including fallback values.} Theorem~\ref{thm:normality} concerns the expansion $\sqrt n(\widehat\theta_n-\theta)=-\sum_i\varepsilon_i/(\beta\sqrt n)+o_P(1)$ under a positive root slope and the stated growing threshold. Theorem~\ref{thm:RSA} describes individual completed tests. {The total-budget argument includes every observation in an incomplete final test and, for the implementation used in these simulations, in an unfinished additional batch.}

\begin{table}[H]\centering\small
\setlength{\tabcolsep}{4pt}\renewcommand{\arraystretch}{1.12}
\caption{Linear ASTB: full-sample moments of $\sqrt n(\widehat\theta_n-\theta)$ for $f(x)=x-0.37$, $\sigma=0.05$ and $\mathcal I=[0,1]$.}\label{sim:tab:clt}
\begin{tabular}{lcccc}\toprule
$n$ & Mean & Variance & Ratio & $R$ \\ \midrule
$10^{3}$ & $4.89\times10^{-1}$ & $1.75\times10^{0}$ & $7.02\times10^{2}$ & $4.00\times10^{2}$ \\
$10^{4}$ & $3.26\times10^{-1}$ & $4.13\times10^{0}$ & $1.65\times10^{3}$ & $4.00\times10^{2}$ \\
$10^{5}$ & $1.03\times10^{-4}$ & $3.00\times10^{-3}$ & $1.20\times10^{0}$ & $4.00\times10^{2}$ \\
$10^{6}$ & $-1.81\times10^{-3}$ & $2.68\times10^{-3}$ & $1.07\times10^{0}$ & $4.00\times10^{2}$ \\
$10^{7}$ & $-7.19\times10^{-4}$ & $2.78\times10^{-3}$ & $1.11\times10^{0}$ & $2.00\times10^{2}$ \\
$10^{8}$ & $1.66\times10^{-3}$ & $2.95\times10^{-3}$ & $1.18\times10^{0}$ & $2.00\times10^{2}$ \\
$10^{9}$ & $-9.76\times10^{-4}$ & $2.31\times10^{-3}$ & $9.22\times10^{-1}$ & $2.00\times10^{2}$ \\
\bottomrule\end{tabular}
\end{table}

In Table~\ref{sim:tab:clt}, Ratio is empirical variance divided by $\sigma^2/\beta^2=0.0025$, and $R$ is the repetition count. At $10^9$, the normalized mean is $-9.76\times10^{-4}$ and variance is $2.31\times10^{-3}$, a ratio $0.922${. The} 95\% bootstrap variance interval is $[1.92\times10^{-3},2.72\times10^{-3}]$. These values are close to the centered normal limit in Theorem~\ref{thm:normality} and the Gaussian efficiency benchmark in Proposition~\ref{prop:weighted-gaussian-efficiency}. The large variances at $10^3$ and $10^4$ show finite-budget instability.

\subsubsection{Positive-slope cubic distribution}

The positive-slope cubic experiment uses $f(x)=(x-\theta)+2(x-\theta)^3$, $\theta=1/3$, $\sigma=0.5$ and thus $\beta=1$. At $10^9$, all 200 QQ pairs lie inside the fixed display window. Unlike the flat cubic comparison, this function satisfies the positive-slope premise of Theorem~\ref{thm:normality}, with Gaussian efficiency variance $0.25$ in Proposition~\ref{prop:weighted-gaussian-efficiency}.
\begin{figure}[!t]\centering
\includegraphics[width=\linewidth,height=.69\textheight,keepaspectratio]{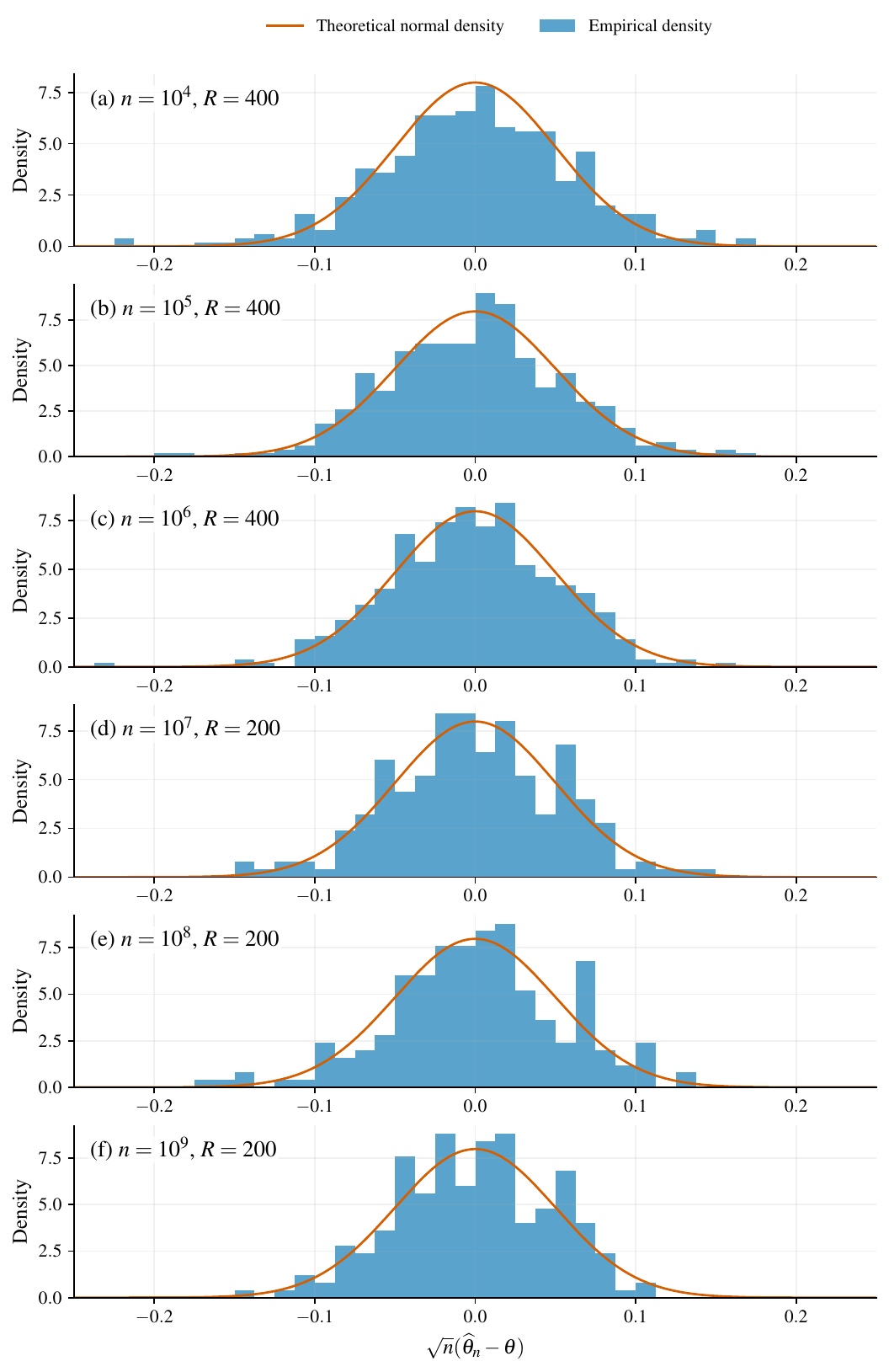}
\caption{Linear ASTB normalized-error histograms and the limiting normal density, at $n=10^4,\ldots,10^9$ from top to bottom.}\label{sim:fig:linearhist}
\end{figure}

The fixed density in Figure~\ref{sim:fig:linearhist} is $\mathcal N(0,0.0025)$ from Theorem~\ref{thm:normality} The larger sample size histograms concentrate around this shape, while the $10^4$ distribution has substantial tail effects. The display spans five theoretical standard deviations on either side of zero{. Observations} outside it still enter the full-sample moments in Table~\ref{sim:tab:clt}.
\clearpage

\clearpage
\begin{figure}[H]\centering
\includegraphics[width=\linewidth,height=.69\textheight,keepaspectratio]{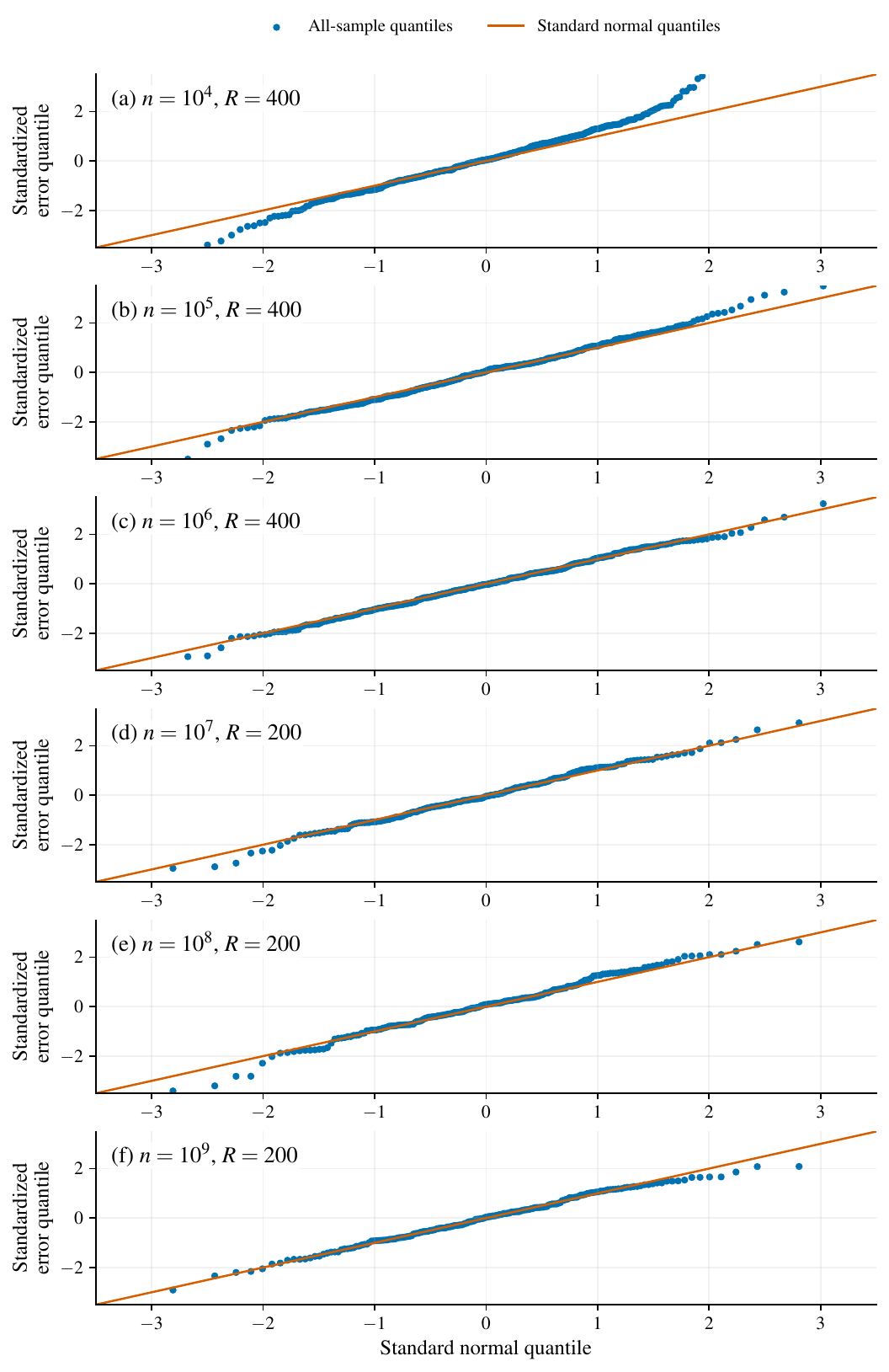}
\caption{Linear ASTB: all-sample standardized-error quantiles against standard normal quantiles, at $n\in\{10^j\}_{4\leq j \leq 9}$ from top to bottom.}\label{sim:fig:linearqq}
\end{figure}

Figure~\ref{sim:fig:linearqq} uses $(\beta/\sigma)\sqrt n(\widehat\theta_n-\theta)$, so agreement with the diagonal assesses the normal limit in Theorem~\ref{thm:normality} at the variance in Proposition~\ref{prop:weighted-gaussian-efficiency}. 
\clearpage

\clearpage
\begin{figure}[H]\centering
\includegraphics[width=\linewidth,height=.69\textheight,keepaspectratio]{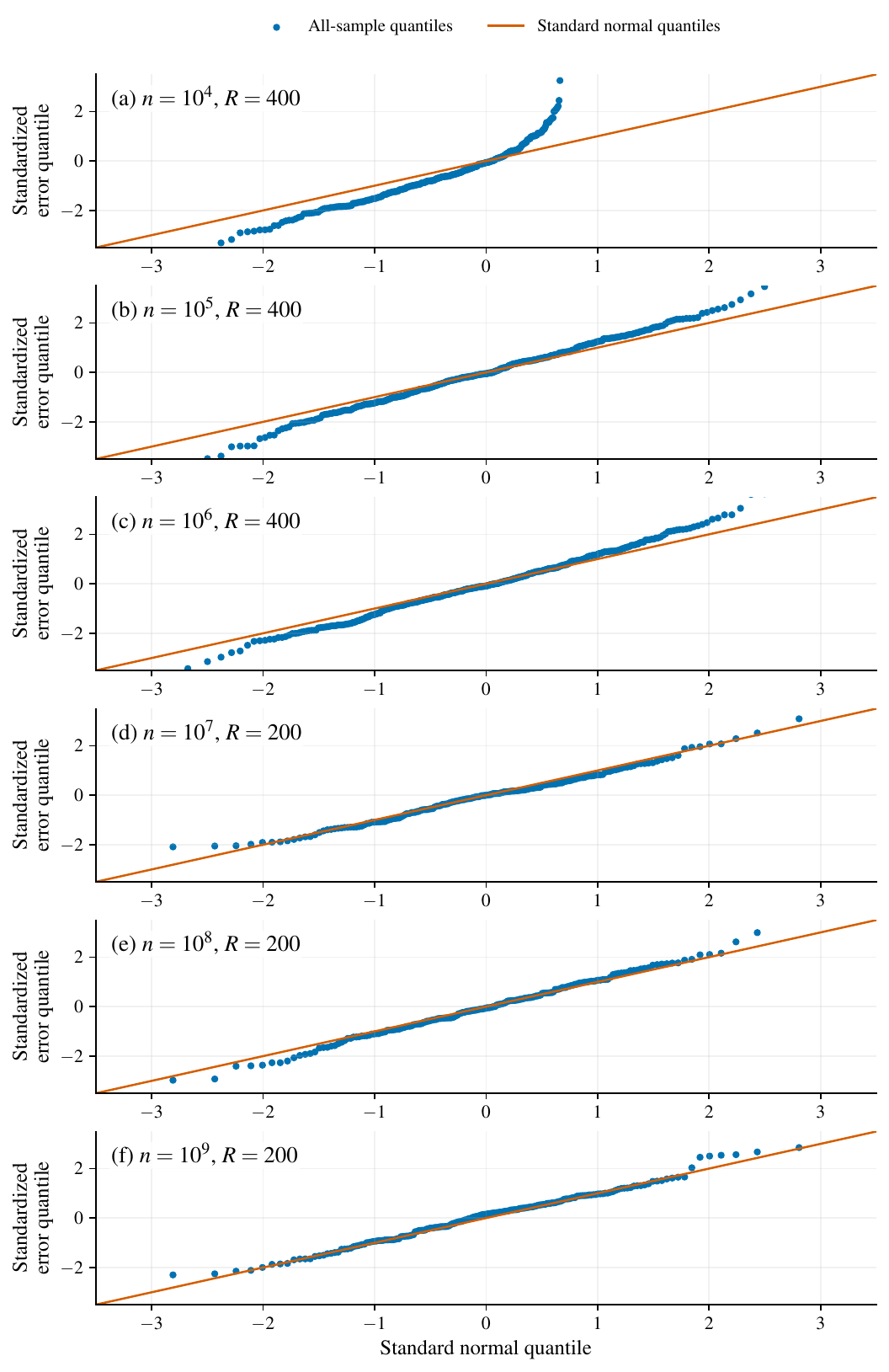}
\caption{Positive-slope cubic ASTB: all-sample standardized-error quantiles against standard normal quantiles, at $n\in\{10^j\}_{4\leq j \leq 9}$ from top to bottom.}\label{sim:fig:cubicqq}
\end{figure}

Figure~\ref{sim:fig:cubicqq} examines Theorem~\ref{thm:normality} under positive root slope. Standardization uses the theoretical variance $0.25$. The large $n$ plots with standard normal quantiles supports the same limiting distribution as in the linear case.
\clearpage

\FloatBarrier